\documentclass[10pt]{amsart}
\usepackage[margin=1.2in]{geometry}
\usepackage{xcolor}
\usepackage{graphicx}
\newtheorem{theorem}{Theorem}[section]
\newtheorem{lemma}[theorem]{Lemma}

\theoremstyle{definition}
\newtheorem{definition}[theorem]{Definition}

\theoremstyle{remark}
\newtheorem{remark}[theorem]{Remark}
\newtheorem{proposition}{Proposition}[section]

\begin{document}

\title[Convergence of a semi-Lagrangian scheme for the ES-BGK model]{Convergence of a semi-Lagrangian scheme for the ellipsoidal BGK model of the Boltzmann equation}

\author{Giovanni Russo}
\address{Universita di Catania, Catania, Italia}
\email{russo@dmi.unict.it}

%\curraddr{Department of Mathematics and Statistics,
%Case Western Reserve University, Cleveland, Ohio 43403}
%%\email{}

\author{Seok-Bae Yun}
\address{Department of mathematics, Sungkyunkwan University, Suwon 440-746, Republic of Korea }
\email{sbyun01@skku.edu}

%    General info
%\subjclass[2000]{Primary 54C40, 14E20; Secondary 46E25, 20C20}

%\date{January 1, 2001 and, in revised form, June 22, 2001.}

\thanks{S.-B. Yun was supported by Basic Science Research Program through the National Research Foundation of Korea(NRF) funded by the Ministry of Education(NRF-2016R1D1A1B03935955)}
\begin{abstract}
The ellipsoidal BGK model is a generalized version of the original BGK model designed to reproduce the physical Prandtl number in the Navier-Stokes limit.
In this paper, we propose a new implicit semi-Lagrangian scheme for the ellipsoidal BGK model, which, by exploiting special structures of the ellipsoidal Gaussian, can be transformed
into a semi-explicit form, guaranteeing the stability of the implicit methods and the efficiency of the explicit methods at the same time.
We then derive an error estimate of this scheme  in a weighted $L^{\infty}$ norm.
Our convergence estimate holds uniformly in the whole range of relaxation parameter $\nu$ including $\nu=0$, which corresponds to the original BGK model.
%We address the convergence of a implicit semi-Lagrangian scheme designed to compute the ellipsoidal BGK model (ES-BGK) of the Boltzmann equation,
%which is a generalized version of the original BGK model introduced to overcome the well-known shortcoming of the BGK model: the incorrect
%Prandtl number in the Navier-Stokes limit.
\end{abstract}

\maketitle
%\tableofcontents
%\[dx=Vdt+\epsilon dB\rightarrow u_t+-\nabla\cdot(Vu)+\epsilon \triangle u\]
\section{Introduction}
%\subsection{Ellipsoidal BGK model}
The BGK model \cite{BGK,Wal} has been widely used as an efficient model for the Boltzmann equation
because the BGK model is much more amenable to numerical treatment, and  still maintains many of the important qualitative properties of the Boltzmann equation.
But several short-commings are also reported where this model fails to reproduce the correct physical data of the Boltzmann equation. One such example is the Prandtl number,
which is defined as the ratio between the viscosity and the heat conductivity.
The Prandtl number computed via the BGK model does not match with the one derived from the Boltzmann equation, resulting in the incorrect hydrodynamic limit at the Navier-Stokes level.
To overcome this, Holway \cite{Holway} suggested a generalized version of the BGK model
 by replacing the local Maxwellian with an ellipsoidal Gaussian
parametrized by a free parameter $-1/2<\nu<1$. This model is called the ellipsoidal
BGK model (ES-BGK model), whose initial value problem  reads
\begin{align}\label{ESBGK}
\begin{split}
&\partial_t f+v\cdot \nabla f=\frac{1}{\kappa}A_{\nu}(\mathcal{M}_{\nu}(f)-f),\cr
&\qquad f(x,v,0)=f_0(x,v).
\end{split}
\end{align}
The velocity distribution function $f(x,v,t)$ is the number density of the particle system on the phase point $(x,v)\in\mathbb{T}^{d_1}\times\mathbb{R}^{d_2}$ ($d_1\leq d_2$)
at time $t\in\mathbb{R}_+$. Here, $\mathbb{T}$ denotes the unit interval with periodic boundary condition and $\mathbb{R}$ is the whole real line.
The Knudsen number $\kappa$ is a dimensionless number defined by the
ratio between the mean free path and the characteristic length.
For later convenience, we allowed a slight abuse of notation so that the convection term $v\cdot\nabla_xf$ is understood as
\[
v\cdot\nabla_xf=\sum_{1\leq i\leq d_1}v^i\partial_{x^i}f.\]
%$\kappa$ is the Kundsen number representing the ratio between the mean free path and the characteristic length of the system.
The collision frequency $A_{\nu}$ takes various
forms depending on modeling assumptions. In this paper, we only consider the
fixed collision frequency: $A_{\nu}=(1-\nu)^{-1}$. %We will return to the meaning of the free parameter $\nu$ later.
The ellipsoidal Gaussian $\mathcal{M}_{\nu}(f)$ reads:
\begin{eqnarray*}
\mathcal{M}_{\nu}(f)=\frac{\rho}{\sqrt{\det(2\pi\mathcal{T}_{\nu}})}\exp\left(-\frac{1}{2}(v-U)^{\top}\mathcal{T}^{-1}_{\nu}(v-U)\right),
\end{eqnarray*}
where the macroscopic density, velocity, temperature and the stress tensor are defined by
\begin{align*}
\rho(x,t)&=\int_{\mathbb{R}^{d_2}}f(x,v,t)dv,\cr
\rho(x,t)U(x,t)&=\int_{\mathbb{R}^{d_2}}f(x,v,t)vdv,\cr
d_2\rho(x,t) T(x,t)&=\int_{\mathbb{R}^{d_2}}f(x,v,t)|v-U(x,t)|^2dv,\cr
\rho(x,t)\Theta(x,t)&=\int_{\mathbb{R}^{d_2}}f(x,v,t)(v-U)\otimes(v-U)\,dv.
%&=\int_{\mathbb{R}^d}f(x,v,t)(v\otimes v)dv+\rho U\otimes U.
\end{align*}
The temperature tensor $\mathcal{T}_{\nu}$ is given by a convex combination of $T$ and $\Theta$:
\begin{align*}
\mathcal{T}_{\nu}
%&=\left(\begin{array}{ccc}(1-\nu) T+\nu\Theta_{11}&\nu\Theta_{12}&\nu\Theta_{13}\cr
%\nu\Theta_{21}&(1-\nu)T+\nu\Theta_{22}&\nu\Theta_{23}\cr\nu\Theta_{31}&\nu\Theta_{32}&(1-\nu) T+\nu\Theta_{33}\end{array}\right)\cr
=(1-\nu) T Id+\nu\Theta,
\end{align*}
where $Id$ is the $d_2\times d_2$ identity matrix. The ellipsoidal relaxation operator satisfies the following cancellation property:
\begin{eqnarray*}
\int_{\mathbb{R}^{d_2}} \big(\mathcal{M}_{\nu}(f)-f\big)\left(1,v,|v|^2\right)dv=0,
\end{eqnarray*}
which leads to the conservation of mass, momentum and energy:
\[
\frac{d}{dt}\int_{\mathbb{R}^{d_1}\times\mathbb{R}^{d_2}}fdxdv=\frac{d}{dt}\int_{\mathbb{R}^{d_1}\times\mathbb{R}^{d_2}}vfdxdv=\frac{d}{dt}\int_{\mathbb{R}^{d_1}\times\mathbb{R}^{d_2}}f|v|^2dxdv=0.
\]
%(See \cite{A-L-P-P,C-C,Holway,Stru}).

When Holway first suggested this model, $H$-theorem was not verified, which was the main reason why the ES-BGK model has been neglected in the literature until very recently.  It was resolved in \cite{ALPP} (See also \cite{B3,Yun4}):
\begin{eqnarray*}
\frac{d}{dt}\int_{\mathbb{R}^{d_1}\times\mathbb{R}^{d_2}}f\ln f\,dxdv\leq 0,
\end{eqnarray*}
and ignited the interest on this model \cite{ABLP,B3,FJ,GT,MS,MWR,Stru,Yun2,Yun3,Yun4,Z-Stru}.

It can be verified via the Chapman-Enskog expansion that the Prandtl number computed using the ES-BGK model is $1/(1-\nu)$.
Therefore, the correct physical Prandtl number can be recovered by choosing appropriate $\nu$, namely, $\nu=1-1/Pr\approx-1/2$, where $Pr$ denotes the correct Prandtl number. When $\nu=0$, the ES-BGK model reduces to the original BGK model.
Hence, any results for the ES-BGK model automatically hold for the original BGK model either.
We also mention that, in the range $-1/2<\nu<1$,
the only possible equilibrium state of the ellipsoidal relaxation operator is the usual Maxwellian, not the ellipsoidal Gaussian.
That is, the only solution satisfying the relation $\mathcal{M}_{\nu}(f)=f$ is the local Maxwellian (See \cite{Yun2,Yun4}
for the proof.):
\[
f=\mathcal{M}_0(f)=\frac{\rho}{\sqrt{(2\pi T)^{d_2}}}\exp\left(-\frac{|v-U|^2}{2T}\right).
\]
Therefore, the ES-BGK model correctly captures the two most important asymptotic behavior of the Boltzmann equation, namely, the time asymptotic limit
($t\rightarrow\infty$) and the hydrodynamic limit at the Euler level ($\kappa\rightarrow0$). \newline

%The correct Prandtl number, which is the main reason that (\ref{ESBGK}) was first introduced is obtained when $\nu=\approx-1/2$.

%In this paper, we propose a semi-Lagrangian scheme for ellipsoidal BGK model and study its convergence property.
In recent years, semi-Lagrangian (SL in short) methods for the numerical solutions of kinetic equations have attracted the attention of several authors, see for example, the papers \cite{besse1,besse2,BS,CDM,CV,CaVe-07,CMS-10,FB,Qiu-Ch-10,Qiu-Shu-11,SRBG-99,XRQ} on the use of SL schemes for the Vlasov-Poisson equations and the works \cite{DMRL-15,FR,GRS,Qiu-Russo,RSY} on SL schemes for the BGK model. SL methods are very attractive, since they allow one to use large time step, with no CFL-type accuracy restriction typical of Eulerian-based schemes.
In the limit of small Knudsen number, they can capture the underlying fluid dynamic limit with an implicit L-stable scheme adopted for the treatment of the collision term. It is also relatively easy to obtain high accuracy, by using high order reconstruction of the solution at the foot of the characteristics. One of the difficulties with semi-Lagrangian schemes is that they are naturally constructed in non-conservative form.
Several papers have been devoted to the construction of conservative SL schemes (see, for example, \cite{CMS-10,Qiu-Ch-10,Qiu-Shu-11,XRQ}). SL schemes have also been used to solve specific problems where accurate solutions are needed. In \cite{Aoki-PRE14}, the  authors adopt a SL method to study the decay of an oscillating plate in a rarefied gas described by the BGK model and compare the results with those obtained by a scheme specifically designed for the problem. For a survey on numerical schemes on BGK model, we refer to the review paper \cite{DiMarco-Pareschi-review-2014,Miu} and references therein.

In view of the increasing interest in the subject, our aim is to introduce a new family of semi-Lagrangian scheme for the ES-BGK model and to investigate their convergence properties. Numerical results of the schemes will be presented in a separate paper.
\subsection{Implicit semi-Lagrangian scheme}
First, we propose an implicit semi-Lagrangian scheme.
In the numerical computation of the collisional or relaxational kinetic equations, it is common to employ
the so called splitting method, which amounts to computing the transport part:
\begin{eqnarray}\label{transport}
\partial_tf+v\cdot\nabla_xf=0,
\end{eqnarray}
and the relaxational time evolution:
\begin{eqnarray}\label{relaxational evolution}
\frac{df}{dt}=\frac{1}{\kappa}A_{\nu}\left\{\mathcal{M}_{\nu}(f)-f\right\},
\end{eqnarray}
separately. The most naive way for this is to use the forward in time methods and the explicit Euler type method
respectively for (\ref{transport}) and (\ref{relaxational evolution}). It is, however, well-known that the first procedure leads to the restriction
on the temporal grid size due to the CFL condition: $\triangle t<\triangle x/\max_j|v_j|$,
while the second procedure entails a stability condition: $\triangle t <C\kappa$ for some constant $C>0$,
resulting in the following two scale restriction on the size of time step:
\[
\triangle t<\min\left\{C\kappa, \triangle x/\max_j |v_j|\right\}.
\]
Since the two parameters of the right-hand side are independent of each other,
%and the size of the grid nodes should be restricted to the smaller scale of the two to guarantee the stability,
the discrepancy between these two scale can get arbitrary large, and the scheme becomes severely resource-consuming accordingly.  Such a stiffness problem  has been one of the key difficulties in developing efficient stable schemes for kinetic equations.
In this paper, we propose a new semi-Lagrangian scheme, which combines two numerical methods known to guarantee stable performances,
namely the semi-Lagrangian treatment for the transport part (\ref{transport}), and the implicit Euler for the evolution part (\ref{relaxational evolution}),
to overcome the CFL restriction and secure the stability of the scheme over the large range of Knudsen number at the same time:
\[
\frac{f^{n+1}_{i,j}-\tilde{f}^n_{i,j}}{\triangle t}=\frac{1}{\kappa}A_{\nu}\big\{\mathcal{M}_{\nu,j}(f^{n+1}_i)-f^{n+1}_{i,j}\big\}.
\]
Here  $\widetilde{f}^n_{i,j}$ denotes the linear reconstruction. (See Definition \ref{def21}.)
%\begin{eqnarray*}\widetilde f^{n}_{i,j}=\frac{x(i,j)-x_{s,j}}{\triangle x}~f^{n}_{s+1,j}+\frac{x_{s+1,j}-x(i,j)}{\triangle x}~f^{n}_{s,j}.
%%\label{shift}\end{eqnarray*}
%which denotes the linear reconstruction of the approximate solution on the intersection of the characteristic line and the spatial domain on the th time step ( figure 3):
%\begin{figure}[p]\centering\includegraphics[width=0.8\textwidth]{semi.jpg}
    %\caption{Awesome Image}
    %\label{fig:awesome_image}
%\end{figure}
%The definitions of $x(i,j)$ and $x_{s,j}$ will be given in Section 2.
%Such implicit semi-Lagrangian scheme is expected to yield stable performance for large CFL number and small Knudsen number.
At first sight, the scheme seems very time consuming due to the implicit implementation of the relaxation part.
%Inefficiency at the price of stability is a well-known shortcoming of such implicit schemes.
In the case of the original BGK model ($\nu=0$),  such difficulties can be circumvented by a clever trick using the fact that (1) the local Maxwellian depends on
the distribution function only through the conservative macroscopic fields, and (2) the macroscopic fields satisfy the following identities:
\begin{equation}\label{con}
\rho^{n+1}_i=\widetilde{\rho}^n_i,\quad U^{n+1}_i=\widetilde{U}^n_i,\quad T^{n+1}_i=\widetilde{T}^n_i,\quad
\end{equation}
with small error, enabling one to explicitly solve for the numerical solutions \cite{FR,PP,RSY}. Here, the macroscopic variables with tilde are those constructed from $\tilde{f}^n_{i,j}$. (see Section 2.)
In this surprising turn of events, the two seemingly contradicting properties: the stable performance of the implicit scheme and the efficiency of the explicit scheme, are reconciled. Such nice feature, of course, can never be expected for the Boltzmann type collision operators.

In the case of the ES-BGK model, however, the conservation laws are not sufficient to make this trick work, since the ellipsoidal
Gaussian contains the stress tensor, which is not a conserved quantity.
Even though we cannot expect the stress tensor to satisfy similar conservation identities as (\ref{con}), we observe that the following approximation holds with small error (See (\ref{Theta n+1}) in Section 2):
\begin{eqnarray*}
\Theta^{n+1}_i\approx\frac{\triangle t}{\kappa+\triangle t}\widetilde{T}^n_iId+\frac{\kappa}{\kappa+\triangle t}\widetilde{\Theta}^n_i,
\end{eqnarray*}
which  enables us to rewrite the implicit ellipsoidal Gaussian in a semi-explicit manner.
%and secure the strongest attraction of this scheme: the stability of the implicit semi-Lagrangian scheme and the efficiency of explicit scheme at one stroke.
The resultant scheme can now be written in an explicit manner as (See Section 2.2)
\begin{align*}
f^{n+1}_{i,j}&=\frac{\kappa}{\kappa+A_{\nu}\triangle t}\widetilde{f}^n_{i,j}+\frac{A_{\nu}\triangle t}{\kappa+A_{\nu}\triangle t}\mathcal{M}_{\widetilde{\nu},j}(\tilde{f}^n_i).
\end{align*}
We remark that the implementation of the scheme and its check on several numerical tests  will be reported in an independent paper \cite{BCRY}.
\newline

\subsection{$L^{\infty}$ convergence theory} We then develop a convergence theory for this scheme that will guarantee the credibility of the method.
In this regard, our main result is the following error estimate stated in Theorem 3.2 in Section 3:
\begin{eqnarray*}
\|f(T^f)-f^{N_t}\|_{L^{\infty}_q}\leq C\Big\{\triangle x+\triangle v+\triangle t+\frac{\triangle x}{\triangle t}~\Big\},
\end{eqnarray*}
for some constant $C>0$.  Compared to our previous result \cite{RSY} where the convergence of a semi-Lagrangian scheme for the original BGK model was established in $L^1$ space:
\begin{eqnarray*}
\|f(T^f)-f^{N_t}\|_{L^{1}_2}\leq C\Big\{\triangle x+\triangle v+\triangle t+\frac{\triangle x}{\triangle t}~\Big\},
\end{eqnarray*}
we make four improvements. The error estimate in the spatial node is improved from $\triangle x$ to $(\triangle x)^2$ by assuming additional regularity on the initial data and refining the analysis of the interpolation part. This enables one to recover the first order error estimate by choosing $\triangle =\triangle t$.  Secondly, we impose size restriction only on the velocity nodes, whereas in \cite{RSY}, we needed to restrict the size of all the node size: $\triangle x$, $\triangle v$ and $\triangle t$. Thirdly, we develop a theory to measure the error in a weighted $L^{\infty}$ norm instead of the weighted $L^1$ norm, which provides a more clear and detailed picture on the convergence of the scheme, since the  error estimate in $L^{\infty}$ norm gives a node-wise convergence estimate. Finally, the proof is greatly simplified, enabling one to extend the convergence theory to the whole range of relaxation parameter $-1/2<\nu<1$. As a result, we derive a convergence estimate that is uniform in $-1/2\nu<1$. Note that, since the above convergence estimate is valid for the whole range of relaxation parameter $-1/2<\nu<1$ uniformly, it also holds for the original BGK model, which corresponds to $\nu=0$.

The most important step of the proof is the derivation of the following uniform stability estimate (For notations, see the next subsection):
\begin{align*}
C_{0,1}e^{-\frac{A_{\nu}}{\kappa}T^f}e^{-C_{0,2}|v_j|^{\alpha}}\leq f^{n}_{i,j}\leq e^{\frac{(C_{\mathcal{M}}-1) T^f}{\kappa+A_{\nu}\triangle t}}\|f^0\|_{L^{\infty}_q}(1+|v_j|)^{-q},
\end{align*}
which comes from the uniform-in-$n$ control of the discrete ellipsoidal Gaussian in a weighted $L^{\infty}$ norm (See Lemma \ref{Control Max}):
\[
\|\mathcal{M}_{\widetilde{\nu}}(f^n)\|_{L^{\infty}_q}\leq C_{\mathcal{M}}\|f^n\|_{L^{\infty}_q}.
\]
Two  technical issues arise in the process of obtaining the above estimates.
First, we need to show that the discrete temperature tensor $\widetilde{\mathcal{T}}^n_{\tilde{\nu},i}$ is
strictly positive definite uniformly in $n$:
\[
k^{\top}\big\{\widetilde{\mathcal{T}}^n_{\widetilde{\nu},i}\big\}k\geq C_{\nu, q, \kappa, T^f}>0~\mbox{ for all }\kappa\in \mathbb{S}^2.
\]
Otherwise, since the discrete ellipsoidal Gaussian involves the inverse of $\mathcal{T}^n_{\tilde{\nu},i}$
and $\det(\mathcal{T}^n_{\tilde{\nu},i})$, it may blow up as  $\mathcal{T}^n_{\tilde{\nu},i}$ approaches arbitrarily close to zero.
%Such a pointwise lower bound estimate then leads to a pointwise lower bound for the discrete temperature tensor ($k\in \mathbb{S}^3$):
%\[
%k^{\top}\big\{\widetilde{\mathcal{T}}^n_{\widetilde{\nu},i}\big\}k\geq C_{\nu, q, \kappa, T^f}>0,
%\]
Second issue is more subtle. It turns out that we need to show that ratios such as
\begin{align*}
\widetilde{\rho}^n_i\widetilde{T}^n_{i}/ \|\widetilde{f}^n\|_{L^{\infty}_q},\quad
\|\widetilde{f}^n\|_{L^{\infty}_q}/\widetilde{\rho}^n_i
%\left(\frac{\left\{\widetilde{\rho}^n_i\right\}^{1/q}(3\widetilde{T}^n_i+|\widetilde{U}^n_i|^2)^{1/2}\big\{\widetilde{T}^n_i\big\}^{1/2}}{(2\pi)^{3/q}\|f^n\|^{1/q}_{L^{\infty}_q}}\right)^{q/(q+3)}
\end{align*}
remain strictly larger than $\triangle v$, to guarantee the existence of proper decomposition of the macroscopic fields to derive necessary moment estimates. (See Lemma \ref{moments estimate discrete}.) The restriction on the velocity node in Theorem 3.2 mostly comes from this subtle technical issue.

%%%%%%%%%%%%%%%%%%%%%%%%%%%%%%%%%%%%%%%%%%%%%%%%%%%%%%%%%%%%%%%%%%%%%%%%%%%%%%%%%%%%%%%%%%%%%%%%%%%%%%%%%%%%%%%%%%%%%%%%%%%%%%%%%%%%%%%%%%%%%%%%%%%%%%%%%%%%%%%%%%%%%%%%%%%%%%%%%%%%%%5
%
%
%
%
%
%%%%%%%%%%%%%%%%%%%%%%%%%%%%%%%%%%%%%%%%%%%%%%%%%%%%%%%%%%%%%%%%%%%%%%%%%%%%%%%%%%%%%%%%%%%%%%%%%%%%%%%%%%%%%%%%%%%%%%%%%%%%%%%%%%%%%%%%%%%%%%%%%%%%%%%%%%%%%%%%%%%%%%%%%%%%%%%%%%%%%%55
\subsection{Notation} Before we finish this introduction, we summarize the notational convention kept throughout this paper:
\begin{itemize}
\item $C$ denotes generic constants. The exact value may change in each line, but they are explicitly computable in principle.
%\item $C_{x,y,..}$ denotes generic constants that depend on $x$, $y$,... but not necessarily exclusively.
\item We will use lower indices $n$, $i$, $j$ exclusively for time, space and velocity variable respectively. For example
$x_i$, $v_j$, $t_n$.
\item We use upper indices for the components of vectors as $v=(v^{1},v^2,v^3)$, while the lower indices
are reserved for the spatial, velocity and temporal nodes.
\item $T^f$ will denote the final time, whereas $T_f$ represents for the local temperature constructed from the distribution function $f$.
\item We use the following notation for weighted $L^{\infty}$-Sobolev norm for continuous solution:
\begin{align*}
\displaystyle \|f\|_{L^{\infty}_q}&=\sup_{x,v}|f(x,v)(1+|v|)^q|,\quad\|f\|_{W^{\ell,\infty}_q}=\sum_{|\alpha|+|\beta|\leq \ell}\|\partial^{\alpha}_{\beta}f\|_{L^{\infty}_q},
\end{align*}
where, $\alpha,\beta\in \mathbb{Z}_+$ and, the differential operator $\partial^{\alpha}_{\beta}$ stands for
$\partial^{\alpha}_{\beta}=\partial^{\alpha}_{x}\partial^{\beta}_{v}$.
\item For any sequence $a^n_{i,j}$, we use the following  notation for the weighted $L^{\infty}$-norm for of the sequence:
\begin{align*}
\|a^n\|_{L^{\infty}_q}&=\sup_{i,j}|a^n_{i,j}(1+|v_{j}|)^q|.
%\|f(t)\|_{L^{\infty}_q}&=\sup_{i,j}|f(x_{i},v_{j},t)(1+|v_{j}|)^q|.
\end{align*}
\item In view of the above norms, $\|f(t_n)-f^n\|_{L^{\infty}_q}$ is understood, with a slight abuse of notation, as
\begin{align*}
\|f(t_n)-f^n\|_{L^{\infty}_q}&=\sup_{i,j}\left|\,\big\{f(x_i,v_j,t_n)-f^n_{i,j}\big\}(1+|v_{j}|)^q\right|,
\end{align*}
for simplicity.
\end{itemize}
%and%%%%%%%%%%%%%%%%%%%%%%%%%%%%%%%%%%%%%%%%%%%%%%%%%%%%%%%%%%%%%%%%%%%%%%%%%%%%%%%%%%%%%%%%%%%%%%%%%%%%%%%%%%%%%%%
%
%            Subsection:  Existence of solution
%
%%%%%%%%%%%%%%%%%%%%%%%%%%%%%%%%%%%%%%%%%%%%%%%%%%%%%%%%%%%%%%%%%%%%%%%%%%%%%%%%%%%%%%%%%%%%%%%%%%%%%%%%%%%%%%%%
The paper is organized as follows. In the following Section 2, we derive
a semi-Lagrangian scheme for the ES-BGK model.
The main convergence result of this paper is presented in Section 3.
In section 4, we establish some technical lemmas.
Section 5 is devoted to the stability estimate of the scheme.
In Section 6, we  transform the ES-BGK
model (\ref{ESBGK}) to a form consistent with our scheme.
In Section 7, we estimate the discrepancy of discrete Gaussian and the continuous one.
Finally, we prove the convergence of our scheme in Section 7.

%%%%%%%%%%%%%%%%%%%%%%%%%%%%%%%%%%%%%%%%%%%%%%%%%%%%%%%%%%%%%%%%%%%%%%%%%%%%%%%%%%%%%%%%%%%%%%%%%
%
%
%           Section:Description of the numerical scheme
%
%
%%%%%%%%%%%%%%%%%%%%%%%%%%%%%%%%%%%%%%%%%%%%%%%%%%%%%%%%%%%%%%%%%%%%%%%%%%%%%%%%%%%%%%%%%%%%%%%%%
\section{ Description of the numerical scheme}
We fix $d_1=1$ and $d_2=3$ case with periodic boundary condition throughout this paper
in order to stay in the simplest possible framework. We believe that the analysis of this paper can be extended to more general conditions
such as higher dimensions in $x$ and/or different boundary conditions, although such extensions may give rise to unexpected difficulties. This will be a topic of future work.
Note that, in contrast to the original BGK model, the velocity domain must be at least 2-dimensional for the ellipsoidal BGK to be meaningful. Otherwise the model reduces to the original BGK model.
We choose a constant time step $\triangle t$ with final time $T^f$. The spatial domain and the velocity domain are divided into
uniform grids with mesh size $\triangle x$, $\triangle v$ respectively: %We denote the grid points as
\begin{align*}%\label{gridnodes1}
&t_n=n\triangle t,\hspace{1cm}n=0,1,...,N_t, \cr
&x_i=i\triangle x ,\hspace{1.1cm}i=0,\pm1,...,\pm N_x, \pm (N_x+1),\cdots
\end{align*}
where $N_t\triangle t = T^f$, $N_x\triangle x=1$, and
\begin{align*}
v_j=(v_{j_{1}}, v_{j_{2}},v_{j_{3}})=(j_1\triangle v, j_2\triangle v,j_3\triangle v),\quad j=(j_1,j_2,j_3)\in \mathbb{Z}^3.
%&&v_{j_{\alpha}}=j_{\alpha}\triangle v,\hspace{0.95cm}j=-N_v,..,0,..,N_v,\nonumber
\end{align*}
Note that the spatial node is defined on the whole line instead of unit interval, even though we are considering periodic
problem. Periodicity will be imposed on the initial data $f_0$, which is defined on the whole line with period 1.
This facilitates the proof in several places.
If not specified otherwise, we assume throughout this paper that $n\leq N_t$, to avoid unnecessary repetition.
%We discretize
%We then truncate the each axis of the velocity coordinate as $-R\leq v\leq R$ and set $\triangle v=\frac{2R}{N_v}$
%$j=(j_1,j_2)$. We denote $v_j=(v^1_{j_1},v^1_{j_2})$ and define $v^{\alpha}_{j_{\alpha}}$ inductively as
%\begin{align}\label{gridnodes2}
%.%=v_1+(j_{\alpha}-1)\triangle v,\hspace{0.95cm}j=1,..,N_v,\cr
%&&v_{j_}=v_1-|j_{\alpha}|\triangle v,\hspace{0.95cm}j=-1,..,-N_v, \end{align}
 We  denote the approximate solution of $f(x_i,v_j,t_n)$ by $f^n_{i,j}$.
To describe the numerical scheme more succinctly, we introduce the following convenient notation.
%%%%%%%%%%%%%%%%%%%%DEFINITION of Tilde%%%%%%%%%%%%%%%%%%%%%%%%%%%%%%%%%%%%%%%%%%%%%%%%%%%%%%%%%%%%%%%%%%%%%%%%
\begin{definition}\label{def21}
(1)~ Let $x(i,j)=x_i-\triangle t \,v_{j_1}$. Then define $s=s(i,j)$ to be the spatial node such that
$x(i,j)$ lies in $[x_s, x_{s+1})$. \newline
(2)~ The reconstructed distribution function $\widetilde f^{n}_{i,j}$ is defined as
\begin{align*}
\widetilde f^{n}_{i,j}=\frac{x(i,j)-x_s}{\triangle x}~f^{n}_{s+1,j}+\frac{x_{s+1}-x(i,j)}{\triangle x}~f^{n}_{s,j}.%\label{shift}
\end{align*}
\begin{remark}
$\tilde{f}^n_{i,j}$ is the linear interpolation of  $f^n_{s,j}$ and $f^{n}_{s+1,j}$ on $x(i,j)$.
\end{remark}
\end{definition}
We also need to define the discrete ellipsoidal Gaussian:
\begin{align*}
{\mathcal M}_{\nu,j}(f^{n}_i)=\frac{ \rho^n_{i}}{\sqrt{\det(2\pi{\mathcal{T}}^n_{\nu,i})}}\exp
\Big(-\frac{1}{2}(v_j- {U}^n_{i})\big\{ {\mathcal{T}}^n_{\nu,i}\big\}^{-1}(v_j- {U}^n_{i})\Big),
\end{align*}
and the discrete macroscopic field $\rho^n_{i}, U^n_{i}$, $ T^n_{i}$, ${\Theta}^n_{i}$ and ${\mathcal{T}}^n_{\nu,i}$:
\begin{align*}
&\rho^n_{i} = \sum_{j}  f^{n}_{i,j}(\triangle v)^3,\quad
\rho^n_{i}  U^{n}_{i} = \sum_{j}  f^{n}_{i,j} v_{j}(\triangle v)^3, \quad
3\rho^n_{i}  T^{n}_{i} = \sum_{j}  f^{n}_{i,j} \big|v_{j}-{U}^{n}_i\big|^2(\triangle v)^3,\cr
&\quad\rho^n_{i}  \Theta^{n}_{i} = \sum_{j}  f^{n}_{i,j}( v_j- {U}^n_{i})\otimes( v_j- {U}^n_{i})
(\triangle v)^3,\quad
{\mathcal{T}}^{n}_{\nu,i}= (1-{\nu}) T^{n}_{i}Id+{\nu} \Theta^{n}_{i}.
\end{align*}
%We now explain how our scheme (\ref{main scheme}) is derived.
%%%%%%%%%%%%%%%%%%%%%%%%%%%%%%%%%%%%%%%%%%%%%%%%%%%%%%%%%%%%%%%%%%%%%%%%%%%%%%%%%%%%%%%%%%%%%%%%%%%%%%%%%5
%
%
%
%
%%%%%%%%%%%%%%%%%%%%%%%%%%%%%%%%%%%%%%%%%%%%%%%%%%%%%%%%%%%%%%%%%%%%%%%%%%%%%%%%%%%%%%%%%%%%%%%%%%%%%%%5%
%The numerical scheme for (\ref{main.1}) is based on the following characteristic formulation of the problem:
%\begin{align}\begin{aligned}\label{characteristic.formulation1}
%&\frac{df}{dt}=\frac{1}{\kappa}(\mathcal{M}(f)-f),\\&\frac{dx}{dt}=v.
%&x(0)=\widetilde{x},\quad f(x,v,0)=f_0(\widetilde{x},v),\quad t\geq0, \quad x,v\in\R.
%\end{aligned}\end{align}
\subsection{Derivation of the scheme (\ref{main scheme})}
%We rewrite the ES-BGK model along the characteristic line:
%\begin{align}\label{characteristic.formulation2}
%\frac{d}{ds}\{f(x-(t-s)v,v,t+s)\}=\{\mathcal{M}_{\nu}(f)-f\}(x-(t-s)v,v,s)
%\end{align}
Let $f_j=f(x,v_j,t)$. We rewrite the ES-BGK model (\ref{ESBGK}) in the characteristic formulation:
%\begin{align*}f(x_i,v_j,t_n+\triangle t)&=f(x_i-\triangle t v_{j_1},v_j,t_n)\cr
%&+\frac{1}{\kappa}A_{\nu}\int^{t_n+\triangle t}_{t_n}\{\mathcal{M}_{\nu}(f)-f\}(x_i-(t_n+\triangle t-s)v_{j_1},v_j,s)ds.\end{align*}
\begin{align}\label{CF}
\begin{split}
\frac{df_j}{dt}&=\frac{1}{\kappa}A_{\nu}\big(\mathcal{M}_{\nu}(f_j)-f_j\big),\cr
\frac{dx}{dt}&=v_j.
\end{split}
\end{align}
Using implicit Euler scheme on (\ref{CF}), we obtain
\[
f^{n+1}_j(x_{i})-f^n_j(x_i-v_{j_1}\triangle t)=\frac{\triangle t}{\kappa}A_{\nu}\big\{\mathcal{M}_{\nu,j}(f^{n+1})-f^{n+1}_j\big\}(x_i).
\]
We then approximate $f^{n}_j(x_{i})$ by $f^n_{i,j}$, and $f^n_j(x_i-v_{j_1}\triangle t)$ by $\widetilde{f}^n_{i,j}$ to obtain
\begin{align}\label{implicit.euler}
\frac{f^{n+1}_{i,j}-\widetilde{f}^n_{i,j}}{\triangle t}=\frac{1}{\kappa}A_{\nu}\big\{\mathcal{M}_{\nu,j}(f^{n+1}_{i})-f^{n+1}_{i,j}\big\}.
\end{align}
 We attempt to convert (\ref{implicit.euler}) into a semi-explicit scheme keeping beneficial features of implicit schemes such as the stability property.
This idea seems to trace back to \cite{CP,PP}, and successfully implemented in the semi-Lagrangian
setting in \cite{FR,GRS,RSY}.
We first impose the conservation of mass, momentum and energy at the discrete level
 (throughout this subsection $\approx$ means that they are identical up to
negligible error):
\begin{align*}
&\frac{\sum_{j}f^{n+1}_{i,j}\phi(v_j)(\triangle v)^3-\sum_{j}\widetilde{f}^n_{i,j}\phi(v_j)(\triangle v)^3}{\triangle t}\cr
&\hspace{1cm}=\frac{1}{\kappa}A_{\nu}\sum_{j}\big\{\mathcal{M}_{\nu,j}(f^{n+1}_i)\phi(v_j)-f^{n+1}_{i,j}\phi(v_j)\big\}(\triangle v)^3\cr
&\hspace{1cm}\approx0,
\end{align*}
for $\phi(v_j)=1,v_j, \frac{1}{2}|v_j|^2$.
Since r.h.s has a spectral accuracy for fast decaying functions, we can legitimately assme that
\begin{equation}\label{imp}
\sum_{j}f^{n+1}_{i,j}\phi(v_j)\approx\sum_{j}\widetilde{f}^n_{i,j}\phi(v_j).
\end{equation}
Therefore, if we define
\begin{align*}
\widetilde\rho^n_{i} = \sum_{j} \widetilde f^{n}_{i,j}(\triangle v)^3,\quad
\widetilde\rho^n_{i} \widetilde U^{n}_{i} = \sum_{j} \widetilde f^{n}_{i,j} v_{j}(\triangle v)^3,\quad
3\widetilde\rho^n_{i} \widetilde T^{n}_{i} = \sum_{j} \widetilde f^{n}_{i,j} \big|v_{j}-\widetilde{U}^{n}_i\big|^2(\triangle v)^3,
%\widetilde \rho^n_{i} \widetilde \Theta^{n}_{i} &= \sum_{j} \widetilde f^{n}_{i,j}( v_j-\widetilde {U}^n_{i})\otimes( v_j-\widetilde {U}^n_{i})
%(\triangle v)^3,\cr
%\widetilde{\mathcal{T}}^{n}_{i}&= (1-\widetilde{\nu})\widetilde T^{n}_{i}Id+\widetilde{\nu} \widetilde\Theta^{n}_{i}.
\end{align*}
then (\ref{imp}) gives the following relation:
\begin{align}\label{macro=tildemacro}
\rho^{n+1}_i\approx\tilde{\rho}^{n}_i,~ U^{n+1}_i\approx\tilde{U}^{n}_i,~T^{n+1}_i\approx\tilde{T}^{n}_i.
\end{align}
In the case of the original BGK mdoel ($\nu=0$), identities in (\ref{macro=tildemacro}) are sufficient to conclude that $\mathcal{M}_{0,j}(f^{n+1}_i)\approx\mathcal{M}_{0,j}(\tilde{f}^n_i)$.
But this is not the case for the ellipsoidal case, since the ellipsoidal Gaussian contains the temperature tensor $\mathcal{T}_{\nu}$, which is not
a conserved quantity. %we also need to consider the stress tensor, which is not a conservation law.
For this, we introduce
%\begin{align*}\phi(v_j)=v-U(v^{}_{j_{}}-U^{}_{i})(v_{j_{}}-U_{})\end{align*}
\begin{align*}
\phi^n_{i,j}\equiv (v_j-U^n_i)\otimes(v_j-U^n_i).
%&\equiv \{(v_{j_{\alpha}}-U^{n,\alpha}_i)(v_{j_{\beta}}-U^{n,\beta}_i)\}_{\alpha\beta}\cr
%&\equiv\left(\begin{array}{ccc}
%(v_{j_{1}}-U^{n,1}_{i})(v^1_{j_{1}}-U^{n,1}_{i})&(v^1_{j_{1}}-U^{n,1}_{i})(v_{j_{2}}-U^{n,2}_{i})&(v^1_{j_{1}}-U^{n,1}_{i})(v_{j_{3}}-U^{n,3}_{i})\cr
%(v_{j_{2}}-U^{n,2}_{i})(v^1_{j_{}}-U^{n,1}_{i})&(v^2_{j_{2}}-U^{n,2}_{i})(v_{j_{2}}-U^{n,2}_{i})&(v^2_{j_{2}}-U^{n,2}_{i})(v_{j_{3}}-U^{n,3}_{i})\cr
%(v_{j_{3}}-U^{n,3}_{i})(v^1_{j_{}}-U^{n,1}_{i})&(v^2_{j_{3}}-U^{n,3}_{i})(v_{j_{2}}-U^{n,2}_{i})&(v^2_{j_{3}}-U^{n,3}_{i})(v_{j_{3}}-U^{n,3}_{i})
%\end{array}\right)
\end{align*}
Multiplying $\phi^{n+1}_{i,j}(\triangle v)^3$ on both sides of (\ref{implicit.euler}) and summing over $i$ and $j$, we get:
\begin{align}\label{prelim}
\begin{split}
&\frac{\sum_j f^{n+1}_{i,j}\phi^{n+1}_{i,j}(\triangle v)^3-\sum_j \widetilde{f}^{n}_{i,j}\phi^{n+1}_{i,j}(\triangle v)^3}{\triangle t}\cr
&\hspace{1.5cm}=\frac{1}{\kappa}A_{\nu}\sum_j \Big\{\mathcal{M}_{\nu,j}(f^{n+1}_i)\phi^{n+1}_{i,j}-f^{n+1}_{i,j}\phi^{n+1}_{i,j}\Big\}(\triangle v)^3.
\end{split}
\end{align}
Let's denote the r.h.s of (\ref{prelim}) by R and the l.h.s by L for simplicity.
We then  recall (\ref{macro=tildemacro}) to observe
\begin{align*}
\phi^{n+1}_{i,j}&=(v_j-U^{n+1}_i)\otimes (v_j-U^{n+1}_i)
\approx(v_j-\widetilde{U}^{n}_i)\otimes (v_j-\widetilde{U}^{n}_i)
=\widetilde{\phi}^n_{i,j}.
\end{align*}
Therefore, the second term of $L$ becomes
\begin{align*}
\sum_j \tilde{f}^{n}_{i,j}\phi^{n+1}_{i,j}(\triangle v)^3\approx\sum_j \tilde{f}^{n}_{i,j}\widetilde{\phi}^{n}_{i,j}(\triangle v)^3=\widetilde{\rho}^n_{i}\widetilde{\Theta}^n_{i},
\end{align*}
where $\widetilde{\Theta}^n_{i}$ is defined by
\begin{align*}
%\widetilde\rho^n_{i} &= \sum_{j} \widetilde f^{n}_{i,j}(\triangle v)^3,\cr
%\widetilde\rho^n_{i} \widetilde U^{n,\alpha}_{i} &= \sum_{j} \widetilde f^{n}_{i,j} v^{\alpha}_{j_{\alpha}}(\triangle v)^3, \qquad(\alpha=1,2,3)\cr
%\widetilde3\rho^n_{i} \widetilde T^{n}_{i} &= \sum_{j} \widetilde f^{n}_{i,j} \big|v_{j}-\widetilde{U}^{n}_i\big|^2(\triangle v)^3,\cr
\widetilde \rho^n_{i} \widetilde \Theta^{n}_{i} &= \sum_{j} \widetilde f^{n}_{i,j}( v_j-\widetilde {U}^n_{i})\otimes( v_j-\widetilde {U}^n_{i})
(\triangle v)^3,
%\widetilde{\mathcal{T}}^{n}_{i}&= (1-\widetilde{\nu})\widetilde T^{n}_{i}Id+\widetilde{\nu} \widetilde\Theta^{n}_{i}.
\end{align*}
so that
\begin{align}\label{lhs}
%&&\frac{\widetilde{\rho}^{n}\Theta^{n+1}-\widetilde{\rho}^{n}\widetilde{\Theta}^n}{\triangle t}\cr
L=\frac{\rho^{n+1}_i\Theta^{n+1}_i-\widetilde{\rho}^{n}_i\widetilde{\Theta}^n_i}{\triangle t}.
%&&\approx\frac{\rho^{n+1}\Theta^{n+1}+U^{n+1}\otimes U^{n+1}-\Big(\widetilde{\rho}^{n}\widetilde{\Theta}+\widetilde{\rho}\widetilde{U}\otimes \widetilde{U}\Big)}{\triangle t}\cr
\end{align}
On the other hand, we find for the right hand side,
\begin{align*}
\begin{split}
R&\approx\frac{1}{\kappa}A_{\nu}\Big\{\rho^{n+1}_i\mathcal{T}^{n+1}_{\nu,i}-\rho^{n+1}_i\Theta^{n+1}_i\Big\}\cr
&= \frac{1}{\kappa}A_{\nu}\Big[\rho^{n+1}_i\big\{(1-\nu)T^{n+1}_i Id+\nu\Theta^{n+1}_i\big\}-\rho^{n+1}_i\Theta^{n+1}_i\Big]\cr
&=\frac{1}{\kappa}A_{\nu}(1-\nu)\left[\rho^{n+1}_iT^{n+1}_iId-\rho^{n+1}_i\Theta^{n+1}_i\right]\cr
&=\frac{1}{\kappa}\left\{\rho^{n+1}_iT^{n+1}_iId-\rho^{n+1}_i\Theta^{n+1}_i\right\},
%&\hspace{2cm}=\frac{1}{\kappa}\widetilde{\rho}^{n}\widetilde{T}^{n}-\frac{1}{\kappa}\Big\{\widetilde{\rho}^{n}\Theta^{n+1}\Big\}.
\end{split}
\end{align*}
and recall (\ref{macro=tildemacro}) to see that
\begin{align}\label{rhs}
R=\frac{1}{\kappa}\left\{\rho^{n+1}_iT^{n+1}_iId-\rho^{n+1}_i\Theta^{n+1}_i\right\}
=\frac{1}{\kappa}\left\{\widetilde{\rho}^{n}_i\widetilde{T}^n_iId-\widetilde{\rho}^{n}_i\Theta^{n+1}_i\right\}.
\end{align}
Now, equating (\ref{lhs}) and (\ref{rhs}), we rewrite (\ref{prelim}) as
\begin{align*}
\frac{\rho^{n+1}_i\Theta^{n+1}_i-\widetilde{\rho}^{n}_i\widetilde{\Theta}^n_i}{\triangle t}
=\frac{1}{\kappa}\left\{\widetilde{\rho}^{n}_i\widetilde{T}^n_iId-\widetilde{\rho}^{n}_i\Theta^{n+1}_i\right\}.
\end{align*}
%We then recall (\ref{macro=tildemacro}) to see\begin{align*}
%\frac{\tilde{\rho}^{n}_i\Theta^{n+1}_i-\widetilde{\rho}^{n}_i\widetilde{\Theta}^n_i}{\triangle t}
%=\frac{1}{\kappa}\widetilde{\rho}^{n}_i\widetilde{T}^n_i-\frac{1}{\kappa}\Big\{\widetilde{\rho}^{n}_i\Theta^{n+1}_i\Big\}.\end{align*}
Dividing both sides by $\rho^{n+1}_i$ and gathering relevant terms, we get
\begin{align}\label{Theta n+1}
\Theta^{n+1}_i=\frac{\triangle t}{\kappa+\triangle t}\widetilde{T}^n_iId+\frac{\kappa}{\kappa+\triangle t}\widetilde{\Theta}^n_i.
\end{align}
Therefore, $\mathcal{T}^{n+1}_{\nu,i}$ can be expressed as
\begin{align}\label{tildeT}
\begin{split}
\mathcal{T}^{n+1}_{\nu,i}&=(1-\nu)T^{n+1}_iId+\nu\Theta^{n+1}_i\cr
&\approx(1-\nu)\widetilde{T}^n_iId+\nu\left\{\frac{\triangle t}{\kappa+\triangle t}\widetilde{T}^n_iId+\frac{\kappa}{\kappa+\triangle t} \widetilde{\Theta}^n_i\right\}\cr
&=\left(1-\frac{ \kappa\nu}{\kappa+\triangle t}\right)\widetilde{T}^n_iId+\left(\frac{\kappa\nu}{\kappa+\triangle t}\right)\widetilde{\Theta}^n_{i}\cr
&\equiv(1-\widetilde{\nu})\widetilde{T}^n_iId+\widetilde{\nu}\widetilde{\Theta}^n_i\cr
&\equiv\widetilde{\mathcal{T}}^n_{\tilde{\nu},i},
\end{split}
\end{align}
where we denoted
\[
\widetilde{\nu}=\frac{ \kappa\nu}{\kappa+\triangle t}.
\]
Using (\ref{macro=tildemacro}) and (\ref{tildeT}), the implicitly defined discrete ellipsoidal Gaussian can now be rewritten in an explicit way as:

\begin{align*}
\mathcal{M}_{\nu,j}(f^{n+1}_i)&=\mathcal{M}_{\nu,j}\big(\rho^{n+1}_i,U^{n+1}_i,\mathcal{T}^{n+1}_{\nu,i}\big)
\approx\mathcal{M}_{\nu,j}\big(\widetilde{\rho}^{n}_i,\widetilde{U}^{n}_i,\widetilde{\mathcal{T}}^{n}_{\tilde{\nu},i}\big)\cr
&=\frac{\widetilde{\rho}^n_i}{\sqrt{\det(2\pi \widetilde{\mathcal{T}}_{\tilde{\nu},i}^n)}}\exp\left(-\frac{1}{2}(v_j-\widetilde{U}^{n}_i)^{\top}
\big\{\widetilde{\mathcal{T}}^{n}_{\tilde{\nu},i}\big\}^{-1}
(v_j-\widetilde{U}^{n}_i)\right).
\end{align*}
With a slight abuse of notation, we now denote the r.h.s as $\mathcal{M}_{\tilde{\nu},j}(\tilde{f}^n_i)$.
%With a slight abuse of notation, we denote the r.h.s as $\mathcal{M}_{\tilde{\nu},j}(\tilde{f}^n_i)$.
We should note carefully that (for example in Lemma \ref{rho-rho})
\begin{align*}
\widetilde{\mathcal{T}}^n_{\tilde{\nu},i}=(1-\tilde{\nu})\widetilde{T}^n_{i}+\tilde{\nu}\widetilde{\Theta}^n_i
\neq(1-\nu)\widetilde{T}^n_{i}+\nu\widetilde{\Theta}^n_i=\widetilde{\mathcal{T}}^n_{\nu,i}
\end{align*}
throughout this paper. We then use this to rewrite the implicit scheme (\ref{implicit.euler}) as the following explicit form:
\begin{align*}
f^{n+1}_{i,j}&=\frac{\kappa}{\kappa+A_{\nu}\triangle t}\widetilde{f}^n_{i,j}+\frac{A_{\nu}\triangle t}{\kappa+A_{\nu}\triangle t}\mathcal{M}_{\widetilde{\nu},j}(\tilde{f}^n_i).
\end{align*}

\subsection{Implicit semi-Lagrangian scheme}
%%%%%%%%%%%%%%%%%%%%%%%%%%%%%%%%%%%%%%%%%%%%%%%%%%%%%%%%%%%%%%%%%
%
%         Semi-Lagrangian scheme
%
%%%%%%%%%%%%%%%%%%%%%%%%%%%%%%%%%%%%%%%%%%%%%%%%%%%%%%%%%%%%%%%%
Summarizing, our semi-Lagrangian scheme for the ES-BGK model (\ref{ESBGK}) reads:
\begin{align}\label{main scheme}
f^{n+1}_{i,j}&=\frac{\kappa}{\kappa+A_{\nu}\triangle t}~\widetilde{f}^{n}_{i,j}+\frac{A_{\nu}\triangle t}{\kappa+A_{\nu}\triangle t}~
{\mathcal M}_{\widetilde{\nu},j}(\widetilde{f}^{n}_{i}),
\end{align}
where the discrete ellipsoidal Gaussian ${\mathcal M}_{\widetilde{\nu},j}(\widetilde{f}^{n}_i)$ is defined as follows:
\begin{align*}
{\mathcal M}_{\widetilde{\nu},j}(\widetilde{f}^{n}_i)=\frac{ {\widetilde\rho}^n_{i}}{\sqrt{\det(2\pi\widetilde {\mathcal{T}}^n_{\widetilde{\nu},i})}}\exp
\Big(-\frac{1}{2}(v_j-\widetilde {U}^n_{i})\big\{\widetilde {\mathcal{T}}^n_{\widetilde{\nu},i}\big\}^{-1}(v_j-\widetilde {U}^n_{i})\Big),
\end{align*}
and discrete macroscopic field $\widetilde{\rho}^n_{i}, \widetilde{U}^n_{i}$, $\widetilde T^n_{i}$, $\widetilde{\theta}^n_{i,j}$ and $\widetilde{\mathcal{T}}^n_{\widetilde{\nu},i}$ $(n\geq 1)$ are given by
\begin{align}\label{Conservatives_Original}
&\widetilde\rho^n_{i} = \sum_{j} \widetilde f^{n}_{i,j}(\triangle v)^3,\quad
\widetilde\rho^n_{i} \widetilde U^{n}_{i} = \sum_{j} \widetilde f^{n}_{i,j} v_{j}(\triangle v)^3,\quad
\widetilde3\rho^n_{i} \widetilde T^{n}_{i} = \sum_{j} \widetilde f^{n}_{i,j} \big|v_{j}-\widetilde{U}^{n}_i\big|^2(\triangle v)^3,\cr
&\quad\widetilde \rho^n_{i} \widetilde \Theta^{n}_{i} = \sum_{j} \widetilde f^{n}_{i,j}\big\{( v_j-\widetilde {U}^n_{i})\otimes( v_j-\widetilde {U}^n_{i})\big\}
(\triangle v)^3,\quad
\widetilde{\mathcal{T}}^{n}_{\widetilde{\nu},i}= (1-\widetilde{\nu})\widetilde T^{n}_{i}Id+\widetilde{\nu} \widetilde\Theta^{n}_{i}.
\end{align}
In the last line, $\tilde{\nu}$ denotes
\[
\tilde{\nu}=\frac{\kappa \nu}{\kappa+\triangle t}.
\]
For the initial step $(n=0)$, to ignore the error arising in the discretization of the  initial data and simplify the convergence proof,
we sample values directly from continuous distribution function and macroscopic fields at $t=0$:
\begin{align*}
f^0_{i,j}=f_0(x_i,v_j),\quad \widetilde{f}^0_{i,j}=\widetilde{f_0}(x_i,v_j)=f_0(x_i-v_j\triangle t,v_j),
\end{align*}
and
\begin{align*}
\widetilde{\rho}^0_{i}&=\widetilde{\rho}(x_i,0)=\int_{\mathbb{R}^3}f_0(x_i-v\triangle t, v)dv,\cr
\widetilde{\rho}^0_{i}\widetilde{U}^0_{i}&=\widetilde{\rho}(x_i,0)\widetilde{U}(x_i,0)=\int_{\mathbb{R}^3}f_0(x_i-v\triangle t, v)vdv,\cr
3\widetilde{\rho}^0_{i}\widetilde{T}^0_{i}&=3\widetilde{\rho}(x_i,0)\widetilde{T}(x_i,0)=\int_{\mathbb{R}^3}f_0(x_i-v\triangle t, v)|v-\widetilde{U}^0_i|^2dv.
\end{align*}
%Note that $\widetilde{\Theta}^n_i$ and $\widetilde{\mathcal{T}}^n_i$ are $3\times 3$ matrices, whose components can be represented in the component-wise notation as
%$(1\leq\alpha,\beta\leq 3)$\begin{align}\label{Conservatives_Original2}
%3\widetilde \rho^n_{i} \big\{\widetilde \Theta^{n}_{i}\big\}^{\alpha,\beta} &= \sum_{j} \widetilde f^{n}_{i,j}( v_{j_{\alpha}}-\widetilde {U}^{n,\alpha}_{i})
%( v_{j_{\beta}}-\widetilde {U}^{n,\beta}_{i})(\triangle v)^3,\cr\big\{\widetilde{\mathcal{T}}^{n}_{i}\big\}^{\alpha,\beta}&= (1-\widetilde{\nu}) \widetilde %{T}^{n}_{i}\delta^{\alpha,\beta}+\widetilde{\nu}\big\{\widetilde\Theta^{n}_{i}\big\}_{\alpha,\beta}.\end{align}
%where $\delta^{\alpha,\beta}$ denotes the Kronecker delta.
\section{Main results}
We are now ready to state our main result. We first record the existence result relevant to our convergence proof.
%%%%%%%%%%%%%%%%%%%%%%%%%%%%%%%%%%%%%%%%%%%%%%%%%%%%%%%%%%%%%%%%%%%%%%%%%%%%%%%%%%%%%%%%%%%%%%%%%%%%%%%%%
%
%       Existence of smooth solutions
%
%
%%%%%%%%%%%%%%%%%%%%%%%%%%%%%%%%%%%%%%%%%%%%%%%%%%%%%%%%%%%%%%%%%%%%%%%%%%%%%%%%%%%%%%%%%%%%%%%%%%%%%%%%%
%%%%%%%%%%%%%%%%%%%%%%%%%%%%%%%%%%%%%%%%%%%%%%%%%%%%%%%%%%%%%%%%%%%%%%%%%%%%%%%%%%%%%%%%%%%%%%%%%%%%%%%%%%%%%%%%%%%%%%%%%%%%%%%5
%%%%%%%%%%%%%%%%%%%%%%%%%%%%%%%%%%%%%%%%%%%%%%%%%%%%%%        Existence Result      %%%%%%%%%%%%%%%%%%%%%%%%%%%%%%%%%%%%%%%%%%%%%%%%5
%%%%%%%%%%%%%%%%%%%%%%%%%%%%%%%%%%%%%%%%%%%%%%%%%%%%%%%%%%%%%%%%%%%%%%%%%%%%%%%%%%%%%%%%%%%%%%%%%%%%%%%%%%%%%%%%%%%%%%%%%%%%%%%%
\begin{theorem}\label{Existence_Theorem} \emph{\cite{Yun3}}~%Suppose $f_0\geq 0$, $f_a0\in L^1(\mathbb{T}\times R)$.
Let $-1/2<\nu<1$ and $q>5$. Let $f_0$ satisfy $ \|f_0\|_{W^{2,\infty}_q}<\infty$. Suppose further that there exist positive constants $C^{1}_0, C_{0}^{2}$ and $\alpha$ such that
\begin{align*}%\label{LowerBoundforLocalDensity}
f_0(x,v)\geq C_{0}^{1}e^{-C_{0}^{2}|v|^{\alpha}}.
\end{align*}
Then, for any final time $T^f>0$, the ES-BGK model (\ref{ESBGK}) has a unique solution $f\in C([0,T^f], \|\cdot\|_{W^{2,\infty}_q})$
such that
\begin{enumerate}
\item $f$ is bounded in $\|\cdot\|_{W^{2,\infty}_q}$ for $[0,T^f]$:
\begin{align*}
 \|f(t)\|_{W^{2,\infty}_q}\leq C_{1}e^{C_{2} t}\left\{\|f_0\|_{W^{1,\infty}_q}+1\right\}, \quad t\in [0,T^f],
\end{align*}
for some constants $C_{1}$ and $C_{2}$.
\item The macroscopic fields satisfy the following lower and upper bounds:
\begin{align*}%\label{ULBoundsMacroscopicContinuous}
&\|\rho(t)\|_{L^{\infty}_x}+\|U(t)\|_{L^{\infty}_x}
+\|T(t)\|_{L^{\infty}_x}\leq C_qe^{C_qT^f}\!\!\!,~\\
&\hspace{2cm}\rho(x,t)\geq C_{N,q}e^{-C_{N,q} t},\\
&\hspace{0.8cm}k^{\top} \big\{\mathcal{T}_{\nu}(x,t)\big\}k\geq C_{N,q}e^{-C_{N,q}t}>0,  \mbox{ for any }k\in \mathbb{S}^2.\nonumber
\end{align*}
\end{enumerate}
\end{theorem}
%%%%%%%%%%%%%%%%%%%%%%%%%%%%%%%%%%%%%%%%%%%%%%%%%%%%%%%%%%%%%%%%%%%%%%%%%%%%%%%%%%%%%%%%%%%%%%%%%%%%%%%%%%%
%
%
% Main theorem.
%
%
%%%%%%%%%%%%%%%%%%%%%%%%%%%%%%%%%%%%%%%%%%%%%%%%%%%%%%%%%%%%%%%%%%%%%%%%%%%%%%%%%%%%%%%%%%%%%%%%%%%%%%%%%%%
Now, we state our main result.
\begin{theorem}\label{maintheorem}
Let $-1/2<\nu<1$. Let $f$ be the unique smooth solution of (\ref{ESBGK}) corresponding to a nonnegative initial datum $f_0$ satisfying the hypotheses of
Theorem \ref{Existence_Theorem}. Let $f^n$ be the approximate solution
constructed iteratively by (\ref{main scheme}) given in Section 2.
Then, there exists a positive number $r_{\triangle v}$, which is explicitly determined in Theorem \ref{stability theorem} in Section 4, such that, if $\triangle v<r_{\triangle v}$, then we have
\begin{eqnarray*}
\|f(T^f)-f^{N_t}\|_{L^{\infty}_{q}}\leq C\Big\{(\triangle x)^2+\triangle v+\triangle t+\frac{(\triangle x)^2}{\triangle t}~\Big\},
\end{eqnarray*}
where $N_t$ is defined by $T^f=N_t\triangle t$ and $C=C(T^f,f_0,q,\kappa,\nu)>0$.
Here, $C$ is uniformly bounded in $\nu$.
%\[\sup_{-1/2<\nu<1}C(\nu)<\infty.\] %denotes
%\begin{eqnarray*}C(T^f,f_0,\kappa,q)=~\max\Big\{\kappa^{-1}, e^{\frac{C T^f}{\kappa}},\|f_0\|_{W^{1,\infty}_q}\Big\}.\cr\end{eqnarray*}
\end{theorem}

\begin{remark} (1) $\nu=0$ corresponds to the original BGK model. Therefore, our result holds for the original BGK model too. (2) For the precise definition of $r_{\triangle v}$, see Theorem \ref{stability theorem}.
(3) When $\kappa=0$, this error estimate breaks down, since the coefficients of the estimate contain $\kappa^{-1}$.
Currently, it is not clear whether this is of inherent nature, or can be avoided by developing finer convergence analysis.
(4) The bad term $1/\triangle t$ is removed in \cite{CDM}. But the argument cannot be implemented in our case since it depends heavily on the fact that the distribution
function for the Vlaosv-Poisson equation remains compactly supported, once it is so initially.
(5) We believe the argument we develop in this work is robust, and  can be extended in many directions such as semi-Lagrangian scheme for polyatomic BGK models, high order semi-Lagrangian schemes, and semi-Lagrangian BDF methods, or Runge-Kutta method. We leave them for the future.
\end{remark}
%%%%%%%%%%%%%%%%%%%%%%%%%%%%%%%%%%%%%%%%%%%%%%%%%%%%%%%%%%%%%%%%%%%%%%%%%%%%%%%%%%%%%%%%%%%%
%
%           Technical lemmas
%
%%%%%%%%%%%%%%%%%%%%%%%%%%%%%%%%%%%%%%%%%%%%%%%%%%%%%%%%%%%%%%%%%%%%%%%%%%%%%%%%%%%%%%%%%%%%
\section{Technical lemmas}
\begin{lemma}
Discrete solutions to (\ref{main scheme}) are periodic in the spatial nodes:
\begin{align*}
f^n_{i+N_x,j}=f^n_{i,j}.
\end{align*}
\end{lemma}
\begin{proof}
We use induction. We recall the definition of $f^0_{i,j}$ to get
\begin{align*}
f^0_{i+N_x,j}=f_0(x_{i+{N_x}},v_j)=f_0(x_i+N_x\triangle x,v_j)=f_0(x_i+1,v_j)=f_{0}(x_i,v_j)=f^0_{i,j}.
\end{align*}
Similarly, we have
\begin{align*}
\tilde{f}^0_{i+N_x,j}&=\widetilde{f}_0(x_i+N_x\triangle x,v_j)=\widetilde{f}_0(x_i+1,v_j)\cr
&=f_0(x_i-\triangle t v_j+1,v_j)=f_0(x_i-\triangle t v_j,v_j)\cr
&=\widetilde{f}_0(x_i,v_j)\cr
&=\tilde{f}^0_{i,j}.
\end{align*}
Then, the periodicity of $f^0_{i,j}$ and $\tilde{f}^0_{i,j}$ implies the periodicity of
$\rho^0_{i}, U^0_{i}, \mathcal{T}^0_{\widetilde{\nu},i}$ and $\widetilde{\rho}^0_{i}, \widetilde{U}^0_{i}, \widetilde{\mathcal{T}}^0_{\tilde{\nu},i}$ by definition. This completes the proof of the initial step of the induction.
Now, assume that $f^n_{i,j}$, $\widetilde{f}^n_{i,j}$,
$\rho^n_{i}, U^n_{i}, \mathcal{T}^n_{\widetilde{\nu},i}$ and $\widetilde{\rho}^n_{i}, \widetilde{U}^n_{i}, \widetilde{\mathcal{T}}^n_{\tilde{\nu},i}$ are all periodic in spatial variable. Then the periodicity of $f^{n+1}_{i,j}$ is immediate from (\ref{main scheme}).
For the periodicity of $\widetilde{f}^{n+1}_{i,j}$, we first observe
\begin{align*}
x(i+N_x,j)=x_{i+N_x}-\triangle t v_{j_1}=x_i+N_x\triangle x-\triangle tv_{j_1}=x(i,j)+N_x\triangle x,
\end{align*}
so that
\begin{align*}
s(i+N_x,j)=s(i,j)+N_x.
\end{align*}
Therefore,
\begin{align*}
x(i+N_x,j)-x_{s(i+{N_x},j)}=x(i,j)-x_{s(i,j)}.
\end{align*}
Likewise,
\begin{align*}
x_{s(i+N_x,j)+1}-x(i+N_x,j)=x_{s(i,j)+1}-x(i,j).
\end{align*}
We then use these identities together with the periodicity of $f^{n+1}_{i,j}$ to derive
\begin{align*}
\widetilde f^{n+1}_{i+N_x,j}&=\frac{x(i+N_x,j)-x_{s(i+N_x,j)}}{\triangle x}~f^{n+1}_{s(i+N_x,j)+1,j}+\frac{x_{s(i+N_x,j)+1}-x(i+N_x,j)}{\triangle x}~f^{n+1}_{s(i+N_x,j),j}\cr
&=\frac{x(i,j)-x_{s(i,j)}}{\triangle x}~f^{n+1}_{s(i,j)+N_x+1,j}+\frac{x_{s(i,j)+1}-x(i,j)}{\triangle x}~f^{n+1}_{s(i,j)+N_x,j}\cr
&=\frac{x(i,j)-x_{s(i,j)}}{\triangle x}~f^{n+1}_{s(i,j)+1,j}+\frac{x_{s(i,j)+1}-x(i,j)}{\triangle x}~f^{n+1}_{s(i,j),j}\cr
&=\widetilde f^{n+1}_{i,j},
\end{align*}
which gives the periodicity of $\widetilde f^{n+1}_{i+N_x,j}$. Then the macroscopic fields associated with
$f^{n+1}_{i,j}$ and $\widetilde f^{n+1}_{i,j}$ are periodic by construction. Therefore, the desired result follows from
induction.
\end{proof}
\begin{lemma} \label{f tilde{f}}  \emph{\cite{Yun3}} The reconstruction procedure does not increase the $\|\cdot\|_{L^{\infty}_q}$-norm of the discrete distribution function:
\[
\|\widetilde  f^n\|_{L^{\infty}_q}\leq \|f^n\|_{L^{\infty}_q}.
\]
\end{lemma}
\begin{proof}
We observe from the definition of $\widetilde f^n$ that, for $n\neq0$
\begin{align*}
\|\widetilde f^n\|_{L^{\infty}_q}
&=\sup_{i,j}\big|\widetilde f^n_{i,j}(1+|v_j|)^q\big|\\
&=\sup_{i,j}\Big|\Big(\frac{x(i,j)-x_{s,j}}{\triangle x}~f^{n}_{s+1,j}
+\frac{x_{s+1,j}-x(i,j)}{\triangle x}~f^{n}_{s,j}\Big)(1+|v_j|)^q\Big|~\\
&\leq\sup_{i,j}\big|\max\{f^{n}_{s,j},f^{n}_{s+1,j}\}(1+|v_j|)^q\big|\\
&\leq\sup_{i,j}\big|f^n_{i,j}(1+|v_j|)^q\big|\\
&=\|f^n\|_{L^{\infty}_q}.
\end{align*}
Here, $s$ denotes $s(i,j)$.
When $n=0$, we have
\begin{align*}
\|\widetilde f^0\|_{L^{\infty}_q}
=\sup_{i,j}\big|f_0(x_i-v_j\triangle t,v_j)(1+|v_j|)^q\big|
\leq\sup_{i,j}\big| f_0(x_i,v_j)(1+|v_j|)^q\big|
=\|f^0\|_{L^{\infty}_q}.
\end{align*}
%\begin{align*}a=\frac{x(i+1,j)-x_{s(i+1,j),j}}{\triangle x}=\frac{x(i,j)-x_{s(i,j),j}}{\triangle x},\end{align*}
%which gives\begin{align*}1-a=\frac{x_{s(i+1,j)+1,j}-x(i+1,j)}{\triangle x}=\frac{x_{s(i,j)+1,j}-x(i,j)}{\triangle x}.\end{align*}
\end{proof}
\begin{lemma}\label{equivalence}Let $-1/2<\nu<1$. Assume $\widetilde{f}^n_{i,j}>0$ and $\widetilde{\rho}^n_{i}>0$. Then the discrete temperature tensor $\mathcal{T}^n_{\tilde{\nu},i}$ and its determinant $\det\mathcal{T}^n_{\tilde{\nu},i}$ satisfy
the following equivalence estimates:
%\begin{align*}
%&&(1)~\mathcal{T}^n_{i}\geq C\min\{1-\nu,1+2\nu\}T^n_iId\cr
%&&(2)~\det\{\mathcal{T}^{n}_i\}\geq C\min\{1-\nu,1+2\nu\}^3(T^n_i)^3
%\end{align*}
\begin{align*}
&(1)~\min\{1-\tilde{\nu},1+2\tilde{\nu}\}T^n_iId\leq\widetilde{\mathcal{T}}^n_{\tilde{\nu},i}\leq \max\{1-\tilde{\nu},1+2\tilde{\nu}\}T^n_iId,\cr
&(2)~\min\{1-\tilde{\nu},1+2\tilde{\nu}\}^3(T^n_i)^3\leq \det\{\widetilde{\mathcal{T}}^{n}_{\tilde{\nu},i}\}\leq \max\{1-\tilde{\nu},1+2\tilde{\nu}\}^3(T^n_i)^3.
\end{align*}
In the first inequality, $A\geq B$ for $3\times 3$ symmetric matrices $A$ and $B$ means $A-B$ is positive definite.
\end{lemma}
\begin{proof}
From (\ref{tildeT}), we see that
\begin{align*}
\widetilde{\rho}^n_i\widetilde{\mathcal{T}}^n_{\tilde{\nu},i}&=(1-\tilde{\nu})\widetilde{\rho}^n_i\widetilde{T}^n_iId+\tilde{\nu}\widetilde{\rho}^n_i\widetilde{\Theta}^n_i\cr
&=\frac{(1-\tilde{\nu})}{3}\Big\{\sum_j \tilde{f}^n_{i,j}|v_j-\widetilde{U}^n_i|^2(\triangle v)^3\Big\}+\tilde{\nu}\sum_j \tilde{f}^n_{i,j}
(v_j-\widetilde{U}^n_i)\otimes(v_j-\widetilde{U}^n_i)(\triangle v)^3.
\end{align*}
Then, in view of the identity: $k^T \left\{U\otimes U\right\}k=(k\cdot U)^2~~(k, U\in \mathbb{R}^3)$, we have
\begin{align*}
k^{\top}\{\widetilde{\rho}_i^n\widetilde{\mathcal{T}}^n_{\tilde{\nu},i}\}k=\frac{(1-\tilde{\nu})}{3}\Big\{\sum_j \tilde{f}^n_{i,j}|v_j-\widetilde{U}^n_i|^2(\triangle v)^3\Big\}|k|^2+\tilde{\nu}\sum_j \tilde{f}^n_{i,j}\left\{(v_j-\widetilde{U}^n_i)\cdot k\right\}^2(\triangle v)^3.
\end{align*}
We note from the definition of $\widetilde{\nu}$ that $-1/2<\nu<1$ implies $-1/2<\widetilde{\nu}<1$, and divide our estimate into the following two cases: \newline
\noindent(a) $0<\tilde{\nu}<1$: Since the second term is non-negative, we have
\begin{align*}
k^{\top}\{\widetilde{\rho}^n_i\widetilde{\mathcal{T}}_{\tilde{\nu},i}^n\}k\geq\frac{(1-\tilde{\nu})}{3}\Big\{\sum_j \tilde{f}^n_{i,j}|v_j-\widetilde{U}^n_i|^2(\triangle v)^3\Big\}|k|^2
=(1-\tilde{\nu})\widetilde{\rho}^n_i \widetilde{T}^n_i|k|^2.
\end{align*}
(b) $-\frac{1}{2}<\tilde{\nu}<0$: By Cauchy-Schwartz inequality, we see that
\begin{align*}
k^{\top}\{\widetilde{\rho}^n_i\widetilde{\mathcal{T}}^n_{\tilde{\nu},i}\}k&\geq\frac{(1-\tilde{\nu})}{3}
\Big\{\sum_j \tilde{f}^n_{i,j}|v_j-\widetilde{U}^n_i|^2(\triangle v)^3\Big\}|k|^2+\tilde{\nu}\Big\{\sum_j\tilde{f}^n_{i,j}|v_j-\widetilde{U}^n_i|^2(\triangle v)^3\Big\}|k|^2\cr
&=\frac{(1+2\tilde{\nu})}{3}\Big\{\sum_j\tilde{f}^n_{i,j}|v_j-\widetilde{U}^n_i|^2(\triangle v)^3\Big\}|k|^2\cr
&=(1+2\tilde{\nu})\widetilde{\rho}^n_i \widetilde{T}^n_i |k|^2.
\end{align*}
We then combine the above estimates to get
\begin{align*}
\min\{1-\tilde{\nu},1+2\tilde{\nu}\} \widetilde{T}^n_i\leq k^{\top}\{\widetilde{\mathcal{T}}^n_{\tilde{\nu},i}\}k.
\end{align*}
The r.h.s of the inequality follows in a similar manner. \newline

(2) Let $\lambda_i$ be the eigenvalues of $\widetilde{\mathcal{T}}_{\tilde{\nu}}$. Then (1) implies that the values of
these eigenvalues lie between $\min\{1-\tilde{\nu},1+2\tilde{\nu}\}\widetilde{T}^n_i$ and $\max\{1-\tilde{\nu},1+2\tilde{\nu}\}\widetilde{T}^n_i$.
This gives
\begin{align*}
\min\{\lambda^n_1,\lambda^n_2,\lambda^n_3\}=\min_{|k|=1}k^{\top}\widetilde{\mathcal{T}}^n_{\tilde{\nu},i}k\geq \min\{1-\tilde{\nu},1+2\tilde{\nu}\} \widetilde{T}^n_i,
\end{align*}
and
\begin{align*}
\max\{\lambda^n_1,\lambda^n_2,\lambda^n_3\}=\max_{|k|=1}k^{\top}\widetilde{\mathcal{T}}^n_{\tilde{\nu},i}k\leq \max\{1-\tilde{\nu},1+2\tilde{\nu}\} \widetilde{T}^n_i,
\end{align*}
so that
\begin{align*}
\min\{1-\nu,1+2\nu\}^3\{\widetilde{T}^n_i\}^3\leq\det\widetilde{\mathcal{T}}^n_{\widetilde{\nu},i}=\lambda^n_1\lambda^n_2\lambda^n_3\leq \max\{1-\nu,1+2\nu\}^3\{\widetilde{T}^n_i\}^3.
\end{align*}
\end{proof}

In the following lemma, the symbol $\left[x\right]$ denotes, as usual, the largest integer that does not exceed $x$.
%%%%%%%%%%%%%%%%%%%%%%%%%%%%%%%%%%%%%%%%%%%%%%%%%%%%%%%%%%%%%%%%%%%%%%%%%%%%%%%%%%%%%%%%%%%%
%
%                   Lemma
%
%%%%%%%%%%%%%%%%%%%%%%%%%%%%%%%%%%%%%%%%%%%%%%%%%%%%%%%%%%%%%%%%%%%%%%%%%%%%%%%%%%%%%%%%%%%%\
\begin{lemma}\label{a}
Fix velocity grid index $j_1$ and the grid size $\triangle x$, $\triangle v$, $\triangle t$. We define $s^k$
inductively as
\begin{align*}
s^{(1)}=s(i,j_1), ~s^{(2)}=s(s^{(1)},j_1), ~s^{(3)}=s(s^{(2)},j_1),\cdots
\end{align*}
Using this notation, we define $a^{s^{(n)}}_{j_1}$ by
\[
a^{s^{(n)}}_{j_1}=\frac{x_{s^{(n)}}-\triangle tv_{j_1}-x_{s^{(n+1)}}}{\triangle x}.
\]
Then, $a^{s^{(n)}}_{j_1}$ is constant for all $n>0$, that is
\[
a^{s^{(n)}}_{j_1}=a^{s^{(m)}}_{j_1},
\]
for all positive integers $m$ and $n$.
\end{lemma}
\begin{proof}
We take a positive integer $\ell_{s^{(n)}}$ such that
\begin{align*}
x_{s^{(n)}}=\ell_{s^{(n)}}\triangle x.
\end{align*}
On the other hand, we can find a positive integer $m_j$ such that
\begin{align*}
m_{j_1}\triangle x\leq v_{j_1}\triangle t\leq (m_{j_1}+1)\triangle x.
\end{align*}
That is,
\begin{equation}\label{Gauss}
m_{j_1}=\left[\frac{v_{j_1}\triangle t}{\triangle x}\right].
\end{equation}
% Note that $m$ is a fixed, while $\ell_{s^{n}}$ changes its value for each $n$
From these, we immediately see that
 \begin{align*}
\left\{\ell_{s^{(n)}}-(m_{j_1}+1)\right\}\triangle x\leq x_{s^{(n)}}-v_{j_1}\triangle t\leq \left\{\ell_{s^{(n)}}-m_{j_1}\right\}\triangle x.
\end{align*}
Since $x_{s^{(n+1)}}$ denotes the closest spatial node that lies before $x_{s^{(n)}}-v_{j_1}\triangle t$, this gives
\[
x_{s^{(n+1)}}=\{\ell_{s^{(n)}}-(m_{j_1}+1)\}\triangle x.
\]
Therefore,
\begin{align}
\{x_{s^{(n)}}-\triangle tv_{j_1}\}-x_{s^{(n+1)}}&=\{\ell_{s^{(n)}}\triangle x-\triangle t v_{j_1}\}-\{\ell_{s^{(n)}}-(m_{j_1}+1)\}\triangle x\cr
&=(m_{j_1}+1)\triangle x-\triangle t v_{j_1}.
\end{align}
Dividing both sides by $\triangle x$
\begin{align}
a^{s^{(n)}}_{j_1}=\frac{x_{s^{(n)}}-\triangle tv_{j_1}-x_{s^{(n+1)}}}{\triangle x}
=(m_{j_1}+1)-\frac{v_{j_1}\triangle t}{\triangle x}.
\end{align}
In view of (\ref{Gauss}), this can be rewritten as
\[
\left[ \frac{v_{j_1}\triangle t}{\triangle x}\right]-\frac{v_{j_1}\triangle t}{\triangle x}+1,
\]
which is dependent on $j$ but not on $n$. This completes the proof.
\end{proof}
%%%%%%%%%%%%%%%%%%%%%%%%%%%%%%%%%%%%%%%%%%%%%%%%%%%%%%%%%%%%%%%%%%%%%%%%%%%%%%%%%%%%%%%%%%%%%%%%%%%%%%%%%
%
%
%                        Section: Estimates
%
%
%%%%%%%%%%%%%%%%%%%%%%%%%%%%%%%%%%%%%%%%%%%%%%%%%%%%%%%%%%%%%%%%%%%%%%%%%%%%%%%%%%%%%%%%%%%%%%%%%%%%%%%%%
\section{Stability of the discrete distribution function }
In this section, we derive uniform lower and upper bounds for the discrete distribution function $\tilde{f}^n_{i,j}$ and corresponding macroscopic fields.
We start with series of definitions most of which were introduced for technical reasons.
\begin{definition}
(1) We define $C_{\alpha}$, $C_{q,\alpha}$ and $C_{q-m}\,(q\geq m)$ by
\begin{align*}
C_{\alpha}=\int_{\mathbb{R}^3}e^{-C_{0,2}|v|^{\alpha}}dv,\quad
C_{q,\alpha}=\sup_{v}\big\{(1+|v|)^qe^{-C_{0,2}|v|^{\alpha}}\big\},\quad
C_{q-m}=\int_{\mathbb{R}^3}\frac{1}{(1+|v|)^{q-m}}dv.
\end{align*}
(2) Throughout this section, we fix $C_{M}$, $C_{\mathcal{M}}$ as is defined in Lemma \ref{moments estimate discrete} and Lemma \ref{Control Max}
respectively.
\end{definition}
\begin{definition}\label{End}
(1) We say that $f^n_{i,j}$ satisfies $E^n_1$ if  the following two statements hold:
\begin{align*}
&(A^n)~ \|\widetilde{f}^n_{i,j}\|_{L^{\infty}_q}
\leq\Big(\frac{\kappa+C_{\mathcal{M}}A_{\nu}\triangle t }{\kappa+A_{\nu}\triangle t}\Big)^{n}\|f^{0}\|_{L^{\infty}_q}
\leq e^{\frac{(C_{\mathcal{M}}-1) A_{\nu}T^f}{\kappa+A_{\nu}\triangle t}}\|f^0\|_{L^{\infty}_q},\cr
&(B^n)~\widetilde{f}^n_{i,j}\geq C_{0,1}\left(\frac{\kappa}{\kappa+A_{\nu}\triangle t}\right)^ne^{-C_{0,2}|v_j|^{\alpha}}
\geq C_{0,1}e^{-\frac{A_{\nu}}{\kappa}T^f}e^{-C_{0,2}|v_j|^{\alpha}}.
\end{align*}
(2) We say that $f^n_{i,j}$ satisfies $E^n_2$ if the following two statements hold:
\begin{align*}
&(C^n)~\widetilde{\rho}^n_i\geq \frac{1}{2}C_{0,1}C_{\alpha}e^{-\frac{A_{\nu}}{\kappa}T^f},\quad
\widetilde{T}^n_i\geq \left(\frac{C_{0,1}C_{\alpha}}{2C_M\|f_0\|_{L^{\infty}_q}}\right)^{2/3}e^{-\frac{2}{3}\left(\frac{1}{\kappa}+\frac{C_{\mathcal{M}}-1 }{\kappa+A_{\nu}\triangle t}\right)A_{\nu}T^f},\cr
&(D^n)~\|\widetilde{\rho}^n\|_{L^{\infty}_x},\|\widetilde{U}^n\|_{L^{\infty}_x},\|\widetilde{T}^n\|_{L^{\infty}_x}\leq 2C_{q}\Big\{1+\big(C_{0,1}C_{\alpha}\big)^{-1}\Big\}e^{\left(\frac{1}{\kappa}+\frac{(C_{\mathcal{M}}-1)}{\kappa+A_{\nu}\triangle t}\right)A_{\nu}T^f}\|f^0\|_{L^{\infty}_q}.
\end{align*}
(3) We define $E^n=E^n_1\wedge E^n_2$.
\end{definition}
\begin{remark}
In fact, the first inequaly in $(A^n)$  implies the second inequality due to the elementary inequality
$(1+x)^n\leq e^{nx}$. We stated them in this seemingly redundant manner since both estimates are interchangeably used in the following proofs. $(B^n)$ is stated in such a redundant manner for the same reason.
\end{remark}

\begin{definition}
(1) We define constants $a_1$, $a_2$ and $a_3$ by
\begin{align*}
a_1&=\left(\frac{1}{2}\right)^{\frac{14}{15}}\left(\frac{3}{\pi}\right)^{\frac{1}{5}}\frac{(C_{0,1}C_{\alpha})^{\frac{1}{3}}}
{C_M^{\frac{2}{15}}\|f_0\|^{\frac{1}{3}}_{L^{\infty}_q}}
e^{-\frac{1}{3}\left(\frac{1}{\kappa}+\frac{(C_{\mathcal{M}}-1) }{\kappa+A_{\nu}\triangle t}\right)A_{\nu}T^f}
,\cr
a_2&=\displaystyle\left(\frac{\pi}{4(q-5)}\right)^{\frac{1}{q-3}}\frac{ (C_{0,1}C_{q,\alpha})^{\frac{1}{q-3}}}
{C_{q}^{\frac{1}{q-3}}\Big\{1+\big(C_{0,1}C_{\alpha}\big)^{-1}\Big\}^{\frac{1}{q-3}}
\|f_0\|^{\frac{1}{q-3}}_{L^{\infty}_q}}e^{-\frac{1}{q-3}\left(\frac{2}{\kappa}+\frac{(C_{\mathcal{M}-1})}{\kappa+A_{\nu}\triangle t}\right)A_{\nu}T^f},\cr
a_3&=\left(\frac{1}{2}\right)^{\frac{2(q+6)}{3(q+3)}}\left(\frac{3^q}{\pi}\right)^{\frac{1}{q+3}}\!\!\!\!\!\!\!\!
\frac{\left(C_{0,1}C_{\alpha}\right)^{\frac{2q+3}{3(q+3)}}}{\big\{C_M\big\}^{\frac{2q}{3(q+3)}}\|f_0\|_{L^{\infty}_q}^{\frac{2q-3}{3(q+3)}}}
e^{-\frac{2q+3}{3(q+3)}\left(\frac{1}{\kappa}+\frac{(C_{\mathcal{M}}-1) }{\kappa+A_{\nu}\triangle t}\right)A_{\nu}T^f}.
\end{align*}
\end{definition}
%%%%%%%%%%%%%%%%%%%%%%%%%%%%%%%%%%%%%%%%%%%%%%%%%%%%%%%%%%%%%%%%%%%%%%%%%%%%%%%%%%%%%%%%%%%%%%%%%%%%%%%%%%%%%%%%%%%%%%%%%%%%%
%
%
%       Main goal.
%
%
%
%%%%%%%%%%%%%%%%%%%%%%%%%%%%%%%%%%%%%%%%%%%%%%%%%%%%%%%%%%%%%%%%%%%%%%%%%%%%%%%%%%%%%%%%%%%%%%%%%%%%%%%%%%%%%%%%%%%%%%%%%%%%%%
The main goal of this section is the following.
\begin{theorem}\label{stability theorem}
Choose $\ell>0$ sufficiently small such that  $\triangle v<\ell$ implies
\begin{align*}
&\hspace{0.3cm}\frac{1}{2}C_{\alpha}\leq\sum_je^{-C_{0,2}|v_j|^{\alpha}}(\triangle v)^3\leq2C_{\alpha},\cr
&\frac{1}{2}C_{q,\alpha}\leq\sup_j\big\{(1+|v_j|)^qe^{-|v_j|C^{\alpha}}\big\}\leq2C_{q,\alpha},\cr
&\hspace{0.1cm}\frac{1}{2}C_{q-m}\leq\sum_j\frac{(\triangle v)^3}{(1+|v_j|)^{q-m}}\leq 2C_{q-m}.
\end{align*}
Now, define $r_{\triangle v}$ by
\[
r_{\triangle v}=\min\{a_1,a_2,a_3,\ell,1/2\},
\]
and suppose $\triangle v$ is sufficiently small in the following sense:
\[
\triangle v<r_{\triangle v}.
\]
Then, $f^n_{i,j}$ satisfies $E^n$ for all $n\geq 0$.
\end{theorem}
We postpone the proof to the end of this section, after establishing several preliminary results. We begin with the discrete moment
estimates:
\begin{lemma}\label{moments estimate discrete} Let $q>5$. Suppose $f^n_{i,j}$ satisfies $E^n$
and $\triangle v$ satisfies the smallness condition stated in Theorem \ref{stability theorem}.
Then there exists a positive constant $C_{M}$
which depends only on $q$, such that
\begin{align}\label{PS1}
\begin{array}{ll}
\displaystyle(1)&\displaystyle \frac{\widetilde{\rho_i}^n}{(\widetilde{T}^n_i)^{\frac{3}{2}}}\leq C_{M} \|f^n\|_{L^{\infty}_q},\\
\displaystyle(2)&\displaystyle\widetilde{\rho}^n_i(\widetilde{T}^n_i+|\widetilde{U}^n_i|^2)^{\frac{q-3}{2}}\leq C_{M} \|f^n\|_{L^{\infty}_q},\\
%\displaystyle(ii)&\displaystyle\frac{\widetilde{\rho}^n}{( \widetilde{T}^n+|\widetilde{U}^n|^2)^{\frac{3-q}{2}}}\leq C_q\|f^n\|_{L^{\infty}_q},
%&~\displaystyle(q<3),\\
\displaystyle(3)&\displaystyle\frac{\widetilde{\rho}^n| \widetilde{U}^n_i|^{q+3}}{[(\widetilde{T}^n_i+|\widetilde{U}^n_i|^2)\widetilde{T}^n_i]^{\frac{3}{2}}}
\leq C_{M}\|f^n\|_{L^{\infty}_q},
\end{array}
\end{align}
for all $n$.
\end{lemma}
\begin{proof}
%%%%%%%%%%%%%%%%%%%%%%%%%%%%%%%%%%%%%%%%%%%%%%%%%%%%%%%%%%%%%%%%%%%%%%%%%%%%%%%%%%%%%%%%%%%%%%%%%%%%%%%%%%%%%%%%%%%%%%%%%%%%%%%%%
%This is the discrete version of the moment estimates in \cite{PP}. Note that the domain decomposition is adjusted to take into consideration the error arising
%in the discretization.
(1) We split the macroscopic density into to the following two parts:
\begin{align*}
\widetilde{\rho}^n_{i}%&=\sum_j \widetilde{f}^n_{i,j}(\triangle v)^2\cr
=\sum_{|v_j-\widetilde{U}^n_i|< r+\triangle v} \widetilde{f}^n_{i,j}(\triangle v)^3
+\sum_{|v_j-\widetilde{U}^n_i|\geq r+\triangle v} \widetilde{f}^n_{i,j}(\triangle v)^3.
%&\equiv&I+II.
\end{align*}
Then we see that
\begin{align*}
\sum_{|v_j-\widetilde{U}^n_i|\leq r+\triangle v} \widetilde{f}^n_{i,j}(\triangle v)^3&\leq \|\widetilde{f}^n\|_{L^{\infty}_0}\!\!\!\!\sum_{|v_j-\widetilde{U}^n_i|\leq r+\triangle v} (\triangle v)^3\cr
%&=C_q\|\widetilde{f}^n\|_{L^{\infty}_0}\int_{|v-\widetilde{U}^n_i|\leq r}dv\cr
&\leq \pi\|\widetilde{f}^n\|_{L^{\infty}_0}(r+2\triangle v)^3\cr
&\leq 8\pi\|\widetilde{f}^n\|_{L^{\infty}_0}(r+\triangle v)^3,
\end{align*}
%where we used
%\[
%\sum_{|v_j-\widetilde{U}^n_i|\leq r+\triangle v}(\triangle v)^3\leq (r+2\triangle v)^3.
%\]
and
\begin{align*}
\sum_{|v_j-\widetilde{U}^n_i|\geq r+\triangle v} \widetilde{f}^n_{i,j}(\triangle v)^3&=\sum_{|v_j-\widetilde{U}^n_i|\geq r+\triangle v} \widetilde{f}^n_{i,j}\frac{|v_j-\widetilde{U}^n_i|^2}{|v_j-\widetilde{U}^n_i|^2}(\triangle v)^3\cr
%&=2\sum_{|v_j|\geq r} \widetilde{f}^n_{i,j}\frac{|v_j-\widetilde{U}^n_i|^2+ (\triangle v)^2}
%{\left\{|v_j-\widetilde{U}^n_i|+\triangle v\right\}^2}(\triangle v)^2\cr
&\leq\frac{1}{(r+\triangle v)^2}\sum_{j} \widetilde{f}^n_{i,j}|v_j-\widetilde{U}^n_i|^2(\triangle v)^3\cr
&=\frac{3\widetilde{\rho}^n_i\widetilde{T}^n_i}{(r+\triangle v)^2},
\end{align*}
yielding
\begin{align*}
\widetilde{\rho}^n_{i}\leq 8\pi\|\widetilde{f}^n\|_{L^{\infty}_q}(r+\triangle v)^3
+\frac{3\widetilde{\rho}^n_i\widetilde{T}^n_i}{(r+\triangle v)^2}.
\end{align*}
We then optimize $r$ by equating the two terms on the right hand sides, to obtain
\begin{eqnarray}\label{find r}
r+\triangle v=\left(\frac{3\widetilde{\rho}^n_i\widetilde{T}^n_i}{8\pi\|\widetilde{f}^n\|_{L^{\infty}_q}}\right)^{1/5}.
\end{eqnarray}
Using the fact that $f^n_{i,j}$ satisfies $E^n$, it can be easily verified that the r.h.s is greater than $a_1$, so that
\begin{eqnarray*}
\left(\frac{3\widetilde{\rho}^n_i\widetilde{T}^n_i}{8\pi\|\widetilde{f}^n\|_{L^{\infty}_q}}\right)^{1/5} \geq a_1>\triangle v.
\end{eqnarray*}
Therefore, we can always find a positive number $r$
satisfying (\ref{find r}). Inserting this, one finds
\begin{align*}
\widetilde{\rho}^n_{i}\leq C_M\{\widetilde{T}^n_i\}^{3/2}\|\widetilde{f}^n\|_{L^{\infty}_q}.
\end{align*}
Since the positivity of $\widetilde{T}^n_{i}$ is guaranteed by $(C^n)$, we can divide both sides by $\big\{\widetilde{T}^n_{i}\big\}^{3/2}$ to get the desired estimate.
\newline
(2) We split the domain into $\{|v_j|>r+2\triangle v\}$ and $\{|v_j|\leq r+2\triangle v\}$:
\begin{align*}
\widetilde{\rho}^n_{i}(3\widetilde{T}^n_i+|\widetilde{U}^n_i|^2)%&=\sum_j\widetilde{f}^n_{i,j}|v_j|^2(\triangle v)^3\cr
&=\sum_{|v_j|>r+2\triangle v}\widetilde{f}^n_{i,j}|v_j|^2(\triangle v)^3
+\sum_{|v_j|\leq r+2\triangle v}\widetilde{f}^n_{i,j}|v_j|^2(\triangle v)^3.
\end{align*}
The first term is bounded by
\begin{align*}
\sum_{|v_j|>r+2\triangle v}\widetilde{f}^n_{i,j}\frac{|v_j|^q}{|v_j|^{q-2}}(\triangle v)^3
&\leq\|\widetilde f^n\|_{L^{\infty}_q}\sum_{|v_j|>r+2\triangle v}\frac{(\triangle v)^3}{|v_j|^{q-2}}\cr
&\leq\|\widetilde f^n\|_{L^{\infty}_q}\int_{|v|>r+\triangle v}\frac{dv}{|v|^{q-2}}\cr
&\leq \frac{4\pi}{q-5}\frac{\|\widetilde{f}^n\|_{L^{\infty}_q}}{(r+\triangle v)^{q-5}}.
%&\leq C\frac{1}{(r+\triangle v)^{q-4}}\|\widetilde{f}^n\|_{L^{\infty}_q}+2\cr
%&\leq C\left\{\widetilde{\rho}^n_i\right\}^{1-\frac{2}{q-3}}\left\{\|f^n\|_{L^{\infty}_q}\right\}^{\frac{2}{q-5}}.
\end{align*}
For the second term, we compute
\begin{align*}
\sum_{|v_j|\leq r+2\triangle v}\widetilde{f}^n_{i,j}|v_j|^2(\triangle v)^3
\leq(r+2\triangle v)^2\sum_{j}\widetilde{f}^n_{i,j}(\triangle v)^3
%&\leq C\widetilde{\rho}^n_i (r+2\triangle v)^2\c
\leq 4\widetilde{\rho}^n_i (r+\triangle v)^2.
\end{align*}
Consequently,
\begin{align*}
\widetilde{\rho}^n_{i}(3\widetilde{T}^n_i+|\widetilde{U}^n_i|^2)
&\leq \frac{4\pi}{q-5}\frac{\|\widetilde{f}^n\|_{L^{\infty}_q}}{(r+\triangle v)^{q-5}}+4\widetilde{\rho}^n_i (r+\triangle v)^2
\leq C\left\{\widetilde{\rho}^n_i\right\}^{1-\frac{2}{q-3}}\left\{\|f^n\|_{L^{\infty}_q}\right\}^{2/(q-3)},
\end{align*}
where the latter inequality follows from optimizing $r$ by taking
%\begin{align*}(r+\triangle v)^{q-3}=C\frac{\|\widetilde{f}^n\|_{L^{\infty}_q}}{\widetilde{\rho}^n_i}\end{align*}or
\begin{align*}
r+\triangle v=\left\{\frac{\pi}{q-5}\frac{\|\widetilde{f}^n\|_{L^{\infty}_q}}{\widetilde{\rho}^n_i}\right\}^{1/(q-3)}.
\end{align*}
The r.h.s is larger than $a_2$ by the fact that $f^n_{i,j}$ satisfies $E^n$,
and we can find the optimizing $r$  by a similar argument as in the previous case.
%\begin{align*}\{\widetilde{\rho}^n_i\}^{1-\frac{2}{q-2}}\|\widetilde{f}^n\|_{L^{\infty}_q}^{\frac{2}{q-2}}\end{align*}
%\begin{align*}C\frac{1}{(r+\triangle v)^{q-4}}\|\widetilde{f}^n\|_{L^{\infty}_q}=2\widetilde{\rho}^n_i (r+\triangle v)^2\end{align*}
%In the last line, we optimized $r$.
\newline
%(3) We compute similarly as in the previous case:
%\begin{align*}\widetilde{\rho}^n_i&=\sum_{|v_j|\leq r}\tilde{f}^n_{i,j}+\sum_{|v_j|>r}\tilde{f}^n_{i,j}\cr
%&\leq \sum_j\frac{|v_j|^q}{|v_j|^q}\widetilde{f}^n_{i,j}(\triangle v)^2+\frac{1}{r^2} \widetilde{\rho}^n_i\left\{ 3\widetilde{T}^n_i+ |\widetilde{U}^n_i|^2\right\}\cr
%&\leq C_qr^{3-q}\|f^n\|_{L^{\infty}_q} \left\{\widetilde{\rho}^n_i (3\widetilde{T}^n_i+|\widetilde{U}^n_i)^2 \right\}^{1-\frac{2}{5-q}}
%\end{align*}
%Optimal $r$ is reached when\begin{align*}
%r^{5-q}=\frac{\widetilde{\rho}^n_i (3\widetilde{T}^n_i+|\widetilde{U}^n_i|^2)}{\|f^n\|_{L^{\infty}_q}}\end{align*}
%which gives the desired result.\newline
(3) We decompose the summational index of  $\widetilde{\rho}^n_i |\widetilde{U}^n_i|$ as follows:
\begin{align*}
\widetilde{\rho}^n_i |\widetilde{U}^n_i|
%&\leq \sum_j\widetilde{f}^n_{i,j}|v_j|(\triangle v)^3\cr
\leq \sum_{|v_j-\widetilde{U}^n_i|\leq r+\triangle v}\widetilde{f}^n_{i,j}|v_j|(\triangle v)^3+ \sum_{|v_j-\widetilde{U}^n_i|> r+\triangle v}\widetilde{f}^n_{i,j}|v_j|(\triangle v)^3
\equiv I+II.
\end{align*}
Then, we apply the H\"{o}lder inequality to bound $I$ as
\begin{align*}
I&\leq\left\{\sum_{|v_j-\widetilde{U}^n_i|\leq r+\triangle v}\widetilde{f}^n_{i,j}(\triangle v)^3\right\}^{1-1/q}
\left\{ \sum_{|v_j-\widetilde{U}^n_i|\leq r+\triangle v}\widetilde{f}^n_{i,j}|v_j|^q(\triangle v)^3\right\}^{1/q}\cr
&\leq \pi^{1/q}\left\{\widetilde{\rho}^n_i\right\}^{1-\frac{1}{q}}\|f\|^{1/q}_{L^{\infty}_q}(r+2\triangle v)^{3/q}\cr
&\leq (8\pi)^{1/q}\left\{\widetilde{\rho}^n_i\right\}^{1-1/q}\|f\|^{1/q}_{L^{\infty}_q}(r+\triangle v)^{3/q}.
\end{align*}
For $II$, we employ the Schwartz inequality to see that
\begin{align*}
II&\leq\frac{1}{r+\triangle v}\sum_{|v_j-\widetilde{U}^n_i|\geq r+2\triangle v}\widetilde{f}^n_{i,j} |v_j-\widetilde{U}^n_i||v_j|(\triangle v)^3\cr
&\leq \frac{1}{r+2\triangle v}\left\{\sum_j\widetilde{f}^n_{i,j}|v_j|^2(\triangle v)^3\right\}^{1/2}\left\{\sum_j\widetilde{f}^n_{i,j}|v_j-U^n_i|^2(\triangle v)^3\right\}^{1/2}\cr
&\leq \frac{1}{r+\triangle v}\left\{\widetilde{\rho}^n_i(3\widetilde{T}^n_i+|\widetilde{U}^n_i|^2)\right\}^{1/2}\left\{3\widetilde{\rho}^n_i\widetilde{T}^n_i\right\}^{1/2}\cr
&= \frac{3^{1/2}\widetilde{\rho}^n_i}{r+\triangle v}\left\{3\widetilde{T}^n_i+|\widetilde{U}^n_i|^2\right\}^{1/2}\left\{\widetilde{T}^n_i\right\}^{1/2}.
\end{align*}
Therefore,
\begin{align*}
\widetilde{\rho}^n_i |\widetilde{U}^n_i|&\leq (8\pi)^{1/q}\left\{\widetilde{\rho}^n_i\right\}^{1-1/q}\|f\|^{\frac{1}{q}}_{L^{\infty}_q}(r+\triangle v)^{3/q}
+\frac{3^{1/2}\widetilde{\rho}^n_i}{r+\triangle v}\left\{3\widetilde{T}^n_i+|\widetilde{U}^n_i|^2\right\}^{1/2}\left\{\widetilde{T}^n_i\right\}^{1/2}.
\end{align*}
We then derive the desired result by optimizing the above estimate by setting $r$ as
\begin{align*}
r+\triangle v
=\left(\frac{3^{1/2}\left\{\widetilde{\rho}^n_i\right\}^{1/q}(3\widetilde{T}^n_i+|\widetilde{U}^n_i|^2)^{1/2}\big\{\widetilde{T}^n_i\big\}^{1/2}}{(8\pi)^{1/q}\|f^n\|^{1/q}_{L^{\infty}_q}}\right)^{q/(q+3)}.
\end{align*}
The fact that $f^n_{i,j}$ satisfies $E^n$ guarantees that the r.h.s is greater than or equal to $a_3$,
%\[\left(\frac{3^{1/2}\left\{\widetilde{\rho}^n_i\right\}^{1/q}(3\widetilde{T}^n_i+|\widetilde{U}^n_i|^2)^{1/2}\big\{\widetilde{T}^n_i\big\}^{1/2}}{(8\pi)^{1/q}\|f^n\|^{1/q}_{L^{\infty}_q}}\right)^{q/(q+3)}\!\!\!\!\!\!\!
%\geq a_3>\triangle v,\]
which guarantees the existence of $r>0$.
\end{proof}

We now show that the ellipsoidal Gaussian is controlled by the discrete distribution in $L^{\infty}_q$.
%%%%%%%%%%%%%%%%%%%%%%%%%%%%%%%%%%%%%%%%%%%%%%%%%%%%%%%%%%%%%%%%%%%%%%%%%%%%%%%%%%%%%%%%%%%%
%
%           Control of Maxwellian by distribution function
%
%%%%%%%%%%%%%%%%%%%%%%%%%%%%%%%%%%%%%%%%%%%%%%%%%%%%%%%%%%%%%%%%%%%%%%%%%%%%%%%%%%%%%%%%%%%%
\begin{lemma}\label{Control Max} Suppose $f^n_{i,j}$ satisfies $E^n$, and $\triangle v<r_{\triangle v}$.
Then we have
\[
\|\mathcal{M}_{\tilde{\nu}}(\tilde{f}^n)\|_{L^{\infty}_q}\leq C_{\mathcal{M}}\|f^n\|_{L^{\infty}_q},
\]
for some constant $C_{\mathcal{M}}$ which depends only on $q$ and $\nu$.
%and
%\[
%N^1_q(\mathcal{M}^n(f^n))\leq C_qN^1_{q}(f^n).
%\]
\end{lemma}
\begin{proof}
%To simplify the presentation, we set\[
%C_{M,\tilde{\nu}}=\max\{1+2\tilde{\nu},1-\tilde{\nu}\},\quad C_{m,\tilde{\nu}}=\min\{1+2\tilde{\nu},1-\tilde{\nu}\}\]
%\begin{align*}\exp\left(-\frac{1}{2}(v_j-U^n_i)^{\top}\{\mathcal{T}^n_i\}^{-1}_{\nu}(v_j-U^n_i)\right)\leq \exp\Big(-\frac{1}{2}\frac{|v-U^n_i|^2}{2C_{M,\tilde{\nu}}\widetilde{T}^n_i}\Big),
%\end{align*}
%Since we are assuming $T^n_i>0$, this implies that
We divide the proof into the two cases: $q=0$ and $q\neq0$.\newline
\noindent (1) $q=0$: Since the exponential part is less than or equal to 1, we see from Lemma \ref{equivalence} that
\begin{align*}
\mathcal{M}_{\tilde{\nu},j}(\tilde{f}^n_i)\leq \frac{\widetilde{\rho}^n_{i}}{\sqrt{\det(2\pi\widetilde{\mathcal{T}}^n_{\tilde{\nu},i})}}
\leq C_{\tilde{\nu}}\frac{\widetilde{\rho}^n_i}{\{\widetilde{T}^n_{i}\}^{3/2}}.
%\leq C_{\tilde{\nu}}\frac{\widetilde{\rho}^n_i}{\{\widetilde{T}^n_i\}^{3/2}}.
\end{align*}
Then Lemma \ref{moments estimate discrete} (1) gives the desired estimate.\newline
\noindent (2) $q\neq 0$:  %The estimate of $\mathcal{M}_j(\widetilde{f}^n_i)|v_j|^q$:
We split $v_j$ as:
\begin{align*}
|v_j|^q\mathcal{M}_{\tilde{\nu},j}(\tilde{f}^n_i)&\leq C_q\left\{|\widetilde{U}^n_{i}|^q+|v_j-\widetilde{U}^n_{i}|^q\right\}\mathcal{M}_{\tilde{\nu},j}(\tilde{f}^n_i)\cr
&\equiv I_1+I_2.
%\left\{|U|^2+|v-U|\right\}\frac{\rho}{\sqrt{\det\mathcal{T}}}\exp\left(-(v-U)\mathcal{T}^{-1}(v-U)\right)
\end{align*}
(i) The estimate for $I_1$: We first bound the exponential part by 1 to get
\begin{align*}
I_1%&=C_q|U|^q\frac{\rho}{\sqrt{\det\mathcal{T}}}\exp\left(-(v-U)\mathcal{T}^{-1}(v-U)\right)\cr
\leq C_{\tilde{\nu}}\frac{|\widetilde{U}^n_i|^q\widetilde{\rho}^n_i}{\{\widetilde{T}^n_i\}^{3/2}}.
%\leq C_{\tilde{\nu}}\frac{2|\widetilde{U}^n_i|^q\widetilde{\rho}^n_i}{\{\widetilde{T}^n_i\}^{3/2}}.
%\left\{|U|^2+|v-U|\right\}\frac{\rho}{\sqrt{\det\mathcal{T}}}\exp\left(-(v-U)\mathcal{T}^{-1}(v-U)\right)
\end{align*}
(a) $|\widetilde{U}^n_i|> \{\widetilde{T}^n_i\}^{\frac{1}{2}}$: Lemma \ref{moments estimate discrete} (3) gives:
\begin{align*}
I_1\leq C\frac{|\widetilde{U}^n_i|^{q}\widetilde{\rho}^n_i}{\{\widetilde{T}^n_i\}^{\frac{3}{2}}}\leq C\frac{|\widetilde{U}^n_i|^{q+3}\widetilde{\rho}^n_i}{|\widetilde{U}^n_i|^3\{\widetilde{T}^n_i\}^{\frac{3}{2}}}
\leq C\frac{|\widetilde{U}^n_i|^{q+3}\widetilde{\rho}^n_i}{\big\{\widetilde{T}^n_i+|\widetilde{U}^n_i|^2\big\}^{\frac{3}{2}}\{\widetilde{T}^n_i\}^{\frac{3}{2}}}
\leq C_{\nu,q}\|\tilde{f}^n\|_{L^{\infty}_q}.
\end{align*}
(b) $|\widetilde{U}^n_i|\leq \{\widetilde{T}^n_i\}^{\frac{1}{2}}$: In this case, we employ Lemma \ref{moments estimate discrete} (2) as
\begin{align*}
I_1\leq \frac{\{\widetilde{T}^n_i\}^{\frac{q}{2}}\widetilde{\rho}^n_i}{\{\widetilde{T}^n_i\}^{3/2}}\leq
\widetilde{\rho}^n_i \{\widetilde{T}^n_i\}^{\frac{q-3}{2}}\leq \widetilde{\rho}^n_i\big\{\widetilde{T}^n_i+|\widetilde{U}^n_i|^2\big\}^{\frac{q-3}{2}}\leq C_{q}
\|\tilde{f}^n\|_{L^{\infty}_q}.
\end{align*}
(ii) The estimate for $I_2$: By Lemma \ref{equivalence}, we have
\begin{align*}
I_2%|v-\widetilde{U}^n_i|^q\frac{\widetilde{\rho}^n_i}{\sqrt{\det(\widetilde{T}^n_i)}}\exp\left(-\frac{1}{2}(v-\widetilde{U}^n_i)^{\top}
%\{\widetilde{\mathcal{T}}^n_i\}^{-1}(v-\widetilde{U}^n_i)\right)\cr
&\leq C_{\tilde{\nu}}|v_j-\widetilde{U}^n_i|^q\frac{\widetilde{\rho}^n_i}{\{\widetilde{T}^n_i\}^{3/2}}\exp\left(-C_{\nu}\frac{|v_j-\widetilde{U}^n_i|^2}{\widetilde{T}^n_i}\right)\cr
&= C_{\tilde{\nu}}\frac{\widetilde{\rho}^n_i}{\{\widetilde{T}^n_i\}^{3/2}}\{\widetilde{T}^n_i\}^{\frac{q}{2}}
\left\{\left(\frac{|v_j-\widetilde{U}^n_i|^2}{\widetilde{T}^n_i}\right)^{\frac{q}{2}}\exp\left(-C_{\nu}\frac{|v_j-\widetilde{U}^n_i|^2}{\widetilde{T}^n_i}\right)\right\}\cr
&\leq C_{\tilde{\nu}}\widetilde{\rho}^n_i \{\widetilde{T}^n_i\}^{\frac{q-3}{2}}.
\end{align*}
In the last line, we used the elementary inequality $|x^ae^{-bx}|\leq C_{a,b}$ for some positive $C_{a,b}$ ($a,b,x>0$).
%If $T^{\frac{1}{2}}>|U|$,
Then, Lemma \ref{moments estimate discrete} (2) gives
\begin{align*}
\widetilde{\rho}^n_i \{\widetilde{T}^n_i\}^{\frac{q-3}{2}}\leq \widetilde{\rho}^n_i \big\{\widetilde{T}^n_i+|\widetilde{U}^n_i|^2\big\}^{\frac{q-3}{2}}\leq C_q\|\widetilde{f}^n\|_{L^{\infty}_q}.
\end{align*}
This completes the proof.
\end{proof}

%
%%%%%%%%%%%%%%%%%%%%%%%%%%%%%%%%%%%%%%%%%%%%%%%%%%%%%%%%%%%%%%%%%%%%%%%%%%%%%%%%%%%%%%%%%%%%
%%%%%%%%%%%%%%%%%%%%%%%%%%%%%%%%%%%%%%%%%%%%%%%%%%%%%%%%%%%%%%%%%%%%%%%%%%%%%%%%%%%%%%%%%%%%%%%%%%%%%%%%%%%%%%%%%%%%%%%%%%%%%%%%%%%%%%%%
%
%
%
%
%
%%%%%%%%%%%%%%%%%%%%%%%%%%%%%%%%%%%%%%%%%%%%%%%%%%%%%%%%%%%%%%%%%%%%%%%%%%%%%%%%%%%%%%%%%%%%%%%%%%%%%%%%%%%%%%%%%%%%%%%%%%%%%%%%%%%%%%%%%%5
%%%%%%%%%%%%%%%%%%%%%%%%%%%%%%%%%%%%%%%%%%%%%%%%%%%%%%%%%%%%%%%%%%%%%%%%%%%%%%%%%%%%%%%%%%%%5
%
%
%
%                      STability of scheme
%
%
%
%%%%%%%%%%%%%%%%%%%%%%%%%%%%%%%%%%%%%%%%%%%%%%%%%%%%%%%%%%%%%%%%%%%%%%%%%%%%%%%%%%%%%%%%%%%5%
\begin{lemma}\label{strict positiveness}%Suppose that .
Let $\triangle v<r_{\triangle v}$. Assume that $f_0$ satisfies the assumptions of Theorem \ref{Existence_Theorem}. Then $f_0$ satisfies $E^0$.
%where\[\displaystyle A_{\kappa}\equiv\frac{\ln\Big(\min\big\{\frac{1}{2},\kappa\big\}\Big)}{1-\min\big\{\frac{1}{2}, \kappa\big\}}<0,\]
\end{lemma}
\begin{proof}
$\bullet$ $(A^0)$~ Thanks to the assumption on the initial data, we have $\|f_0\|_{L^{\infty}_q}<\infty$. Therefore, in view of Lemma \ref{f tilde{f}}, we have
\[
\|\tilde{f}^0\|_{L^{\infty}_q}\leq\|f^0\|_{L^{\infty}_q}\leq \|f_0\|_{L^{\infty}_q}<\infty.
\]
%and $\widetilde{\rho}^0_i$, $\widetilde{\rho}^0_i\widetilde{U}^0_i$ and
%$3\widetilde{\rho}^0_i\widetilde{T}^0_i+\widetilde{\rho}^0_i|\widetilde{U}^0_i|^2$ are all well-defined.\newline

\noindent $\bullet$ $(B^0)$~ We recall  the lower bound assumption imposed on the initial data in Theorem \ref{Existence_Theorem} to see that
\[
\widetilde{f}^0_{i,j}=f_0(x_i-v_j\triangle t,v_j)\geq C_{0,1}e^{-C_{0,2}|v_j|^{\alpha}}.
\]
\noindent$\bullet$ $(C^0)$~ Using the lower bound on the initial data again, one finds
\begin{align*}
\widetilde{\rho}^0_i&=\int_{\mathbb{R}^3}f_0(x_i-v\triangle t,v)dv
\geq C_{0,1}\int_{\mathbb{R}^3}e^{-C_{0,2}|v|^{\alpha}}dv
= C_{0,1}C_{\alpha}>0.
\end{align*}
%Note that this strict positivity ensures that $\widetilde{U}^0_i$ and $\widetilde{T}^0_i$ are well-defined.
For the upper bound for $\widetilde{\rho}^0_{i}$, we decompose the integral domain as
\begin{align*}
\widetilde{\rho}^0_{i}&=\int_{\mathbb{R}^3}f_0(x_i-v\triangle t, v)dv\cr
&\leq \int_{|v-\widetilde{U}^0_i|\leq r}f_0(x_i-v\triangle t, v)dv+\int_{|v-\widetilde{U}^0_i|>r}f_0(x_i-v\triangle t, v)dv\cr
&\leq \frac{4\pi}{3}\|f_0\|_{L^{\infty}_q}r^3+\frac{3}{r^2}\widetilde{\rho}^0_i\widetilde{T}^0_i
\end{align*}
and optimize $r$ with
\begin{align*}
r=\left(\frac{9\widetilde{\rho}^0_i\widetilde{T}^0_i}{4\pi\|f_0\|_{L^{\infty}_q}}\right)^{1/5}
\end{align*}
to get
\begin{align*}
\widetilde{T}^0_i&\geq \left(\frac{4\pi}{9}\right)^{1/3}\left(\frac{\widetilde{\rho}^0_i}{\|f_0\|_{L^{\infty}_q}}\right)^{2/3}
\geq \left(\frac{4\pi}{9}\right)^{1/3}\left(\frac{C_{0,1}C_{\alpha}}{\|f_0\|_{L^{\infty}_q}}\right)^{2/3}.
\end{align*}
If necessary, we can replace $C_M$ in Lemma 4.4 by $\max\left\{C_M, \frac{3}{4\sqrt{\pi}}\right\}$ to get the desired result.\newline
%\begin{align*}
%\widetilde{T}^0_i&\geq \left(\frac{C_{0,1}C_{\alpha}}{2C_M\|f_0\|_{L^{\infty}_q}}\right)^{2/3}.
%\end{align*}
%Note in this case that we are not using Lemma \ref{moments estimate discrete}, which need to use the employ the estimates in $E^0$. Therefore we are safe
%from circular argument.\newline
%Now the pointwise upper bound estimate of $U$ and $T$ follows directly
%%%%%%%%%%%%%%%%%%%%%%%%%%%%%%%%%%%%%%%%%%%%%%%%%%%%%%%%%%%%%%%%%%%%%%%%%%%%%%%%%%%%%%%%%%%%
%
%                   Lemma
%
%%%%%%%%%%%%%%%%%%%%%%%%%%%%%%%%%%%%%%%%%%%%%%%%%%%%%%%%%%%%%%%%%%%%%%%%%%%%%%%%%%%%%%%%%%%%
\noindent $\bullet$  $(D^0)$~ Lemma \ref{f stability} gives
\begin{align*}
\widetilde{\rho}^0_{i}=\int_{\mathbb{R}^3} f_0(x_i-v \triangle t,v)dv \leq \|f_0\|_{L^{\infty}_q}\int_{\mathbb{R}^3}\frac{1}{(1+|v|)^q}
= C_q\|f_0\|_{L^{\infty}_q}.
\end{align*}
The estimate for $\widetilde{U}^0_i$ follows from
\begin{align*}
\big|\widetilde{U}^0_i\big|%&=\Big|\frac{1}{{\widetilde{\rho}^0_i}}\sum_j \tilde{f}^0_{i,j}v_j(\triangle v)^3\Big|\cr
%&\leq \{C_{\alpha}C_{0,1}\}^{-1}\|f^0\|_{L^{\infty}_q}\sum_j\frac{(\triangle v)^3}{(1+|v_j|)^{q-1}}\cr
&=\Big|\frac{1}{{\widetilde{\rho}^0_i}}\int_{\mathbb{R}^3} f_0(x_i-v \triangle t,v)vdv\Big|\cr
&\leq \{C_{\alpha}C_{0,1}\}^{-1}\|f_0\|_{L^{\infty}_q}\int_{\mathbb{R}^3}\frac{1}{(1+|v|)^{q-1}}dv\cr
&= C_{q-1}\{C_{\alpha}C_{0,1}\}^{-1}\|f_0\|_{L^{\infty}_q}.
\end{align*}
For the estimate of $\widetilde{T}^n_i$, we compute
\begin{align*}
3\widetilde{T}^0_i
%&\leq\frac{1}{\widetilde{\rho}^0_{i}}\sum \tilde{f}^0_{i,j}|v^j-\widetilde{U}^n_{i}|^2(\triangle v)^3\cr
%&= \frac{1}{\widetilde{\rho}^0_{i}}\sum \tilde{f}^0_{i,j}|v^j|^2(\triangle v)^3-|\widetilde{U}^0_{i}|^2\cr
%&\leq \frac{1}{\widetilde{\rho}^0_{i}}\sum \tilde{f}^0_{i,j}|v^j|^2(\triangle v)^3\cr
%&\leq \frac{1}{\widetilde{\rho}^0_{i}}\|f^0\|_{L^{\infty}_q}\sum_j\frac{(\triangle v)^3}{(1+|v^j|)^{q-2}}\cr
&= \frac{1}{\widetilde{\rho}^0_{i}}\int_{\mathbb{R}^3}f_0(x_i-v\triangle t,v)|v|^2dv-|\widetilde{U}^0_{i}|^2\cr
&\leq \frac{1}{\widetilde{\rho}^0_{i}}\int_{\mathbb{R}^3}f_0(x_i-v\triangle t,v)|v|^2dv\cr
&\leq \frac{1}{\widetilde{\rho}^0_{i}}\|f_0\|_{L^{\infty}_q}\int_{\mathbb{R}^3}\frac{1}{(1+|v|)^{q-2}}\cr
&= C_{q-2}\{C_{0,1}C_{\alpha}\}^{-1}\|f_0\|_{L^{\infty}_q}.
\end{align*}
This completes the proof for $E^0$.
%The desired estimate for $\widetilde{\mathcal{T}}^n_i$ then follows from Lemma \ref{equivalence}.
\end{proof}

%%%%%%%%%%%%%%%%%%%%%%%%%%%%%%%%%%%%%%%%%%%%%%%%%%%%%%%%%%%%%%%%%%%%%%%%%%%%%%%%%%%%%%%%%%%%%%%%%%%%%%%%%%%%%%%%%%%%%%
%
%
%
%
%%%%%%%%%%%%%%%%%%%%%%%%%%%%%%%%%%%%%%%%%%%%%%%%%%%%%%%%%%%%%%%%%%%%%%%%%%%%%%%%%%%%%%%%%%%%%%%%%%%%%%%%%%%%%%%%%%%%%%%
\begin{lemma}\label{LB f prop}%Suppose that .
Assume $f^{n-1}_{i,j}$ satisfies $E^{n-1}$. Then, $f^{n}_{i,j}$ satisfies $B^n$:
\begin{align*}%\label{LB f rho}
\widetilde{f}^{n}_{i,j}\geq C_{0,1}\left(\frac{\kappa}{\kappa+A_{\nu}\triangle t}\right)^ne^{-C_{0,1}|v_j|^{\alpha}}
\geq C_{0,1}e^{-\frac{A_{\nu}}{\kappa}T^f}e^{-C_{0,2}|v_j|^{\alpha}},
\end{align*}
for all $i,j$. From this, we also have
\begin{align*}%\label{LB f rho}
\|\widetilde{f}^{n}\|_{L^{\infty}_q}\geq
 \frac{1}{2}C_{0,1}C_{q,\alpha}e^{-\frac{A_{\nu}}{\kappa}T^f}.
\end{align*}
%where\[\displaystyle A_{\kappa}\equiv\frac{\ln\Big(\min\big\{\frac{1}{2},\kappa\big\}\Big)}{1-\min\big\{\frac{1}{2}, \kappa\big\}}<0,\]
\end{lemma}
\begin{proof}
Since $f^{n-1}_{i,j}$ satisfies $E^{n-1}$, $\mathcal{M}_{\widetilde{\nu},j}(\widetilde{f}^n_i)$ is strictly positive. Therefore, we have
from (\ref{main scheme})
\begin{align*}
f^{n}_{i,j}&\geq\frac{\kappa}{\kappa+A_{\nu}\triangle t}\widetilde{f}^{n-1}_{i,j},
%&=\frac{\kappa}{\kappa+A_{\nu}\triangle t}\left\{a_{j_1}f^{n-1}_{s^{1}+1,j}+(1-a_{j_1})f^{n-1}_{s^{(1)},j}\right\}.
\end{align*}
which, in view of the definition of $B^{n-1}$, gives
\begin{align}\label{same reasoning}
f^{n}_{i,j}
\geq C_{0,1}\left(\frac{\kappa}{\kappa+A_{\nu}\triangle t}\right)\left(\frac{\kappa}{\kappa+A_{\nu}\triangle t}\right)^{n-1}e^{-C_{0}^{2}|v_j|^{\alpha}}
= C_{0,1}\left(\frac{\kappa}{\kappa+A_{\nu}\triangle t}\right)e^{-C_{0}^{2}|v_j|^{\alpha}}.
\end{align}

This immediately leads to the same lower bound estimate for $\tilde{f}^n_{i,j}$:
\begin{align*}
\tilde{f}^{n}_{i,j}=a_{j_1}f^n_{s(i,j),j}+(1-a_{j_1})f^n_{s(i,j)+1,j}
%&\geq a_{j_1}C_{0,1}\left(\frac{\kappa}{\kappa+A_{\nu}\triangle t}\right)^ne^{-C_{0,2}|v_j|^{\alpha}}+(1-a_{j_1})C_{0,1}\left(\frac{\kappa}{\kappa+A_{\nu}\triangle t}\right)^ne^{-C_{0,2}|v_j|^{\alpha}}\cr
\geq C_{0,1}\left(\frac{\kappa}{\kappa+A_{\nu}\triangle t}\right)^ne^{-C_{0,2}|v_j|^{\alpha}}
\end{align*}
with $a_{j_1}=a^{s^{(n)}}_{j_1}$. We suppressed the dependence on $n$, which is justified by Lemma \ref{a}.
Then we employ the following elementary inequality
$(1+x)^{n}\leq e^{nx}, (x\geq 0)$
to derive
\begin{eqnarray*}
\left(\frac{\kappa}{\kappa+A_{\nu}\triangle t}\right)^n=\left(1+A_{\nu}\frac{\triangle t}{\kappa}\right)^{-n}\geq e^{-nA_{\nu}\frac{\triangle t}{\kappa}}\geq e^{-\frac{A_{\nu}}{\kappa}T^f},
\end{eqnarray*}
where we used $n\triangle t\leq N_t\triangle t=T^f$.
%\[(1-x)^n\geq e^{n\left(\frac{\ln\varepsilon}{1-\varepsilon}\right)x}\quad(0<x<1-\varepsilon)~\mbox{ for } 0<\varepsilon<1,\]
%to see\begin{align}\Big(\frac{\kappa}{\kappa+\triangle t}\Big)^n=\Big(1-\frac{\triangle t}{\kappa+\triangle t}\Big)^n
%\geq e^{\frac{A_{\kappa}n\triangle t}{\kappa+\triangle t}}\geq e^{\frac{A_{\kappa}T^f}{\kappa+\triangle t}},\nonumber\end{align}
%where $A_{\kappa}$ is defined by\[
%\displaystyle A_{\kappa}=\frac{\ln\Big(\min\big\{\frac{1}{2},\kappa\big\}\Big)}{1-\min\big\{\frac{1}{2}, \kappa\big\}}<0,\]
% and we used the fact that the boundedness assumption on $\triangle t$ implies\[
%\frac{\triangle t}{\kappa+\triangle t}<1-\min\Big\{\frac{1}{2},\kappa\Big\}.\]
%We then choose $\varepsilon=2\kappa$ if $\kappa<\frac{1}{2}$ and $\varepsilon=\frac{1}{2}$ if $\kappa\geq\frac{1}{2}$.
The second estimate follows directly from this:
\[
\|\widetilde{f}^n_{i,j}\|_{L^{\infty}_q}\geq C_{0,1}e^{-\frac{A_{\nu}}{\kappa}T^f}\sup_j\big\{(1+|v_j|)^qe^{-C_{0,2}|v_j|^{\alpha}}\big\}
\geq\frac{1}{2}C_{0,1}C_{q,\alpha}e^{-\frac{A_{\nu}}{\kappa}T^f}
\]
This completes the proof.
%\begin{align}\widetilde{f}^n_{i,j}\geq C_1e^{-\frac{T^f}{\kappa}}e^{-C_2|v_j|^2}.\end{align}
%This completes the proof.
\end{proof}
%%%%%%%%%%%%%%%%%%%%%%%%%%%%%%%%%%%%%%%%%%%%%%%%%%%%%%%%%%%%%%%%%%%%%%%%%%%%%%%%%%%%%%%%%%%%
%
%            Main Lemma
%
%%%%%%%%%%%%%%%%%%%%%%%%%%%%%%%%%%%%%%%%%%%%%%%%%%%%%%%%%%%%%%%%%%%%%%%%%%%%%%%%%%%%%%%%%%%%
\begin{lemma}\label{f stability}
Assume $f^{n-1}_{i,j}$ satisfies $E^{n-1}$. Then, $f^{n}_{i,j}$ satisfies $A^n$:
\begin{align*}
\|f^{n}\|_{L^{\infty}_q}\leq\Big(\frac{\kappa+C_{\mathcal{M}}A_{\nu}\triangle t }{\kappa+A_{\nu}\triangle t}\Big)^{n}\|f^{0}\|_{L^{\infty}_q}
\leq e^{\frac{(C_{\mathcal{M}}-1)A_{\nu} T^f}{\kappa+A_{\nu}\triangle t}}\|f^0\|_{L^{\infty}_q}.
\end{align*}
\end{lemma}
\begin{proof}
Applying  Lemma \ref{Control Max} and to (\ref{main scheme}), one finds
\begin{align*}
\|f^{n}\|_{L^{\infty}_q}&\leq\frac{\kappa}{\kappa+A_{\nu}\triangle t}\|\widetilde{f}^{n-1}\|_{L^{\infty}_q}
+\frac{A_{\nu}\triangle t}{\kappa+A_{\nu}\triangle t}\|{\mathcal M}_{\tilde{\nu}}(\widetilde{f}^{n-1})\|_{L^{\infty}_q}\\
&\leq\frac{\kappa}{\kappa+A_{\nu}\triangle t}\|\widetilde{f}^{n-1}\|_{L^{\infty}_q}
+\frac{A_{\nu}\triangle t}{\kappa+A_{\nu}\triangle t}C_{\mathcal{M}}\|\widetilde{f}^{n-1}\|_{L^{\infty}_q}\\
&\leq\frac{\kappa+C_{\mathcal{M}}A_{\nu}\triangle t }{\kappa+A_{\nu}\triangle t}\|\widetilde{f}^{n-1}\|_{L^{\infty}_q}.
%&=\Big(1+\frac{(C_q-1)A_{\nu}\triangle t }{\kappa+A_{\nu}\triangle t}\Big)\|f^{n-1}\|_{L^{\infty}_q}\cr
\end{align*}
We then recall $A_{n-1}$ to bound this further by
\begin{align*}
\Big(\frac{\kappa+C_{\mathcal{M}}A_{\nu}\triangle t }{\kappa+A_{\nu}\triangle t}\Big)\Big(\frac{\kappa+C_{\mathcal{M}}A_{\nu}\triangle t }{\kappa+A_{\nu}\triangle t}\Big)^{n-1}\|f^{0}\|_{L^{\infty}_q}
\leq\Big(\frac{\kappa+C_{\mathcal{M}}A_{\nu}\triangle t }{\kappa+A_{\nu}\triangle t}\Big)^{n}\|f^{0}\|_{L^{\infty}_q}.
\end{align*}
The second estimate follows from
\begin{align*}
\Big(\frac{\kappa+C_{\mathcal{M}}A_{\nu}\triangle t }{\kappa+A_{\nu}\triangle t}\Big)^{n}\leq \Big(1+\frac{(C_{\mathcal{M}}-1)A_{\nu}\triangle t }{\kappa+A_{\nu}\triangle t}\Big)^{n}
\leq e^{\frac{(C_{\mathcal{M}}-1)A_{\nu} \,n\triangle t}{\kappa+A_{\nu}\triangle t}}
\leq e^{\frac{(C_{\mathcal{M}}-1)A_{\nu}\,T^f}{\kappa+A_{\nu}\triangle t}},
\end{align*}
where we used $(1+x)^n\leq e^{nx}$ and $n\triangle t\leq N_t\triangle t=T^f$.
\end{proof}

Using this, we can prove the uniform lower bound of the macroscopic fields:
%%%%%%%%%%%%%%%%%%%%%%%%%%%%%%%%%%%%%%%%%%%%%%%%%%%%%%%%%%%%%%%%%%%%%%%%%%%%%%%%%%%%%%%%%%%%
%
%                   Lemma
%
%%%%%%%%%%%%%%%%%%%%%%%%%%%%%%%%%%%%%%%%%%%%%%%%%%%%%%%%%%%%%%%%%%%%%%%%%%%%%%%%%%%%%%%%%%%%
\begin{lemma}\label{LB T Lemma}%Suppose that .
Assume $f^{n}_{i,j}$ satisfies $A^{n} \wedge B^n$. Then, $f^{n}_{i,j}$ satisfies $C^n $:
\begin{align*}
%&\|\rho^n(t)\|_{L^{\infty}_x}+\|U^n(t)\|_{L^{\infty}_x}+\|T^n(t)\|_{L^{\infty}_x}\leq C_qN_q(f^0)e^{C_{\kappa,T^f}T^f},\nonumber\\
%\rho^n(x,t)&\geq C_qe^{-C_{\kappa,T^f}T^f},\label{LB f rho}\\
\widetilde{\rho}^n_i&\geq \frac{1}{2}C_{0,1}C_{\alpha}e^{-\frac{A_{\nu}}{\kappa}T^f},\cr
\widetilde{T}^n_i
&\geq \left(\frac{C_{0,1}C_{\alpha}}{2C_M\|f_0\|_{L^{\infty}_q}}\right)^{2/3}e^{-\frac{2}{3}\left(\frac{1}{\kappa}+\frac{(C_{q,\nu}-1) T^f}{\kappa+A_{\nu}\triangle t}\right)A_{\nu}T^f}.
\end{align*}
Note also that Lemma \ref{equivalence} then immediately yields the lower bound for $\widetilde{\mathcal{T}}^n_{\tilde{\nu}}$:
\begin{align*}
\widetilde{\mathcal{T}}^n_{\tilde{\nu}}&\geq C_{M}C_{0,1}C_{\alpha}\min\{1-\nu,1+2\nu\}e^{-\left(\frac{1}{\kappa}+\frac{(C_{\mathcal{M}}-1) T^f}{\kappa+A_{\nu}\triangle t}\right)A_{\nu}T^f}Id.
\end{align*}
\end{lemma}
\begin{proof}
For lower bound control for the discrete local density, we multiply $\widetilde{f}^n_{i,j}$ by $(\triangle v)^3$ and sum over $j$ to get
\begin{align*}
\widetilde{\rho}^n_{i}&=\sum_j\tilde{f}^n_{i,j}(\triangle v)^3
%&\geq C_{0,1}\left(\frac{\kappa}{\kappa+A_{\nu}\triangle t}\right)^n\sum_je^{-C_{0,2}|v_j|^{\alpha}}(\triangle v)^3\cr
%&\geq C_{0,1}C_{2,\alpha}\left(\frac{\kappa}{\kappa+A_{\nu}\triangle t}\right)^n\cr
\geq C_{0,1}e^{-\frac{A_{\nu}}{\kappa}T^f}\sum_je^{-C_{0,2}|v_j|^{\alpha}}(\triangle v)^3
\geq \frac{1}{2}C_{\alpha}C_{0,1}e^{-\frac{A_{\nu}}{\kappa}T^f}.
\end{align*}
This, together with Lemma \ref{moments estimate discrete} (1) and Lemma \ref{f stability} gives the lower bound for $\widetilde{T}^{n}_i$:
\begin{align*}
\widetilde{T}^{n}_i&\geq\left(\frac{\widetilde{\rho}^n_i}{C_{M}\|\tilde{f}^n\|_{L^{\infty}_q}}\right)^{2/3}
%&\geq \left(\frac{ \frac{1}{2}C_{0,1}C_{\alpha}e^{-\frac{A_{\nu}}{\kappa}T^f}}{C_{M}e^{\frac{(C_{q,\nu}-1) T^f}{\kappa+A_{\nu}\triangle t}}\|f_0\|_{L^{\infty}_q}}\right)^{2/3}\cr
\geq \left(\frac{C_{0,1}C_{\alpha}}{2C_M\|f_0\|_{L^{\infty}_q}}e^{-\left(\frac{1}{\kappa}+\frac{(C_{q,\nu}-1) T^f}{\kappa+A_{\nu}\triangle t}\right)A_{\nu}T^f}\right)^{2/3}.
\end{align*}
Then the lower bound for $\widetilde{\mathcal{T}}^{n}_i$ follows from the equivalence estimate in Lemma \ref{equivalence}.
\end{proof}

%Now the pointwise upper bound estimate of $U$ and $T$ follows directly
%%%%%%%%%%%%%%%%%%%%%%%%%%%%%%%%%%%%%%%%%%%%%%%%%%%%%%%%%%%%%%%%%%%%%%%%%%%%%%%%%%%%%%%%%%%%
%
%                   Lemma
%
%%%%%%%%%%%%%%%%%%%%%%%%%%%%%%%%%%%%%%%%%%%%%%%%%%%%%%%%%%%%%%%%%%%%%%%%%%%%%%%%%%%%%%%%%%%%
\begin{lemma}\label{UB rho Lemma}%Suppose that .
Assume $f^{n}_{i,j}$ satisfies $A^{n} \wedge B^n$. Then, $f^{n}_{i,j}$ satisfies $D^n $:
\begin{align*}
\|\widetilde{\rho}^n\|_{L^{\infty}_x},~ \|\widetilde{U}^n\|_{L^{\infty}_x},~ \|\widetilde{T}^n\|_{L^{\infty}_x}
\leq 2C_{q}\Big\{1+\big(C_{0,1}C_{\alpha}\big)^{-1}\Big\}e^{\left(\frac{1}{\kappa}+\frac{(C_{\mathcal{M}}-1)}{\kappa+A_{\nu}\triangle t}\right)A_{\nu}T^f}\|f^0\|_{L^{\infty}_q}.
\end{align*}
\end{lemma}
\begin{proof}
Lemma \ref{f stability} gives
\begin{align*}
\widetilde{\rho}^n_{i}=\sum_j\widetilde{f}^n_{i,j}(\triangle v)^3 \leq \|f^n\|_{L^{\infty}_q}\sum_{j}\frac{(\triangle v)^3}{(1+|v_{j}|)^q}
\leq 2C_qe^{\frac{(C_{\mathcal{M}}-1) A_{\nu}T^f}{\kappa+A_{\nu}\triangle t}}\|f^0\|_{L^{\infty}_q}.
\end{align*}
We combine this with the lower bound estimates of $\widetilde{\rho}^n_i$ established in Lemma \ref{LB T Lemma} and the upper bound of the discrete solution in Lemma \ref{f stability} to obtain
\begin{align*}
\big|\widetilde{U}^n_i\big|&=\Big|\frac{1}{{\widetilde{\rho}^n_i}}\sum_j \tilde{f}^n_{i,j}v_j(\triangle v)^3\Big|\cr
&\leq \{C_{\alpha}C_{0,1}\}^{-1}e^{\frac{A_{\nu}}{\kappa}T^f}\|f^n\|_{L^{\infty}_q}\sum_j\frac{(\triangle v)^3}{(1+|v_j|)^{q-1}}\cr
&\leq 2C_{q-1}\{C_{\alpha}C_{0,1}\}^{-1}e^{\left(\frac{1}{\kappa}+\frac{(C_{\mathcal{M}}-1)}{\kappa+A_{\nu}\triangle t}\right)A_{\nu}T^f}\|f_0\|_{L^{\infty}_q}.
\end{align*}
The estimate for $\widetilde{T}^n_i$ follows from
\begin{align*}
3\widetilde{T}^n_i
%&\leq\frac{1}{\widetilde{\rho}^n_{i}}\sum \tilde{f}^n_{i,j}|v^j-\widetilde{U}^n_{i}|^2(\triangle v)^3\cr
&= \frac{1}{\widetilde{\rho}^n_{i}}\sum \tilde{f}^n_{i,j}|v_j|^2(\triangle v)^3-|\widetilde{U}^n_{i}|^2\cr
&\leq \frac{1}{\widetilde{\rho}^n_{i}}\sum \tilde{f}^n_{i,j}|v_j|^2(\triangle v)^3\cr
&\leq \frac{1}{\widetilde{\rho}^n_{i}}\|f^n\|_{L^{\infty}_q}\sum_j\frac{(\triangle v)^3}{(1+|v^j|)^{q-2}}\cr
&\leq 2C_{q-2}\{C_{0,1}C_{\alpha}\}^{-1}e^{\left(\frac{1}{\kappa}+\frac{(C_{\mathcal{M}}-1)}{\kappa+A_{\nu}\triangle t}\right)A_{\nu}T^f}\|f^0\|_{L^{\infty}_q}
\end{align*}
by a similar manner.
%The desired estimate for $\widetilde{\mathcal{T}}^n_i$ then follows from Lemma \ref{equivalence}.
\end{proof}
\subsection{Proof of Theorem \ref{stability theorem}}  Due to Lemma \ref{strict positiveness}$, E_0$ holds.
Assume $E^{n-1}$ is satisfied. Then Lemma \ref{LB f prop}, \ref{f stability}, \ref{LB T Lemma} and \ref{UB rho Lemma} respectively show that
$f^n_{i,j}$ satisfies $A^n$, $B^n$,$C^n$, $D^n$, that is $E^n$. Therefore,
we can conclude that $f^n_{i,j}$ satisfies $E^n$ for all $n\geq0$ by induction.
%We summarize what we've prove so far
%%%%%%%%%%%%%%%%%%%%%%%%%%%%%%%%%%%%%%%%%%%%%%%%%%%%%%%%%%%%%%%%%%%%%%%%%%%%%%%%%%%%%%%%%%%%
%
%                   Lemma
%
%%%%%%%%%%%%%%%%%%%%%%%%%%%%%%%%%%%%%%%%%%%%%%%%%%%%%%%%%%%%%%%%%%%%%%%%%%%%%%%%%%%%%%%%%%%%
%\begin{lemma}\label{so far}%Suppose that .the following estimates hold for approximate macroscopic fields:
%\begin{align}&\|\widetilde{\rho}^n_i\|_{L^{\infty}_x}+\|\widetilde{U}^n_i\|_{L^{\infty}_x}+\|\widetilde{T}^n_i\|_{L^{\infty}_x}\leq C_qN_q(f^0)e^{C_{\kappa,T^f}T^f},\nonumber\\
%&\widetilde{\rho}^n_i\geq C_qe^{-C_{\kappa,T^f}T^f},\label{LB f rho}\\&\widetilde{T}^n_i\geq C_{q}e^{-C_{\kappa, T^f}T^f}>0.\nonumber\end{align}\end{lemma}
%%%%%%%%%%%%%%%%%%%%%%%%%%%%%%%%%%%%%%%%%%%%%%%%%%%%%%%%%%%%%%%%%%%%%%%%%%%%%%%%%%%%%%%%%%%%%%%%%%%%%%%%%%%%%%%
%
%
%            Section: Consistency
%
%
%%%%%%%%%%%%%%%%%%%%%%%%%%%%%%%%%%%%%%%%%%%%%%%%%%%%%%%%%%%%%%%%%%%%%%%%%%%%%%%%%%%%%%%%%%%%%%%%%%%%%%%%%%%%%%%
\section{Consistent form}
In this section, we transform the ES-BGK model (\ref{ESBGK}) into a form which is consistent to our scheme (\ref{main scheme}). We use the following notation for continuous solutions:
\[
\widetilde f(x,v,t)=f(x-v_1\triangle t,v,t).
\]
\begin{theorem}
Under the assumption of Theorem \ref{Existence_Theorem}, (\ref{ESBGK}) can be represented in the following form:
\begin{align}\label{Consist3}
f(x,v,t+\triangle t)&=\frac{\kappa}{\kappa+A_{\nu}\triangle t}\widetilde{f}(x,v,t)
+\frac{A_{\nu}\triangle t}{\kappa+A_{\nu}\triangle t}{\mathcal M}_{\tilde{\nu}}(\widetilde{f})(x,v,t)\cr
&+\frac{A_{\nu}}{\kappa+A_{\nu}\triangle t}\big\{R_{1}+R_2\big\},
\end{align}
where
\begin{align*}
R_1&=\int^{t+\triangle t}_t\Big\{{\mathcal M}_{\nu}(f)(x,v,t)-\mathcal{M}_{\widetilde{\nu}}(\tilde{f})(x,v,t)\Big\}ds\cr
&-\int^{t+\triangle t}_t\big\{(t+\triangle t-s)\partial_x\mathcal{M}(x_{\theta_1},v,t_{\theta_1})ds
+(s-t)\partial_t\mathcal{M}_{\nu}(f)(x_{\theta_1},v,t_{\theta_1})\big\}ds,\cr
R_2&=\int^{t+\triangle t}_t(s-t-\triangle t)\big(\mathcal{M}(f)-f\big)(x_{\theta_2},v,t_{\theta_2})ds
\end{align*}
for some $(x_{\theta_i},v,t_{\theta_i})$ $(i=1,2)$.
\end{theorem}
\begin{proof}
%Let $f$ be the smooth solution of obtained in theorem 1.1.
%\begin{equation}
%\frac{\partial f}{\partial t}+v\cdot\nabla f = \frac{1}{\kappa}({\mathcal {M}}f-f).\label{BGK}
%\end{equation}
Along the characteristic line, (\ref{ESBGK}) reads
\begin{align*}
\frac{d f}{d t}(x+v_1t,v,t) = \frac{1}{\kappa}A_{\nu}({\mathcal M}_{\nu}(f)-f)(x+v_1t,v,t).
\end{align*}
%We integrate in time from $t$ to $t+\triangle t$ to obtain
%\begin{equation*}
%f(x+(t+\triangle t)v,v,t+\triangle t)=f(x+tv,v,t)+\frac{1}{\kappa}\int^{t+\triangle t}_{t}({\mathcal M}f-f)(x+vs,v,s)ds,
%\end{equation*}
%or, equivalently,
Integrating  on $[t,t+\triangle t]$, we get
\begin{align*}
f(x,v,t+\triangle t)
&=f(x-\triangle tv_1,v,t)+\frac{1}{\kappa}A_{\nu}
\int^{t+\triangle t}_{t}\!\!({\mathcal M}_{\nu}(f)-f)(x-(t+\triangle t-s)v_1,v,s)ds\cr
&\qquad\equiv f(x-\triangle tv_1,v,t)+ \frac{1}{\kappa}A_{\nu}(I_1-I_2).
\end{align*}
By Taylor's  theorem around  $(x-\triangle t v_1,v,t)$, we see that there exist
$x_{\theta_1}$ which lies between $x$ and $x-(t+\triangle t-s)$ and $t_{\theta_1}$ $\in$ $[s,t]$ such that
\begin{align*}
&{\mathcal M}_{\nu}(f)(x-(t+\triangle t-s)v_1,v,s)\cr
&\qquad={\mathcal M}_{\nu}(f)(x,v,t)
-(t+\triangle t-s)\partial_x\mathcal{M}_{\nu}(f)(x_{\theta_1},v,t_{\theta_1})\cr
&\qquad+(s-t)\partial_t\mathcal{M}_{\nu}(f)(x_{\theta_1},v,t_{\theta_1})\cr
&\qquad=\mathcal{M}_{\tilde{\nu}}(\tilde{f})(x,v,t)+\big\{{\mathcal M}_{\nu}(f)(x,v,t)-\mathcal{M}_{\tilde{\nu}}(\tilde{f})(x,v,t)\big\}\cr
&\qquad-(t+\triangle t-s)\partial_x\mathcal{M}_{\nu}(f)(x_{\theta_1},v,t_{\theta_1})\cr
&\qquad+(s-t)\partial_t\mathcal{M}_{\nu}(f)(x_{\theta_1},v,t_{\theta_1}).
\end{align*}
Therefore, we have
\begin{equation}\label{IM}
I_1=\triangle t{\mathcal M}_{\widetilde{\nu}}(x-\triangle tv_1,v,t)+R_1.
\end{equation}
On the other hand, by Taylor expansion around $(x,v,t+\triangle t)$, we get
\begin{align*}
&f(x-(t+\triangle t-s)v,v,s)\cr
&\qquad=f(x,v,t+\triangle t)
+(s-t-\triangle t)\big(\partial_t+v\cdot\nabla_x\big)f(x_{\theta_2},v,t_{\theta_2})\cr
&\qquad=f(x,v,t+\triangle t)
+\frac{1}{\kappa}(s-t-\triangle t)\left\{\mathcal{M}_{\nu}(f)-f\right\}(x_{\theta_2},v,t_{\theta_2})
%+(s-t-\triangle t)\frac{df}{dt}(x_{\theta_2},v,t_{\theta_2}),
\end{align*}
for some $x_{\theta_2}$ lies between $x$ and $x-(t+\triangle t-s)$ and $t_{\theta_2}$ $\in$ $[s,t]$.
Therefore, $I_2$ can be rewritten as
\begin{align}\label{If}
I_2=\triangle t f(x,v,t+\triangle t)+R_2.
\end{align}
%for some appropriate $\theta_2\in$ $[t, t+\triangle t].$
%\[
%f(x-(t+\triangle t-s)v,v,s)=f(x,v,t+\triangle t)+s\frac{d}{dt}f(x-(t+\triangle t-\theta_2)v,v,\theta_2).
%\]
Substituting (\ref{IM}) and (\ref{If}) into (\ref{Consist3}), we get
\begin{align*}
f(x,v,t+\triangle t)=\widetilde f(x,v,t)+\frac{A_{\nu}\triangle t}{\kappa}\mathcal{M}_{\tilde{\nu}}(\tilde{f})(x,v,t)-
\frac{A_{\nu}\triangle t}{\kappa} f(x,v,t+\triangle t)
+\frac{A_{\nu}}{\kappa}\big\{R_1-R_2\big\}.
\end{align*}
We then collect relevant terms to derive the desired result.
%\begin{align*}f(x,v,t+\triangle t)%&=\frac{\kappa}{\kappa+\triangle t}f(x-\triangle tv,v,t)+\frac{\triangle t}{\kappa+\triangle %t}{\mathcal M}(f)(x-\triangle tv,v,s)\\
%&+\frac{1}{\kappa+\triangle t}(~R_1-R_2)\\
%&=\frac{\kappa}{\kappa+A_{\nu}\triangle t}\widetilde{f}(x,v,t)+\frac{\triangle t}{\kappa+A_{\nu}\triangle t}
%{\mathcal M}(\tilde{f})(x,v,s)\cr
%&+\frac{A_{\nu}}{\kappa+\triangle t}\big\{R_1-R_2\big\}.
%\end{align*}
\end{proof}
%Before we study these remainder terms, we gather some necessary estimates for continuous solutions of (\ref{ESBGK}). For the proof, we refer to \cite{Yun2}.
We now estimate the remainder terms. First, we need the following estimates.
%%%%%%%%%%%%%%%%%%%%%%%%%%%%%%%%%%%%%%%%%%%%%%%%%%%%%%%%%%%%%%%%%%%%%%%%%%%%%%%%%%%%%%%%%%%%%%%%%%%%%%%%%%%%%%%%%%%%%%%%%%
%
%
%             Lemma
%
%
%%%%%%%%%%%%%%%%%%%%%%%%%%%%%%%%%%%%%%%%%%%%%%%%%%%%%%%%%%%%%%%%%%%%%%%%%%%%%%%%%%%%%%%%%%%%%%%%%%%%%%%%%%%%%%%%%%%%%%%%%%
\begin{lemma}\label{ddtM1}\emph{\cite{Yun2}} Let $f$ be solution in Theorem \ref{Existence_Theorem} corresponding to the initial data $f_0$.
Then we have for $q\geq 5$
\begin{align*}
\|\mathcal{M}_{\nu}(f)-\mathcal{M}_{\nu}(g)\|_{L^{\infty}_q}&\leq C_q\|f-g\|_{L^{\infty}_q},\cr
%\|f\|_{L^{\infty}_q}+\|\mathcal{M}_{\nu}(f)\|_{L^{\infty}_q}&\leq C_{T^f}\|f_0\|_{L^{\infty}_q},\cr
\sum_{0\leq |\alpha|+|\beta|\leq 1}\|\partial^{\alpha}_{\beta}\mathcal{M}_{\nu}(f)\|_{L^{\infty}}&\leq C_{T^f}\big\{\|f\|_{W^{1,\infty}_q}+1\big\}.
\end{align*}
\end{lemma}

%%%%%%%%%%%%%%%%%%%%%%%%%%%%%%%%%%%%%%%%%%%%%%%%%%%%%%%%%%%%%%%%%%%%%%%%%%%%%%%%%%%%%%%%%%%%
%
%   Estimates of the remainder
%
%
%%%%%%%%%%%%%%%%%%%%%%%%%%%%%%%%%%%%%%%%%%%%%%%%%%%%%%%%%%%%%%%%%%%%%%%%%%%%%%%%%%%%%%%%%%%%
\begin{lemma}\label{Remainder1}
$R_1$, $R_2$ satisfy the following estimate:
\begin{align}
\|R_1\|_{L^{\infty}_q}+\|R_2\|_{L^{\infty}_q}
\leq C_{T^f,f_0}(\triangle t)^2.
\end{align}
\end{lemma}
\begin{proof}
We start with $R_1$. We decompose
\begin{align*}
\mathcal{M}_{\tilde{\nu}}(\tilde{f})-\mathcal{M}_{\nu}(f)
& =\Big\{\mathcal{M}_{\tilde{\nu}}(\tilde{f})-\mathcal{M}_{\nu}(\tilde{f})\Big\}
+\Big\{\mathcal{M}_{\nu}(\tilde{f})-\mathcal{M}_{\nu}(f)\Big\}\cr
&\equiv I+II.
\end{align*}
For $I$, we compute
\begin{align*}
I=(\tilde{\nu}-\nu)\frac{\partial\mathcal{T}_{\nu}}{\partial \nu}\frac{\partial\mathcal{M}_{\nu}}{\partial\mathcal{T}_{\nu}}
=-\frac{\nu A_{\nu}\triangle t}{\kappa+A_{\nu}\triangle t}\big\{\widetilde{T}Id-\widetilde{\Theta}\big\}\frac{\partial\mathcal{M}_{\nu}}{\partial\mathcal{T}_{\nu}}.
\end{align*}
Then, since
\begin{align*}
|TId-\Theta|&\leq \frac{1}{\rho}\int_{\mathbb{R}^3}f\big||v-U|^2Id+(v-U)\otimes(v-U)\big|dv\cr
&\leq\frac{C}{\rho}\int_{\mathbb{R}^3}f|v-U|^2dv\cr
&=\frac{C}{\rho}\left\{\int_{\mathbb{R}^3}f|v|^2dv+\rho|U|^2\right\}\cr
&\leq C\rho^{-1}\|f\|_{L^{\infty}_q}+|U|^2,
\end{align*}
the lower and upper bound estimates of the macroscropic field given in Theorem \ref{Existence_Theorem} yield
$|TId-\Theta|\leq C_{T^f}$. On the other hand, it was derived  in \cite{Yun2} that
\begin{align*}
\Big\|\frac{\partial\mathcal{M}_{\nu}}{\partial\mathcal{T}_{\nu}}\Big\|_{L^{\infty}_q}\leq C_{T^f}.
\end{align*}
Therefore we can estimate $I$ as
\begin{align*}
\|I\|_{L^{\infty}_q}\leq C^{T^f\!\!,\kappa,\nu}\triangle t.
\end{align*}
For $II$, we first recall  Lemma \ref{ddtM1} to deduce
\[
\|II\|_{L^{\infty}_q}\leq C\|\tilde{f}-f\|_{L^{\infty}_q}.
\]
Then we apply the mean value theorem to estimate
\begin{align*}
\|f-\tilde f\|_{L^{\infty}_q}\leq \|\triangle t v\cdot \nabla_x f\|_{L^{\infty}_q}\leq C_q\|f\|_{W^{1,\infty}_{q+1}}\triangle t\leq
C_{T^f}\Big\{\|f_0\|_{W^{1,\infty}_{q+1}}+1\Big\}\triangle t
\end{align*}
to get
\[
\|II\|_{L^{\infty}_q}\leq C_{T^f}\Big\{\|f_0\|_{W^{1,\infty}_{q+1}}+1\Big\}\triangle t.
\]
We combine these estimates to obtain
\[
\|\mathcal{M}_{\tilde{\nu}}(\tilde{f})-\mathcal{M}_{\nu}(f)\|_{L^{\infty}_q}\leq C_{T^f,f_0}\triangle t,
\]
from which we can estimate
\begin{align*}
\Big\|\int^{t+\triangle t}_t \mathcal{M}_{\tilde{\nu}}(\tilde{f})-\mathcal{M}_{\nu}(f)ds\Big\|_{L^{\infty}_q}
\leq \int^{t+\triangle t}_t \big\|\mathcal{M}_{\tilde{\nu}}(\tilde{f})-\mathcal{M}_{\nu}(f)\big\|_{L^{\infty}_q}ds
%&\leq C_q\int^{t+\triangle t}_t \big\|\tilde{f}-f\big\|_{L^{\infty}_q}ds\cr
%&\leq C_{T^f,f_0}\left\{\int^{t+\triangle t}_t ds\right\}\triangle t\cr
\leq C_{T^f,f_0}(\triangle t)^2.
\end{align*}
On the other hand, again from Lemma \ref{ddtM1},
\begin{align*}
\Big\|\int^{t+\triangle t}_t(t+\triangle t-s)\partial_x\mathcal{M}_{\nu}(f)ds\Big\|_{L^{\infty}_q}
\leq\|\partial_x\mathcal{M}_{\nu}(f)\|_{L^{\infty}_q}\int^{t+\triangle t}_t(t+\triangle t-s)ds
\leq C_{f_0,q}(\triangle t)^2,
\end{align*}
and \begin{align*}
\Big\|\int^{t+\triangle t}_t(s-t)\partial_t\mathcal{M}_{\nu}(f)ds\Big\|_{L^{\infty}_q}
\leq\|\partial_t\mathcal{M}_{\nu}(f)\|_{L^{\infty}_q}\int^{t+\triangle t}_t(s-t)ds
\leq C_{f_0,q}(\triangle t)^2.
\end{align*}
This gives the desired remainder estimate for $R_{1}$. We compute $R_{2}$ similarly as
\begin{align*}
&\Big\|\int^{t+\triangle t}_t(s-t-\triangle t)\big(\mathcal{M}_{\nu}(f)-f\big)(x_{\theta_2},v,t_{\theta_2})ds\Big\|_{L^{\infty}_q}\cr
&\hspace{2cm}\leq\|\mathcal{M}_{\nu}(f)-f\|_{L^{\infty}_q}\int^{t+\triangle t}_t(s-t-\triangle t)ds\cr
&\hspace{2cm}\leq C_{T^f}\|f^0\|_{L^{\infty}_q}(\triangle t)^2.
\end{align*}
\end{proof}
%%%%%%%%%%%%%%%%%%%%%%%%%%%%%%%%%%%%%%%%%%%%%%%%%%%%%%%%%%%%%%%%%%%%%%%%%%%%%%%%%%%%%%%%%%%%%%%%%%%%%%%%%%%%%%%%%%%%%%
%
%
%                        Section: Convergence
%
%
%%%%%%%%%%%%%%%%%%%%%%%%%%%%%%%%%%%%%%%%%%%%%%%%%%%%%%%%%%%%%%%%%%%%%%%%%%%%%%%%%%%%%%%%%%%%%%%%%%%%%%%%%%%%%%%%%%%%%%%%
\section{ Estimate of $\mathcal{M}_{\widetilde{\nu}}(\tilde{f}(t_n))-\mathcal{M}_{\widetilde{\nu}}(\tilde{f}^n)$}
The main purpose of this section is to establish the continuity estimate of the discrete ellipsoidal Gaussian. (See Proposition \ref{ContinuityofMaxwellianDiscrete}).
%%%%%%%%%%%%%%%%%%%%%%%%%%%%%%%%%%%%%%%%%%%%%%%%%%%%%%%%%%%%%%%%%%%%%%%%%%%%%%%%%%%%%%%%%%%%%%%%%%%%%%%%%%%
%
%
% Lemma
%
%
%%%%%%%%%%%%%%%%%%%%%%%%%%%%%%%%%%%%%%%%%%%%%%%%%%%%%%%%%%%%%%%%%%%%%%%%%%%%%%%%%%%%%%%%%%%%%%%%%%%%%%%%%%%
\begin{lemma}\label{Additional_Lemma1} Assume $q>5$ and $\triangle v<r_{\triangle v}$. Let $f$ and $f^n$ denote the solution of (\ref{ESBGK}) and (\ref{main scheme}) respectively. Then,
the difference of $\tilde{f}^n$ and $\tilde{f}(t_n)$ satisfies
\begin{align*}
\|\tilde{f}^n-\tilde{f}(t_n)\|_{L^{\infty}_q}\leq \|f^n-f(t_n)\|_{L^{\infty}_q}+C_{T^f}\|f_0\|_{W^{1,\infty}_q}(\triangle x)^2.
\end{align*}
\end{lemma}
\begin{proof}
(1) We expand  $\tilde{f}(x_i,v,t_n)=f(x_i-\triangle t v,v,t_n)$ in the following two ways using Taylor's theorem:
\begin{align*}
f(x_{s},v_j,t_n)&=\tilde{f}(x_i,v_j,t_n)+(x_s-x_i+v_{j_1}\triangle t)\widetilde{\partial_xf}(x_i,v_j,t_n)+\frac{1}{2}(x_s-x_i+v_{j_1}\triangle t)^2\widetilde{\partial^2_xf}(x_{\theta_1},v_j,t_n),\cr
f(x_{s+1},v_j,t_n)&=\tilde{f}(x_i,v_j,t_n)+(x_{s+1}-x_i+v_{j_1}\triangle t)\widetilde{\partial_xf}(x_i,v_j,t_n)+\frac{1}{2}(x_{s+1}-x_i+v_{j_1}\triangle t)^2
\widetilde{\partial^2_xf}(x_{\theta_1},v_j,t_n),
\end{align*}
for some $x_s< x_{\theta_1}< x_i+v_{j_1}\triangle t$ and $x_{i}-v_{j_1}\triangle t< x_{\theta_2}< x_{s+1}$. Making a convex combination of above two identities, we get
(for the definition of $a_{j_1}$, see Lemma \ref{a}.)
\begin{align*}
\tilde{f}(x_i,v,t_n)=(1-a_{j_1})f(x_{s},v_j,t_n)+a_{j_1}f(x_{s+1},v_j,t_n)+R_{\theta}\cr
\end{align*}
where $R_{\theta}$ is given by
\begin{align}\label{R0}
R_{\theta}=\frac{1}{2}(x_{s+1}-x_i+v_{j_1}\triangle t)^2
\widetilde{\partial^2_xf}(x_{\theta_1},v_j,t_n)+\frac{1}{2}(x_{s+1}-x_i+v_{j_1}\triangle t)^2\widetilde{\partial^2_xf}(x_{\theta_2},v_j,t_n).
\end{align}
Note that the first order derivitives cancelled each other due to the definition of $a_{j_1}$:
\begin{align}\label{aj1}
(x_s-x_i+v_{j_1}\triangle t)=-a_{j_1}\triangle x,\quad (x_{s+1}-x_i+v_{j_1}\triangle t)=(1-a_{j_1})\triangle x.
\end{align}
We then use
\begin{align}\label{together}
|x_s-(x_i-v_{j_1}\triangle t)|,~|x_{s+1}-(x_i-v_{j_1}\triangle t)|\leq \triangle x
\end{align}
together with
\begin{align*}
\displaystyle|\widetilde{\partial^2_xf}(x_{\theta_i},v_j,t_n)|, &\leq \frac{\|f\|_{W^{2,\infty}_q}}{(1+|v_j|)^q}
%&\leq \frac{\|f\|_{W^{1,\infty}_q}}{(1-|\triangle v|+|v_j|)^q}\cr
%&\leq 2^q\frac{\|f\|_{W^{1,\infty}_q}}{(1+|v_j|)^q}\cr
\leq C_{T^f}\frac{\|f_0\|_{W^{2,\infty}_q}+1}{(1+|v_j|)^q},
\end{align*}
%In the last inequality, we used Lemma \ref{ddtM1} (2).
to estimate $R_{\theta}$ as
\begin{align}\label{R00}
|R_{\theta}|%&=\big|\left\{(1-a_{j_1})(-a_{j_1})-a_{j_1}(1-a_{j_1})\right\}\triangle x\partial_x\tilde{f}(x_i,v_j,t_n)\cr
%&+\frac{1}{2}(x_{s+1}-x_i+v_{j_1}\triangle t)^2
%\partial^2_x\tilde{f}(x_{\theta_1},v_j,t_n)+\frac{1}{2}(x_{s+1}-x_i+v_{j_1}\triangle t)^2\partial^2_x\tilde{f}(x_{\theta_2},v_j,t_n)\big|\cr
%&=\big|\frac{1}{2}(x_{s+1}-x_i+v_{j_1}\triangle t)^2
%\partial^2_x\tilde{f}(x_{\theta_1},v_j,t_n)+\frac{1}{2}(x_{s+1}-x_i+v_{j_1}\triangle t)^2\partial^2_x\tilde{f}(x_{\theta_2},v_j,t_n)\big|\cr
\leq C_{T^f}\frac{\|f_0\|_{W^{2,\infty}_q}}{(1+|v_j|)^q}(\triangle x)^2.
\end{align}

%\begin{equation}\label{R0}
%|R_{\theta}|\leq C_{T^f}\frac{\|f_0\|_{W^{1,\infty}_q}}{(1+|v_j|)^q}\textcolor{red}{(\triangle x)^2}.
%\end{equation}
With these estimates, we can compute the difference of $\widetilde{f}^n_{i,j}$ and $\widetilde{f}(x_i,v,t_n)$
 as follows:
\begin{align*}
|\widetilde{f}^n_{i,j}-\widetilde{f}(x_i,v_j,t_n)|
%&\quad=|(1-a)f^n_{s,j}+af^n_{s+1,j}-f(x_{i},v,t_n)|\cr
&\leq|(1-a_{j_1})f^n_{s,j}+a_{j_1}f^n_{s+1,j}-(1-a_{j_1})f(x_{s},v_j,t_n)-a_{j_1}f(x_{s+1},v_j,t_n)|+|R_{\theta}|\cr
%&\quad|\widetilde{f}^n_{i,j}-f(x_i-v_{j_1}^1,v_j,t_n)|\cr
%&&\quad=(1-a)f^n_{s,j}+af^n_{s+1,j}-f(x_i-v_{j_1}^1,v_j,t_n)\cr
%&\quad\leq|(1-a)f^n_{s,j}+af^n_{s+1,j}-(1-a)f(x_{s},v_j,t_n)-af(x_{s+1},v_j,t_n)|\cr
%&\quad+C_{T^f}\frac{\|\partial_xf_0\|_{L^{\infty}_q}}{(1+|v_j|)^q}\triangle x\cr
&\leq(1-a_{j_1})|f^n_{s,j}-f(x_{s},v_j,t_n)|+a_{j_1}|f^n_{s+1,j}-f(x_{s+1},v_j,t_n)|+|R_{\theta}|.
\end{align*}
We then multiply $(1+|v_j|)^q$ on both sides and take supremum over $i,j$.
The desired result follows from (\ref{R0}).
\end{proof}
%%%%%%%%%%%%%%%%%%%%%%%%%%%%%%%%%%%%%%%%%%%%%%%%%%%%%%%%%%%%%%%%%%%%%%%%%%%%%%%%%%%%%%%%%%%%%%%%%%%%%%%%%%%
%
%
% Lemma
%
%
%%%%%%%%%%%%%%%%%%%%%%%%%%%%%%%%%%%%%%%%%%%%%%%%%%%%%%%%%%%%%%%%%%%%%%%%%%%%%%%%%%%%%%%%%%%%%%%%%%%%%%%%%%%
\begin{lemma}\label{Additional_Lemma2} Assume $q>5$ and $|\triangle v|<r_{\triangle v}$. Let $f$ and $f^n$ denote the solution of (\ref{ESBGK}) and (\ref{main scheme}) respectively. Let  $\phi(v)$ denote one of  $1,v,|v|^2, (1-\nu)|v|^2Id+\nu v\otimes v$ for $v\in \mathbb{R}^3$.
Then we have
\begin{align*}
&\Big|\,\sum_j\widetilde{f}^n_{i,j}\phi(v_j)(\triangle v)^3-\int_{\mathbb{R}^3}\widetilde{f}(x_i,v,t_n)\phi(v)dv\Big|\cr
&\qquad\leq C_{T_f}\|f^n-f(t_n)\|_{L^{\infty}_q}+C_{T^f}\|f_0\|_{W^{2,\infty}_q}
\big\{(\triangle x)^2+\triangle v+\triangle v\triangle t\big\}.
%&(b)~\|\widetilde{f}^n-E\big(\widetilde{f}^n\big)\|_{L^{\infty}_q}\leq C_q \|f_0\|_{L^{\infty}_q}\triangle x.
\end{align*}
\end{lemma}
\begin{proof}
 For simplicity, we define
\[
\triangle_j=[v_{j_1},v_{j_1+1}]\times[v_{j_2},v_{j_2+1}]\times [v_{j_3},v_{j_3+1}],
\]
so that
\begin{align*}
\int_{\mathbb{R}^3}\widetilde{f}(x_i,v,t_n)\phi(v)dv=\sum_j\int_{\triangle_j}\widetilde{f}(x_i,v,t_n)\phi(v)dv.
%&=\sum_j\int_{\triangle_j}\widetilde{f}(x_i,v,t_n)\phi(v_j)dv
%+\sum_j\int_{\triangle_j}\widetilde{f}(x_i,v,t_n)\big\{\phi(v)-\phi(v_j)\big\}dv.
%&\equiv\sum_j\int_{\triangle_j}\widetilde{f}(x_i,v,t_n)|v_j|^pdv+I.
\end{align*}
Therefore,
\begin{align*}
&\sum_j\widetilde{f}^n_{i,j}\phi(v_j)(\triangle v)^3-\int_{\mathbb{R}^3}\widetilde{f}(x_i,v,t_n)\phi(v)dv\cr
&\hspace{1.2cm}=\sum_j\widetilde{f}^n_{i,j}\phi(v_j)(\triangle v)^3-\sum_j\int_{\triangle_j}\widetilde{f}(x_i,v,t_n)\phi(v)dv\cr
&\hspace{1.2cm}=\Big\{\sum_j\widetilde{f}^n_{i,j}\phi(v_j)(\triangle v)^3-\sum_j\int_{\triangle_j}\widetilde{f}(x_i,v,t_n)\phi(v_j)dv\Big\}\cr
&\hspace{1.2cm}+\sum_j\int_{\triangle_j}\widetilde{f}(x_i,v,t_n)\big\{\phi(v_j)-\phi(v)\big\}dv\cr
&\hspace{1.2cm}\equiv I+II.
\end{align*}

$\bullet$ {\bf(The estimate of $I$)}:
Assume $v\in\triangle_j$ and expand $\tilde{f}(x_i,v,t_n)=f(x_i- v^1\triangle t,v,t_n)$ in the following two ways:

\begin{align*}
f(x_{s},v_j,t_n)&=\tilde{f}(x_i,v,t_n)+(x_s-x_i+v^1\triangle t)\widetilde{\partial_xf}(x_i,v,t_n)+\frac{1}{2}(x_s-x_i+v^1\triangle t)^2
\widetilde{\partial^2_xf}(z_{\theta_{s,1}})\cr
&+(v-v_{j})\cdot\nabla_v \tilde{f}(z_{\theta_{s,2}}),
\end{align*}
and
\begin{align*}
f(x_{s+1},v_j,t_n)&=\tilde{f}(x_i,v,t_n)+(x_{s+1}-x_i+v^1\triangle t)\widetilde{\partial_xf}(x_i,v,t_n)+\frac{1}{2}(x_{s+1}-x_i+v^1\triangle t)^2
\widetilde{\partial^2_xf}(z_{\theta_{s+1,1}})\cr
&+(v-v_{j})\cdot\widetilde{\nabla_v f}(z_{\theta_{s+1,2}}),
\end{align*}
%\begin{align*}\tilde{f}(x_i,v,t_n)&=f(x_{s+1},v_j,t_n)-(x_i-v^1\triangle t-x_{s+1})\partial_xf(z_{\theta_1})+(v-v_{j})\cdot\nabla_v f(z_{\theta_2}),
%\end{align*}
where $z_{\theta_{s,i}}$ and $z_{\theta_{s+1,i}}$ $(i=1,2)$ denote the mean value points defined similarly as in the previous case.
%$z_{\theta_{s,i}}=(x+\theta^x_{s,i}\triangle x,v_{j}+\theta^v_{s,i}\triangle v,t_n),\quad z_{\theta_{s+1,i}}=(x+\theta^x_{s+1,i}\triangle x,v_{j}+\theta^v_{s+1,2}\triangle v,t_n)$
%for some constants $\theta^x_{s,i},\,\theta^v_{s,i}, \theta^x_{s+1,i},\,\theta^v_{s+1,i}\in [-1,1]$ $(i=1,2)$.
We rewrite $f(x_{s},v_j,t_n)$ and $f(x_{s+1},v_j,t_n)$ as
\begin{align*}
f(x_{s},v_j,t_n)&=\tilde{f}(x_i,v,t_n)+(x_s-x_i+v_{j_1}\triangle t)\widetilde{\partial_xf}(x_i,v,t_n)
+\frac{1}{2}(x_s-x_i+v^1\triangle t)^2\widetilde{\partial^2_xf}(z_{\theta_{s,1}})\cr
&+(v^1-v_{j_1})\triangle t\widetilde{\partial_xf}(x_i,v,t_n)+(v-v_{j})\cdot\widetilde{\nabla_v f}(z_{\theta_{s,2}}),\cr
f(x_{s+1},v_j,t_n)&=\tilde{f}(x_i,v,t_n)+(x_{s+1}-x_i+v_{j_1}\triangle t)\widetilde{\partial_xf}(x_i,v,t_n)+\frac{1}{2}(x_{s+1}-x_i+v^1\triangle t)^2\widetilde{\partial^2_xf}(z_{\theta_{s+1,1}})\cr
&+(v^1-v_{j_1})\triangle t\widetilde{\partial_xf}(x_i,v,t_n)+(v-v_{j})\cdot\widetilde{\nabla_v} \tilde{f}(z_{\theta_{s+1,2}}),
\end{align*}
and make a linear combination and cancel out the first order derivatives in $x$ using (\ref{aj1}) as in the proof of the previous lemma, to get
\begin{align}\label{toget}
\tilde{f}(x_i,v,t^n)&=(1-a_{j_1})f(x_s,v_j,t_n)+a_{j_1}f(x_{s+1},v_j,t_n)+R_{\theta}
\end{align}
%These two identities are then combined to yield
%\begin{align*}
%\tilde{f}(x_i,v,t_n)=(1-a_{j_1})f(x_{s},v_j,t_n)+a_{j_1}f(x_{s+1},v_j,t_n)+R_{\theta},
%\end{align*}
where
\begin{align*}
R_{\theta}&=-(v^1-v_{j_1})\triangle t\widetilde{\partial_xf}(x_i,v,t_n)\cr
&+\frac{1}{2}(1-a_{j1})(x_s-x_i+v^1\triangle t)^2\tilde{\partial^2_xf}(z_{\theta_{s,1}})+\frac{1}{2}a_{j1}(x_s-x_i+v^1\triangle t)^2\widetilde{\partial^2_xf}(z_{\theta_{s+1,1}})\cr
%&+(1-a_{j_1})(\triangle v)^2(\triangle t)^2\partial^2_x\tilde{f}(z_{\theta_{s,1}})+a_{j_1}(\triangle v)^2(\triangle t)^2\partial^2_x\tilde{f}(z_{\theta_{s,1}})\cr
&+(1-a_{j_1})(v-v_j)\cdot\widetilde{\nabla_vf}(z_{\theta_{s,2}})+a_{j_1}(v-v_j)\cdot\widetilde{\nabla_vf}(z_{\theta_{s+1,2}})
\end{align*}
%\begin{align*}
%R_{\theta}&=\triangle v\triangle t\partial_x\tilde{f}(x_i,v,t_n)\cr
%&+\frac{1}{2}(1-a_{j1})(x_s-x_i+v_{j_1}\triangle t)^2\tilde{f}(z_{\theta_{s,1}})+\frac{1}{2}a_{j1}(x_s-x_i+v_{j_1}\triangle t)^2\partial^2_x\tilde{f}(z_{\theta_{s+1,1}})\cr
%&+(1-a_{j_1})(\triangle v)^2(\triangle t)^2\partial^2_x\tilde{f}(z_{\theta_{s,1}})+a_{j_1}(\triangle v)^2(\triangle t)^2\partial^2_x\tilde{f}(z_{\theta_{s,1}})\cr
%&+(1-a_{j_1})(v-v_j)\cdot\nabla_vf(z_{\theta_{s,2}})+a_{j_1}(v-v_j)\cdot\nabla_vf(z_{\theta_{s+1,2}})
%\end{align*}
From this, we see that
\begin{align*}
&\sum_j\int_{\triangle_j}\widetilde{f}(x_i,v,t_n)\phi(v_j)dv\cr
&\qquad=\sum_j\int_{\triangle_j}\left\{(1-a_{j_1})f(x_s,v_j,t_n)+a_{j_1}f(x_{s+1},v_j,t_n)+R_{\theta}\right\}\phi(v_j)dv\cr
&\qquad=\sum_j\left\{(1-a_{j_1})f(x_s,v_j,t_n)+a_{j_1}f(x_{s+1},v_j,t_n)\right\}(\triangle v)^3\phi(v_j)\cr
&\qquad+\sum_j\int_{\triangle_j}R_{\theta}\phi(v_j)dv,
\end{align*}
so that
\begin{align*}
&\Big|\sum_j\widetilde{f}^n_{i,j}\phi(v_j)(\triangle v)^3-\int_{\mathbb{R}^3}\widetilde{f}(x_i,v,t_n)\phi(v)dv\Big|\cr
%&\quad=\sum_j\widetilde{f}^n_{i,j}\phi(v_j)(\triangle v)^3-\sum_j\int_{\triangle_j}\widetilde{f}(x_i,v,t_n)\phi(v)dv\cr
%&\quad=\sum_j\widetilde{f}^n_{i,j}\phi(v_j)(\triangle v)^3-\sum_j\int_{\triangle_j}\widetilde{f}(x_i,v,t_n)(1+|v_j|)^pdv+I\cr
%&\quad=\sum_j\widetilde{f}^n_{i,j}\phi(v_j)(\triangle v)^3-\sum_j\int_{\triangle_j}\left\{(1-a)f(x_s,v_j,t_n)+af(x_{s+1,v_j,t_n})+R_{\theta}\right\}(1+|v_j|)^pdv+I\cr
&\qquad\leq\sum_j\big|\,\widetilde{f}^n_{i,j}-\big\{(1-a_{j_1})f(x_s,v_j,t_n)+a_{j_1}f(x_{s+1},v_j,t_n)\big\}\big||\phi(v_j)|(\triangle v)^3\cr
&\qquad+\sum_j\int_{\triangle_j}|R_{\theta}||\phi(v_j)|dv\cr
&\qquad=I_1+I_2.
\end{align*}
We first estimate $I_1$ since we have from the definition of $\widetilde{f}_{i,j}$,
\begin{align*}
&|\widetilde{f}^n_{i,j}-(1-a_{j_1})f(x_{s},v_j,t_n)-a_{j_1}f(x_{s+1},v_j,t_n)|\cr
%&\quad=|(1-a)f^n_{s,j}+af^n_{s+1,j}-f(x_{i},v,t_n)|\cr
&\quad=|(1-a_{j_1})f^n_{s,j}+a_{j_1}f^n_{s+1,j}-(1-a_{j_1})f(x_{s},v_j,t_n)-a_{j_1}f(x_{s+1},v_j,t_n)|\cr
%&\quad|\widetilde{f}^n_{i,j}-f(x_i-v_{j_1}^1,v_j,t_n)|\cr
%&&\quad=(1-a)f^n_{s,j}+af^n_{s+1,j}-f(x_i-v_{j_1}^1,v_j,t_n)\cr
%&\quad\leq|(1-a)f^n_{s,j}+af^n_{s+1,j}-(1-a)f(x_{s},v_j,t_n)-af(x_{s+1},v_j,t_n)|\cr
%&\quad+C_{T^f}\frac{\|\partial_xf_0\|_{L^{\infty}_q}}{(1+|v_j|)^q}\triangle x\cr
&\quad\leq(1-a_{j_1})|f^n_{s,j}-f(x_{s},v_j,t_n)|+a_{j_1}|f^n_{s+1,j}-f(x_{s+1},v_j,t_n)|,
\end{align*}
$I_1$ can be estimated as follows:
\begin{align*}
I_1&=\sum_j\big|\,\widetilde{f}^n_{i,j}-\big\{(1-a_{j_1})f(x_s,v_j,t_n)+a_{j_1}f(x_{s+1},v_j,t_n)\big||\phi(v_j)|(\triangle v)^3\cr
&\leq\sum_j\Big\{(1-a_{j_1})\|f^n-f(t_n)\|_{L^{\infty}_q}+a_{j_1}\|f^n-f(t_n)\|_{L^{\infty}_q}\Big\}\frac{|\phi(v_j)|(\triangle v)^3}{(1+|v_j|)^q}\cr
&\leq\Big\{\sum_j\frac{\phi(v_j)(\triangle v)^3}{(1+|v_j|)^q}\Big\}\|f^n-f(t_n)\|_{L^{\infty}_q}\cr
&\leq C\|f^n-f(t_n)\|_{L^{\infty}_q},
\end{align*}
where we used $\phi(v_j)\leq C(1+|v_j|)^p$, $p=0,1,2$ and $q-p>3$.\newline
For $I_2$, we first observe from the definition of $x_s$ and the fact that $v\in\triangle_j$,
\begin{eqnarray*}
|x_i-v^1\triangle t-x_s|^2\leq 2|x_i-v_{j_1}\triangle t-x_s|^2+2|(v^1-v_{j_1})\triangle t|^2\leq 2\left\{(\triangle x)^2+(\triangle v)^2(\triangle t)^2\right\}.
\end{eqnarray*}
Similarly,
\begin{eqnarray*}
|x_{s+1}-(x_i-v^1\triangle t)|\leq  2\left\{(\triangle x)^2+(\triangle v)^2(\triangle t)^2\right\}.
\end{eqnarray*}
On the other hand, since $\triangle v<1/2$, we have ($i=1,2$)
\begin{align*}
\sum_{|\alpha|+|\beta|\leq 2}\big|\widetilde{\partial^{\alpha}_xf}(z_{\theta_i})\big|
+\big|\widetilde{\partial^{\beta}_vf}(z_{\theta_i})\big|&\leq \frac{\|f\|_{W^{2,\infty}_q}}{(1+|v_j+\theta_i\triangle v|)^q}\cr
&\leq \frac{\|f\|_{W^{2,\infty}_q}}{(1-|\triangle v|+|v_j|)^q}\cr
%&\leq 2^q\frac{\|f\|_{W^{1,\infty}_q}}{(1+|v_j|)^q}\cr
&\leq C_{T^f}\frac{\|f_0\|_{W^{2,\infty}_q}+1}{(1+|v_j|)^q},
\end{align*}
so that
\begin{align}\label{R}
|R_{\theta}|&\leq C_{T^f}\{(\triangle x)^2+(\triangle v)^2(\triangle t)^2+\triangle v+\triangle v\triangle t\}\frac{\|f_0\|_{W^{2,\infty}_q}+1}{(1+|v_j|)^q}\cr
&\leq C_{T^f}\{(\triangle x)^2+\triangle v+\triangle v\triangle t\}\frac{\|f_0\|_{W^{2,\infty}_q}+1}{(1+|v_j|)^q}.
\end{align}
Hence we have
\begin{align*}
I_2\leq
C_{T^f}\{(\triangle x)^2+\triangle v+\triangle v\triangle t\}\|f_0\|_{W^{2,\infty}_q}\sum_j\frac{|\phi(v_j)|(\triangle v)^3}{(1+|v_j|)^q}
\leq C_{T^f,f_0}\{(\triangle x)^2+\triangle v+\triangle v\triangle t\},
\end{align*}
where we used $\int_{\triangle_j}dv=(\triangle v)^3$.
Therefore, we have the following estimate for $I$
\begin{align*}
I\leq C\|f^n-f(t_n)\|_{L^{\infty}_q}+C_{T^f,f_0}\{(\triangle x)^2+\triangle v+\triangle v\triangle t\}.
\end{align*}
\newline
$\bullet$ {\bf (The estimate of $II$)}:
Since $|v_j|=|v+\theta \triangle v|\leq (1+|v|)$  in $\triangle_j$, we have for $v\in\triangle_j$
\[
|\phi(v)-\phi(v_j)\big|\leq C|v-v_j|\big\{|v|^p+|v_j|^p\big\}\leq C\triangle v(1+|v|)^p.\quad (p=0,1,2)
\]
With this, we can estimate $II$ as
\begin{align*}
II%&=\triangle v\sum_j\int_{\triangle_j}\widetilde{f}(x_i,v,t_n) \Big\{\sum_{k=0}^{p-1} |v_j|^k|v|^{p-1-k}\Big\}dv\cr
&\leq C\triangle v \sum_j\int_{\triangle_j}\widetilde{f}(x_i,v,t_n)(1+|v|)^{p}dv\cr
&\leq C\triangle v \|f\|_{L^{\infty}_{q}}\sum_j\int_{\triangle j}\frac{dv}{(1+|v|)^{q-p+1}}\cr
&\leq C\triangle v \|f\|_{L^{\infty}_{q}}\int_{\mathbb{R}^3}\frac{dv}{(1+|v|)^{q-p+1}}\cr
&\leq C_{q-p+1}\|f\|_{L^{\infty}_{q}}\triangle v,
\end{align*}
which, upon applying Lemma \ref{ddtM1} yields
\begin{align}\label{I}
II\leq C_{T^f}\|f_0\|_{L^{\infty}_{q}}\triangle v.
\end{align}
Finally, we combine the estimates for $I$ and $II$ to get the desired result.
\end{proof}
%%%%%%%%%%%%%%%%%%%%%%%%%%%%%%%%%%%%%%%%%%%%%%%%%%%%%%%%%%%%%%%%%%%%%%%%%%%%%%%%%%%%%%%%%%%%%
%
%
%
%
%%%%%%%%%%%%%%%%%%%%%%%%%%%%%%%%%%%%%%%%%%%%%%%%%%%%%%%%%%%%%%%%%%%%%%%%%%%%%%%%%%%%%%%%%%%%%5
\begin{lemma} \label{rho-rho}Let $\triangle v<r_{\triangle}$. Suppose $f_0$ satisfies the assumptions in Theorem \ref{Existence_Theorem}, then we have
\begin{align*}
&|\widetilde{\rho}^n_{i}-\widetilde{\rho}(x_i,t_n)|, ~|\widetilde{U}^n_{i}-\widetilde{U}(x_i,t_n)|,~ |\widetilde{\mathcal{T}}^n_{\nu,i}-\widetilde{\mathcal{T}}_{\widetilde{\nu}}(x_i,t_n)|\cr
&\hspace{2cm}\leq C\|f^n-f(t_n)\|_{L^{\infty}_q}+C_{T^f,f_0}\{(\triangle x)^2+\triangle v+\triangle v\triangle t\}.
\end{align*}
\end{lemma}
\begin{proof}
Note that
\begin{align*}
\widetilde{\rho}^n_{i}-\widetilde{\rho}(x_i,t_n)=\sum_{j} \widetilde{f}^n_{i,j}(\triangle v)^3-\int_{\mathbb{R}^3}\widetilde{f}(x_i,v,t_n)dv.
\end{align*}
Therefore, it's a direct consequence of the Lemma \ref{Additional_Lemma2}. For the second estimate, we observe that
\begin{align*}
\widetilde{U}^n_{i}-\widetilde{U}(x_i,t_n)&=\frac{1}{\widetilde{\rho}^n_i}\big\{\widetilde{\rho}^n_i \widetilde{U}^n_i-\big\{\widetilde{\rho}
\widetilde{U}\big\}(x_i,t_n)\big\}-\frac{\widetilde{U}}{\widetilde{\rho}^n_i}\big\{\widetilde{\rho}^n_i-\widetilde{\rho}(x_i,t_n)\big\}.
\end{align*}
%which is legitimate due to the positivity estimate established in Proposition \ref{LB T Lemma}.
We then recall the lower bound of the discrete density in Lemma \ref{LB T Lemma}:
\begin{align*}
\widetilde{\rho}^n_i\geq C^1_0C_{\alpha}e^{-\frac{A_{\nu}}{\kappa}T^f},
\end{align*}
to compute
\begin{align*}
\bigg|\frac{\widetilde{U}^n_i}{\widetilde{\rho}^n_i}\bigg|
&=\frac{1}{(\widetilde{\rho}^n_i\big)^2}\big|\widetilde{\rho}^n_i\widetilde{U}^n_i\big|
\leq C_{\alpha,T^f}\sum_j|\widetilde{f}^n_{i,j}||v_j|(\triangle v)^3\cr
&\leq C_{\alpha,T^f}\|f^n\|_{L^{\infty}_q}\sum_j\frac{(\triangle v)^3}{(1+|v_j|)^{q-1}}\cr
&\leq C_{\alpha,T^f}e^{\frac{(C_{\mathcal{M}-1}) T^f}{\kappa+A_{\nu}\triangle t}}\|f^0\|_{L^{\infty}_q}.
\end{align*}
Here, we used Lemma \ref{f stability}. Hence we have
\begin{align*}
|\widetilde{U}^n_{i}-\widetilde{U}(x_i,t_n)|
&\leq C_{T^f,\kappa,f_0}\Big\{|\big\{\widetilde{\rho}^n_i \widetilde{U}^n_i-\big\{\widetilde{\rho}
\widetilde{U}\big\}(x_i,t_n)\big\}|+|\widetilde{\rho}^n_i-\widetilde{\rho}(x_i,t_n)|\Big\}\cr
&\leq C_{T^f,\kappa,f_0}\Big|\sum_j\widetilde{f}^n_{i,j}(\triangle v)^3-\int_{\mathbb{R}^3}\widetilde{f}(x_i,v,t_n)dv\Big|\cr
&+ C_{T^f,\kappa,f_0}\Big|\sum_j\widetilde{f}^n_{i,j}\,v_j(\triangle v)^3-\int_{\mathbb{R}^3}\widetilde{f}(x_i,v,t_n)\,vdv\Big|,
\end{align*}
which, in view of Lemma \ref{Additional_Lemma2},  gives the desired result. For the estimate of the temperature tensor, we recall that $\widetilde{\mathcal{T}}^n_{\tilde{\nu},i}$ contains $\tilde{\nu}$, and decompose it as
\begin{align*}
\widetilde{\mathcal{T}}^n_{\widetilde{\nu},i}-\widetilde{\mathcal{T}}_{\widetilde{\nu}}(x_i,t_n)
&=(1-\widetilde{\nu})\widetilde{T}^n_iId+\widetilde{\nu}\widetilde{\Theta}^n_i
-\left\{(1-\nu)\widetilde{T}(x_i,t_n)Id+\nu\widetilde{\Theta}(x_i,t_n)\right\}\cr
&=(\nu-\widetilde{\nu})\left\{\widetilde{T}^n_iId-\widetilde{\Theta}^n_i\right\}\cr
&+\left[\left\{(1-\nu)\widetilde{T}^n_iId+\nu\widetilde{\Theta}^n_i\right\}
-\left\{(1-\nu)\widetilde{T}(x_i,t_n)Id+\nu\widetilde{\Theta}(x_i,t_n)\right\}\right]\cr
&\equiv I+II.
\end{align*}
Since $|\widetilde{\nu}-\nu|=\frac{\nu\triangle t}{\kappa+\triangle t}$, $I$ is bounded by $ C_{T^f}\frac{\kappa\triangle t}{\kappa+\triangle t}$.
The estimate of $II$ can be carried out in the exactly same manner as in the previous case, through a tedious computation, using the following identity:
\begin{align*}
\mathcal{T}_{\nu,f}-\mathcal{T}_{\nu,g}&=\rho^{-1}_f\left\{(\rho\mathcal{T}_{\nu}+\rho U\otimes U)_f-(\rho\mathcal{T}_{\nu}+\rho U\otimes U)_g\right\}
+\frac{(\rho \mathcal{T}_{\nu}+\rho U\otimes U)_g}{\rho_f\rho_g}(\rho_f-\rho_g)\cr
&+\frac{1}{\rho^2_f}\{\rho_f U_f\otimes \rho_f U_f-\rho_g U_g\otimes \rho_gU_g\}
-\frac{(\rho_f+\rho_g)}{\rho_f^2\rho_g^2}(\rho_g U_g)^2(\rho_f-\rho_g).
\end{align*}
We omit the computation.
\end{proof}
%%%%%%%%%%%%%%%%%%%%%%%%%%%%%%%%%%%%%%%%%%%%%%%%%%%%%%%%%%%%%%%%%%%%%%%%%%%%%%%%%%%%%%%%%%%%
%
%                   Continuity of Maxwellian
%
%%%%%%%%%%%%%%%%%%%%%%%%%%%%%%%%%%%%%%%%%%%%%%%%%%%%%%%%%%%%%%%%%%%%%%%%%%%%%%%%%%%%%%%%%%%%
\begin{proposition}\label{ContinuityofMaxwellianDiscrete}Assume that $\|f^n\|_{L^{\infty}_q}<\infty$
 with $q>5$. Then we have
\begin{align*}
\|\mathcal{M}_{\widetilde{\nu}}(\widetilde{f}(t_n))-\mathcal{M}_{\widetilde{\nu}}(\widetilde{f}^n)\|_{L^{\infty}_q}\leq C_{T^f}\|f^n - f(t_n)\|_{L^{\infty}_q}
+C_{T^f}\{(\triangle x)^2+\triangle v+\triangle v\triangle t\}.
\end{align*}
The constants depend only on $q, \nu, T^f$ and $f_0$.
\end{proposition}
\begin{proof}
We note that
\begin{align*}
&\mathcal{M}_{\widetilde{\nu}}(\tilde{f}(x_i,t_n))(v_j)-\mathcal{M}_{\widetilde{\nu},j}(\widetilde{f}^n_i)\cr
&\qquad=\mathcal{M}_{\widetilde{\nu}}\big(\widetilde{\rho}(x_i,t_n)),\widetilde{U}(x_i,t_n),\widetilde{\mathcal{T}}(x_i,t_n)\big)(v_j)
-\mathcal{M}_{\widetilde{\nu}}\big(\widetilde{\rho}_i^n,\widetilde{U}_i^n,\widetilde{\mathcal{T}}_i^n\big)(v_j).
\end{align*}
Therefore, applying the Taylor series, we expand this as
\begin{align}\label{turnbackto}
\begin{split}
\mathcal{M}_{\widetilde{\nu}}(\tilde{f}(x_i,t_n))(v_j)-\mathcal{M}_{\widetilde{\nu},j}(\widetilde{f}^n_i)
&=\big\{\widetilde{\rho}(x_i,t_n)-\widetilde{\rho}_i^n\big\}\int^1_0\frac{\partial\mathcal{M}_{\widetilde{\nu}}(\theta)}{\partial \rho}d\theta\cr
&+\big\{\widetilde{U}(x_i,t_n)-\widetilde{U}_i^n\big\}\int^1_0\frac{\partial\mathcal{M}_{\widetilde{\nu}}(\theta)}{\partial U}d\theta\cr
&+\big\{\widetilde{\mathcal{T}}_{\widetilde{\nu}}(x_i,t_n)-\widetilde{\mathcal{T}}_{\widetilde{\nu},i}^n\big\}\int^1_0\frac{\partial\mathcal{M}_{\widetilde{\nu}}(\theta)}{\partial\mathcal{T}_{\nu}}d\theta\cr
&\equiv I_1+I_2+I_3,
\end{split}
\end{align}
where
\begin{align*}
\frac{\partial\mathcal{M}_{\widetilde{\nu}}(\theta)}{\partial X}=\frac{\partial\mathcal{M}_{\widetilde{\nu}}}{\partial X}
\Big|_{(\widetilde{\rho}^n_{i,\theta}, \widetilde{U}^n_{i,\theta}, \widetilde{\mathcal{T}}^n_{i,\theta})}.
\end{align*}
For simplicity of notation, we define transitional macroscopic fields $\widetilde{\rho}^n_{\theta i}$,
$\widetilde{U}^n_{\theta i}$ and $\widetilde{\mathcal{T}}^n_{\theta i}$ by
\begin{align*}
(\widetilde{\rho}^n_{\theta i}, \widetilde{U}^n_{\theta i}, \widetilde{\mathcal{T}}^n_{\theta i})
=(1-\theta)\Big(\widetilde{\rho}(x_i,t_n), \widetilde{U}(x_i,t_n), \widetilde{\mathcal{T}}_{\nu}(x_i,t_n)\Big)
+\theta\Big(\widetilde{\rho}^{n}_{i}, \widetilde{U}^{n}_{i}, \widetilde{\mathcal{T}}^{n}_{\widetilde{\nu},i}\Big).
\end{align*}
Since this is a linear combination, we can derive the following estimates for the transitional macroscopic fields:
\begin{align}\label{trans Macro}
\begin{split}
& \|\widetilde{\rho}^n_{\theta i}\|_{L^{\infty}_x}, \|\widetilde{U}^n_{\theta i}\|_{L^{\infty}_x},
\|\widetilde{T}^n_{\theta i}\|_{L^{\infty}_x}
\leq C_{T^f},\cr
&\widetilde{\rho}^n_{\theta i}\geq C_{T^f}e^{-CT^f}, ~\widetilde{T}^n_{\theta i}\geq C_{T^f}e^{-C T^f},\cr
& k^{\top}\big\{\widetilde{\mathcal{T}}^n_{\theta i}\big\}k\geq C_{T^f}e^{-CT^f}|k|^2,\quad k\in \mathbb{R}^3.
\end{split}
\end{align}
from the lower and upper bounds of continuous and discrete macroscopic fields given in Theorem \ref{Existence_Theorem}, Proposition \ref{LB T Lemma} and Proposition
\ref{UB rho Lemma}. On the other hand,  Brum-Minkowski inequality implies that
\begin{align}\label{BM}
\begin{split}
\det \big\{\widetilde{\mathcal{T}}^n_{\theta i}\big\}&=\det\big\{(1-\theta)\widetilde{\mathcal{T}}_{\widetilde{\nu}}(x_i,t_n)+\theta \widetilde{\mathcal{T}}^n_{\widetilde{\nu},i}\big\}\cr
&\geq \det\big\{\widetilde{\mathcal{T}}_{\widetilde{\nu}}(x_i,t_n)\big\}^{1-\theta}\det\big\{ \widetilde{\mathcal{T}}^n_{\widetilde{\nu},i}\big\}^{\theta}\cr
&\geq \big\{C_{T^f}e^{-CT^f}\big\}^{1-\theta}\big\{C_{T^f}e^{-CT^f}\big\}^{\theta}\cr&=C_{T^f}e^{-CT^f}.
\end{split}
\end{align}
From these observations we have
\begin{align}\label{43}
\begin{split}
\mathcal{M}_{\widetilde{\nu}}(\theta)&=\frac{\widetilde{\rho}^n_{\theta i}}{\sqrt{\det(2\pi
\widetilde{\mathcal{T}}^n_{\theta i})}}
\exp\left(-\frac{1}{2}(v_j-\widetilde{U}^n_{\theta i})^{\top}\big\{\widetilde{\mathcal{T}}^n_{\theta i}\big\}^{-1}(v_j-\widetilde{U}^n_{\theta i})\right)\cr
&\leq C_{T^f}\exp\left(-C_{T^f}|v_j-\widetilde{U}^n_{\theta i}|^2\right)\cr
&\leq C_{T^f}\exp\left(C_{T^f}|\widetilde{U}^n_{\theta i}|^2\right) \exp\left(-C_{T^f}|v_j|^2\right)\cr
&\leq C_{T^f,1}e^{-C_{T^f,2}|v_j|^2}.
\end{split}
\end{align}
We now estimate each integral in $I_i$ $(i=1,2,3)$.
%%%%%%%%%%%%%%%%%%%%%%%%%%%%%%%%%%%%%%%%%%%%%%%%%%%%%%%%%%%%%%%%%%%%%%%%%%%%%%%%%%%%%%%%%%%%%%%%%%%%%%%%%%%%%%%%%%%%%%55
%
%
%
%
%%%%%%%%%%%%%%%%%%%%%%%%%%%%%%%%%%%%%%%%%%%%%%%%%%%%%%%%%%%%%%%%%%%%%%%%%%%%%%%%%%%%%%%%%%%%%%%%%%%%%%%%%%%%%%%%%%%%%%%%%%%%
The integral in $I_1$  comes straightforwardly from (\ref{43}):
\begin{align}\label{sub1}
\begin{split}
\left|\int^1_0\frac{\partial\mathcal{M}_{\widetilde{\nu}}(\theta)}{\partial\widetilde{\rho}^n_{\theta i}}ds\right|
= \left|\int^1_0\frac{1}{\widetilde{\rho}^n_{\theta i}}\mathcal{M}_{\widetilde{\nu}}(\theta)ds\right|
\leq C_{T^f,1}e^{-C_{T^f,2}|v_j|^2}.
%&=C_{\nu,T^f}e^{-C_{\nu,T^f}|v_j|^2}.
\end{split}
\end{align}
For the integral in $I_2$, we compute
\begin{align}\label{sub2}
\begin{split}
\left|\frac{\partial\mathcal{M}_{\widetilde{\nu}}(\theta)}{\partial \widetilde{U}^n_{\theta i}}\right|
&=\Big|(v_j-\widetilde{U}^n_{\theta i})^{\top}\{\widetilde{\mathcal{T}}^n_{\theta i}\}^{-1}
+\{\widetilde{\mathcal{T}}^n_{\theta i}\}^{-1}(v_j-\widetilde{U}^n_{\theta i})
\Big|\mathcal{M}_{\tilde{\nu}}(\theta)\cr
&\leq C_{T^f,1}\Big|(v_j-\widetilde{U}^n_{\theta i})^{\top}\{\widetilde{\mathcal{T}}^n_{\theta i}\}^{-1}+\{\widetilde{\mathcal{T}}^n_{\theta i}\}^{-1}(v_j-\widetilde{U}^n_{\theta i})
\Big|e^{-C_{T^f,2}|v_j|^2}.
\end{split}
\end{align}
In the last line, we used (\ref{43}).
For simplicity, set $X=v_j-\widetilde{U}^n_{\theta i}$. Then,
\begin{align*}
\big|X^{\top}\big\{\widetilde{\mathcal{T}}^n_{\theta i}\big\}^{-1}\big|&=\sup_{|Y|=1}\big|X^{\top}\{\widetilde{\mathcal{T}}^n_{\theta i}\}^{-1}Y\big|\cr
&=\frac{1}{2}\sup_{|Y|=1}\Big|(X+Y)^{\top}
\{\widetilde{\mathcal{T}}^n_{\theta i}\}^{-1}(X+Y)-X^{\top}
\{\widetilde{\mathcal{T}}^n_{\theta i}\}^{-1}X-Y^{\top}
\{\widetilde{\mathcal{T}}^n_{\theta i}\}^{-1}Y\Big|\cr
%&\qquad\leq\frac{1}{2}\sup_{|Y|=1}\Big\{(X+Y)^{\top}\{\widetilde{\mathcal{T}}^n_{\theta i}\}^{-1}(X+Y)+X^{\top}\{\widetilde{\mathcal{T}}^n_{\theta i}\}^{-1}X+Y^{\top}
%\{\widetilde{\mathcal{T}}^n_{\theta i}\}^{-1}Y\Big\}\cr
&\leq C\sup_{|Y|=1}\left(\frac{|X+Y|^2+|X|^2+1}{\widetilde{T}^n_{\theta i}}\right)\cr
&\leq C\left(\frac{1+|v_j-\widetilde{U}^n_{\theta i}|^2}{\widetilde{T}^n_{\theta i}}\right),
\end{align*}
which is, by (\ref{trans Macro}),  bounded by  $C_{T^f}(1+|v_j|)^2$. Similarly, we can derive
\begin{align}
|\{\widetilde{\mathcal{T}}^n_{\theta i}\}^{-1}(v_j-\widetilde{U}^n_{\theta i})|\leq C_{T^f}(1+|v_j|)^2,
\end{align}
so that from (\ref{sub2})
\begin{align}
\left|\int^1_0\frac{\partial\mathcal{M}_{\widetilde{\nu}}(\theta)}{\partial \widetilde{U}^n_{\theta i}}ds\right|\leq C_{T^f,1}(1+|v_j|)^2e^{-C_{T^f,2}|v_j|^2}.
\end{align}
We now turn to the integral in $I_3$. We first observe
\begin{align*}
\frac{\partial \mathcal{M}_{\widetilde{\nu}}(\theta)}{\partial \widetilde{\mathcal{T}}^{n\alpha\beta}_{\theta i}}
&=-\frac{1}{2}\left\{\frac{1}{\{\det\widetilde{\mathcal{T}}^n_{\theta i}\}^{1/2}}\frac{\partial\det\widetilde{\mathcal{T}}^n_{\theta}}{\partial\widetilde{\mathcal{T}}^{n\alpha\beta}_{\theta i}}
\right\}
\mathcal{M}_{\tilde{\nu}}(\theta)\cr
&+\frac{1}{2}\left\{(v_j-\widetilde{U}^n_{\theta i})^{\top}\{\widetilde{\mathcal{T}}^n_{\theta i}\}^{-1}\left(\frac{\partial\widetilde{\mathcal{T}}^n_{\theta i}}{\partial{\widetilde{\mathcal{T}}^{n\alpha\beta}_{\theta i}}}\right)
\{\widetilde{\mathcal{T}}^n_{\theta}\}^{-1}(v_j-\widetilde{U}^n_{\theta i})\right\}
\mathcal{M}_{\tilde{\nu}}(\theta).
\end{align*}
%\begin{align*}
%\left|-\frac{1}{|\det\mathcal{T}|}\frac{\partial\det\mathcal{T}}{\partial\mathcal{T}_{12}}\right|
%+\left|(v-U)\mathcal{T}^{-1}\left(\frac{\partial\mathcal{T}}{\partial{\mathcal{T}_{12}}}\right)\mathcal{T}^{-1}(v-U)\right|
%\end{align*}
To estimate the first term, we observe through an explicit computation that
\begin{align*}
\det\widetilde{\mathcal{T}}^n_{\theta i}
=\widetilde{\mathcal{T}}^{n11}_{\theta i}\widetilde{\mathcal{T}}^{n22}_{\theta i}\widetilde{\mathcal{T}}^{33}_{\theta i}
-2\widetilde{\mathcal{T}}^{n12}_{\theta i}\widetilde{\mathcal{T}}^{n23}_{\theta i}\widetilde{\mathcal{T}}^{n31}_{\theta i}-\widetilde{\mathcal{T}}^{n11}_{\theta i}\left\{\widetilde{\mathcal{T}}^{n23}_{\theta i}\right\}^2
-\widetilde{\mathcal{T}}^{n22}_{\theta i}\left\{\widetilde{\mathcal{T}}^{n31}_{\theta i}\right\}^2
-\widetilde{\mathcal{T}}^{n33}_{\theta i}\left\{\widetilde{\mathcal{T}}^{n12}_{\theta i}\right\}^2.
\end{align*}
Therefore, $\frac{\partial\det\widetilde{\mathcal{T}}_{\theta}}{\partial\widetilde{\mathcal{T}}^{n\alpha\beta}_{\theta i}}$ is a homogeneous polynomial of degree 2 having $\widetilde{\mathcal{T}}^{n \alpha\beta}_{\theta i}$ $(1\leq \alpha,\beta\leq 3)$ as variables. Now, recalling
$\widetilde{\mathcal{T}}_{\theta i}^{n \alpha\beta}\leq Ce^{T^f}$ from (\ref{trans Macro}), together with the lower bound of
$\det\widetilde{\mathcal{T}}^n_{\theta i}$ in (\ref{BM}), we conclude
\begin{align*}
\left|\frac{1}{\{\det\widetilde{\mathcal{T}}^n_{\theta i}\}^{1/2}}\frac{\partial\det\widetilde{\mathcal{T}}^n_{\theta i}}{\partial\widetilde{\mathcal{T}}_{\theta i}^{n \alpha\beta}}\right|\leq C_{T^f}.
\end{align*}
On the other hand, since all the entries of $\frac{\partial\mathcal{T}_{\theta}}{\partial{\mathcal{T}_{\theta}^{\alpha\beta}}}$ are 0 except for the $\alpha,\beta$ entry, we see that
\begin{align*}
\left|(v-\widetilde{U}^n_{\theta i})^{\top}\!\big\{\widetilde{\mathcal{T}}^{n}_{\theta i}\big\}^{-1}\!\!\left(\frac{\partial\widetilde{\mathcal{T}}^n_{\theta i}}{\partial{\widetilde{\mathcal{T}}_{\theta i}^{n \alpha\beta}}}\right)\!\!\big\{\widetilde{\mathcal{T}}^{n}_{\theta i}\big\}^{-1}\!
(v-\widetilde{U}^n_{\theta i})\!\right|
\leq \!\left|(v-\widetilde{U}^n_{\theta \alpha})^{\top}\!\{\widetilde{\mathcal{T}}^{n}_{\theta \alpha}\}^{-1}\!\right|
\!\left|\!\{\widetilde{\mathcal{T}}^{n}_{\theta \beta}\}^{-1}(v-\widetilde{U}^n_{\theta \beta})\!\right|,
\end{align*}
which can be shown to be bounded by $C_{T^f,\nu}(1+|v_j|)^2$
using the same argument as in the previous case. Combining these two estimates, we bound the integral in $I_3$ as
\begin{align*}
\int^1_0\Big|\frac{\partial \mathcal{M}_{\widetilde{\nu}}(\theta)}{\partial \widetilde{\mathcal{T}}_{\theta}}\Big|d\theta\leq C_{T^f}e^{-C_{T^f}|v_j|^2}.
\end{align*}
We now insert all the above estimates into (\ref{turnbackto}) to get
\begin{align*}
&|{\mathcal M}^n_{\widetilde{\nu},j}(\tilde{f}^n_i)-{\mathcal M}_{\widetilde{\nu}}(\tilde{f}(x_i,t_n))(v_j)|\\
&\hspace{2cm}\leq C_{T^f}\Big\{|\widetilde{\rho}^n_{i}-\widetilde{\rho}(x_i,t_n)|
+|\widetilde{U}^n_{i}-\widetilde{U}(x_i,t_n)|+|\widetilde{\mathcal{T}}^n_{\widetilde{\nu},i}-\widetilde{\mathcal{T}}_{\widetilde{\nu}}(x_i,t_n)|\Big\}e^{-C_{T^f}|v_j|^2}.
\end{align*}
Then, Lemma \ref{rho-rho}  gives the desired result.
\end{proof}
%%%%%%%%%%%%%%%%%%%%%%%%%%%%%%%%%%%%%%%%%%%%%%%%%%%%%%%%%%%%%%%%%%%%%%%%%%%%%%%%%%%%%%%%%%%%%%%%%%%%%%%%%%%%%%%%%%%%%%%%%%%%%%%%%%%%%%
%
%
%
%
%%%%%%%%%%%%%%%%%%%%%%%%%%%%%%%%%%%%%%%%%%%%%%%%%%%%%%%%%%%%%%%%%%%%%%%%%%%%%%%%%%%%%%%%%%%%%%%%%%%%%%%%%%%%%%%%%%%%%%%%%%%%%%%%%%%%%
\section{ Proof of Theorem \ref{maintheorem}}
We are now ready to prove our main theorem.
Subtracting (\ref{main scheme}) from (\ref{Consist3}) and taking $L^{\infty}_q$ norms, we get
\begin{align*}%\label{Reform-Consist}
\|f(t_{n+1})-f^{n+1}\|_{L^{\infty}_q}&\leq\frac{\kappa}{\kappa+A_{\nu}\triangle t}
\|\widetilde{f}(t_n)-\widetilde{f}^{n}\|_{L^{\infty}_q}\cr
&+\frac{A_{\nu}\triangle t}{\kappa+A_{\nu}\triangle t}\|{\mathcal M}_{\widetilde{\nu}}(\widetilde{f}(t_n))-{\mathcal M}_{\widetilde{\nu}}(\widetilde{f}^n)\|_{L^{\infty}_q}\nonumber\cr
&+\frac{A_{\nu}}{\kappa+A_{\nu}\triangle t}\|R_1\|_{L^{\infty}_q}
+\frac{A_{\nu}}{\kappa+A_{\nu}\triangle t}\|R_2\|_{L^{\infty}_q}.
\end{align*}
We then recall the estimates in Lemma \ref{Remainder1}, Lemma \ref{Additional_Lemma1} and Proposition \ref{ContinuityofMaxwellianDiscrete}:
%Lemma \ref{Additional_Lemma2}:
\begin{align*}
\|\widetilde{f}(t_n)-\widetilde{f}^n\|_{L^{\infty}_q}
&\leq\|f(t_n)-f^n\|_{L^{\infty}_q}+C_{T^f}(\triangle x)^2,\nonumber\cr
\|{\mathcal M}_{\widetilde{\nu}}(\widetilde{f}(t_n))-{\mathcal M}_{\widetilde{\nu}}(\widetilde{f}^n)\|_{L^{\infty}_q}
%&\leqC\|\widetilde{f}-\widetilde{f}^n\|_{L^{\infty}_q}+C_{T^f}(\triangle x+\triangle v)\cr
&\leq C_{T^f}\|f(t_n)-f^n\|_{L^{\infty}_q}+C_{T^f}\left\{(\triangle x)^2+\triangle v+\triangle v\triangle t\right\},
\nonumber\cr
\|R_1\|_{L^{\infty}_q}+\|R_2\|_{L^{\infty}_q}&\leq C_{T^f}\|f_0\|_{W^{1,\infty}_q}(\triangle t)^2,\nonumber
\end{align*}
to derive the following recurrence inequality for the numerical error:
\begin{align}\label{NE}
\|f(t_{n+1})-f^{n+1}\|_{L^{\infty}_q}&=\Big(1+\frac{C_{T^f}\triangle t}{\kappa+A_{\nu}\triangle t}\Big)\|f(t_n) - f^{n}\|_{L^{\infty}_q}
+C_{T^f}P(\triangle x,\triangle v, \triangle t),
\end{align}
where
\[
P(\triangle x,\triangle v, \triangle t)=\frac{(\triangle x)^2+\triangle t\{(\triangle x)^2+\triangle v+\triangle v\triangle t\}+(\triangle t)^2}{\kappa+A_{\nu}\triangle t}.
\]
Put $\Gamma=\frac{C_{T^f}\triangle t}{\kappa+A_{\nu}\triangle t}$ for simplicity of notation and
iterate (\ref{NE}) until  the initial step is reached, to derive:
\begin{align}\label{a.f.estimate}
\|f(N_t\triangle t)-f^{N_t}\|_{L^{\infty}_q}\leq(1+\Gamma)^{N_t}\|\,f_0 - f^{0}\,\|_{L^{\infty}_q}
+C_{\nu}\sum_{i=1}^{N_t}(1+\Gamma)^{i-1} P.
\end{align}
We first note from the definition of $f^0_{i,j}=f_0(x_i,v_j)$ that
\[
\|\,f_0 - f^{0}\,\|_{L^{\infty}_q}=\sup_{i,j}|f_0(x_i,v_j)-f^0_{i,j}|(1+|v_j|)^q=0.
\]
Besides, we use $(1+x)^n\leq e^{nx}$ to see that
\begin{align*}
(1+\Gamma)^{N_t}\leq e^{N_t\Gamma }\leq e^{\frac{C_{T^f}N_t\triangle t}{\kappa+A_{\nu}\triangle t}}
= e^{\frac{C_{T^f}T^f}{\kappa+\triangle t}},
\end{align*}
so that
\begin{align*}
\sum_{i=1}^{N_t}(1+\Gamma)^{i-1}=\frac{(1+\Gamma)^{N_t}-1}{(1+\Gamma)-1}
%&\leq\frac{\Gamma_T}{\Gamma}\\
\leq C_{T^f}\left(\frac{\kappa+A_{\nu}\triangle t}{\triangle t}\right)\left\{e^{\frac{C_{T^f}T^f}{\kappa+A_{\nu}\triangle t}}-1\right\},
\end{align*}
which gives
\begin{align*}
\sum_{i=1}^{N_t}(1+\Gamma)^{i-1}P\leq C_{T^f}\left(e^{\frac{C_{T^f}T^f}{\kappa+A_{\nu}\triangle t}}-1\right)\left\{\triangle x+\triangle v+\triangle t+\frac{\triangle x}{\triangle t}\right\}.
\end{align*}
Substituting these estimates into (\ref{a.f.estimate}), we find
\begin{align*}
\|f(T^f)-f^{N_t}\|_{L^{\infty}_q}&\leq C_{T^f}\left(e^{\frac{C_{T^f}T^f}{\kappa+\triangle t}}-1\right)
\left\{(\triangle x)^2+\triangle v+\triangle t+\frac{(\triangle x)^2}{\triangle t}\right\}.
\end{align*}
%\begin{align*}
%\|f(T^f) - f^{N_t}\|_{L^{\infty}_q}\leq C(T^f,q,f_0,\kappa)
%\Big(\triangle x+\triangle v+\triangle t+\frac{\triangle x}{\triangle t}\Big).
%\end{align*}
This completes the proof.
\section*{Acknowledgement}
The work of S.-B. Yun was supported by Samsung Science and Technology Foundation under Project Number SSTF-BA1801-02.\newline

%This research was supported by Basic Science Research Program through the National Research Foundation of Korea(NRF) funded by the Ministry of Education(NRF-2016R1D1A1B03935955)

\end{document}